%% file: main.tex
\documentclass{article}
\usepackage{graphicx} 
\usepackage[margin=1in]{geometry}
\usepackage{amsmath, amsfonts, amssymb, amsthm}
\usepackage{graphicx}
\usepackage{hyperref}
\usepackage{color}
\usepackage{dsfont}
\usepackage{comment}
\usepackage{esint}
\usepackage[utf8]{inputenc}
\usepackage{amsmath}
\usepackage{gensymb}
\usepackage{mathrsfs}
\usepackage{bbm}
\usepackage{tikz-cd}
\usepackage{graphicx}
\usepackage{subfig}
\usepackage{authblk}

\newtheorem{theorem}{Theorem}[section]
\newtheorem*{theorem*}{Theorem}
\newtheorem{corollary}{Corollary}[theorem]

\newtheorem{lemma}[theorem]{Lemma}
\newtheorem{remark}[theorem]{Remark}
\newtheorem{proposition}{Proposition}
\newtheorem{question}{Question}
\newtheorem{example}{Example}

\newtheorem{definition}[theorem]{Definition}

\newcommand{\nl}{\newline}
\newcommand{\st}{\text{ s.t. }}

\newcommand{\Z}{\mathbb{Z}}
\newcommand{\R}{\mathbb{R}}

\newcommand{\N}{\mathbb{N}}
\newcommand{\Q}{\mathbb{Q}}
\newcommand{\C}{\mathbb{C}}

\newcommand{\tr}{\text{tr}}

\newcommand{\Ric}{\text{Ric}}

\newcommand{\eps}{\epsilon}
\newcommand{\n}{\nabla}
\newcommand{\p}{\partial}
\newcommand{\e}{\epsilon}
\newcommand{\cF}{\mathcal{F}}
\newcommand{\Ind}{\mathrm{Ind}}

\renewcommand{\P}{\mathbb{P}}

\renewcommand{\div}{\text{div}}

\newcommand{\cohom}{\mathrm{Cohom}}
\newcommand{\an}{\textnormal{An}^G}
\newcommand{\ov}[1]{\overline{#1}}
\newcommand{\ti}[1]{\tilde{#1}}

\title{Equivariant Allen Cahn Solutions and the Existence of Cohomogeneity $2$ Minimal Surfaces}

\author[1]{Rayssa Caju}
\author[2]{Pedro Gaspar}
\author[3]{Jared Marx-Kuo}
\affil[1]{\normalsize Departamento de Ingenier\'ia Matem\'atica and Centro de Modelamiento Matem\'atico (CNRS IRL 2807). Universidad de Chile, Beauchef 851, Santiago, Chile}
\affil[2]{\normalsize Facultad de Matem\'aticas, Pontificia Universidad Cat\'olica de Chile, Avenida Vicuña Mackenna 4860, Santiago, Chile}
\affil[3]{\normalsize Department of Mathematics, Rice University, Houston, TX 77005, USA}
\affil[]{rcaju@dim.uchile.cl \ ${}^2$ pedro.gaspar@uc.cl \ ${}^3$ jm307@rice.edu}

\date{}

\begin{document}

\maketitle

\begin{abstract}
We develop a regularity theory for equivariant Allen--Cahn solutions on closed Riemannian manifolds with a Lie group acting isometrically. When the cohomogeneity of the action is between $3$ and $7$, we show that a sequence of equivariant Allen--Cahn solutions with uniformly bounded energy and equivariant index converge to embedded minimal hypersurfaces with optimal regularity, meaning that the singular set is at least codimension $7$ and lies in the union of all non-principal orbits. When the cohomogeneity is $2$ and the action has no exceptional orbits, we show the same result but the minimal hypersurfaces may be immersed. As a result, any closed Riemmanian manifold with cohomogeneity $2$ Lie group action and no exceptional orbits admits a minimal hypersurface with optimal regularity. A key tool is the regularity theory of Chodosh--Mantoulidis \cite{ChodoshMantoulidisWidths}, building on the work of Wang--Wei \cite{wang2019finite}. However, we adapt their arguments to a modified Allen--Cahn equation with a drift Laplacian. We also show that appropriate index bounds hold for the limiting minimal hypersurface when it is smooth. \newline 
\indent We also extend the variational constructions of solutions of the Allen--Cahn equation of \cite{GuaracoMinmax} and \cite{GasparGuaracoClosed} by defining an equivariant mountain pass invariant, as well as the equivariant Allen--Cahn $p$-widths. This builds on the work of Gromov \cite{gromov2006dimension} and is the Allen--Cahn parallel to Wang's equivariant volume spectrum \cite{WangDensity} in the Almgren-Pitts setting \cite{marques2023morse}. We show that in the limit as $\epsilon$ tends to $0$, the equivariant Allen--Cahn $p$-widths converge to the equivariant $p$-widths, as defined by Wang \cite{WangDensity}. 

\end{abstract}

\tableofcontents

\section{Introduction} \label{section.introduction}
\indent In this work, we are interested in the construction of minimal hypersurfaces with symmetries given by a Lie group $G$, acting isometrically on a Riemannian manifold $(M^{n+1}, g)$. There is a lengthy history to the construction of such minimal surfaces - the story begins with the work of Almgren--Pitts \cite{Almgren} \cite{pitts2014existence} who laid much of the foundational theory for the existence and regularity of minimal surfaces via min--max. This program was revived by Marques--Neves \cite{MarquesNevesPositive, MarquesNevesWillmore, marques2016morse} and experienced an explosion of work by various authors in constructing minimal surfaces (as well as the related surfaces with constant/prescribed mean curvature). We mention the following references for a by no means complete list of relevant works \cite{LiokumovichMarquesNeves, irie2018density,marques2019equidistribution,zhou2020multiplicity, ambrozio2021riemannian, song2023existence, marques2023morse, li2023existence, GuthLiokumovich}. \nl 
\indent In the setting of manifolds with isometries given by a Lie group, Pitts--Rubenstein \cite{pitts1988equivariant} set forth a program to construct $G$-equivariant minimal surfaces, i.e. minimal surfaces, $\Sigma$, such that $G \cdot \Sigma = \Sigma$. This program was further expanded upon by Ketover \cite{ketover2016equivariant}, in his seminal work on the construction of equivariant minimal surfaces in closed three manifolds. Liu \cite{LiuCVPDE} furthered the program to manifolds of higher dimension and codimension, showing the existence of a single $G$-equivariant minimal surface. Wang also developed important constructions in the min-max theory of minimal surfaces to min-max theory of $G$-equivariant minimal surfaces \cite{wang2022min, WangJGA, WangIndex, WangDensity,wang2026equivariant,li2026infinite}. We also mention the recent works of Ko who establishes a min-max theorem for isotopy minimization in the G-equivariant setting \cite{ko2025regularity} for cohomogeneity $3$, the work of Ketover \cite{KetoverFBMS}, Buzano--Nguyen--Schulz \cite{BNS} and Carlotto--Franz--Schulz \cite{CFS} on the construction of equivariant free boundary minimal surfaces and self-shrinkers for the mean curvature flow, as well as the work of Wang--Wang--Zhou who also address isotopy minimization in the form of Simon--Smith min-max in the $G$-equivariant setting to show a Bernstein type theorem in round spheres \cite{wang2026equivariant}. \nl 
\indent Parallel to the theory of minimal surfaces is the theory of solutions to the Allen--Cahn equation on closed manifolds. Such solutions $u: M \to \R$ satisfy the equation 
\begin{equation} \label{ACEquation}
\eps^2 \Delta_g u = W'(u)
\end{equation}
for $W: \R \to \R$ a ``double well potential" and $\eps$, a small parameter tending to $0$. Such solutions are critical points of the energy functional 
\[
E_{\eps}(u) = \int_M \eps \frac{|\n u|^2}{2} + \frac{W(u)}{\eps}
\]
see \S \ref{sec.regularity} for more details on the restrictions of $W$. \nl 
\indent Solutions of \eqref{ACEquation} converge in various sense to minimal surfaces, $\Sigma$, in the limit that $\eps \to 0$ \cite{Modica,Sternberg}. In the context of varifolds, this was first established by Padilla--Tonegawa \cite{PadillaTonegawa} and Hutchinson--Tonegawa \cite{HutchinsonTonegawa}, expanded upon by Guaraco \cite{GuaracoMinmax} and Gaspar--Guaraco \cite{GasparGuaracoClosed, GasparGuaracoWeyl}, using the regularity theory of Wickramasekera \cite{Wickramasekera} and the work of Tonegawa--Wickramasekera \cite{TWStable}. We present a by no means exhaustive list of some developments in the geometric aspects of the theory of the Allen--Cahn including \cite{HutchinsonTonegawa, TWStable, wang2017some, smith2016bifurcation, GasparGuaracoClosed, GasparGuaracoWeyl, ChodoshMantoulidisMinimal, HiesmayrLow, gaspar2020second, mantoulidis2021allen, dey2022comparison, ChodoshMantoulidisWidths, marx2025isospectral, marx2025p, MarxKuoSarnataroStryker, FloritSimonSerra, marx2026p}. \nl 
\indent While several theorems in the theory of minimal surfaces have been \textit{reproved} with the Allen--Cahn equation, we note the recent improvements in min-max for geodesics on closed surfaces. Historically, the process of Almgren--Pitts min-max for $1$-dimensional stationary varifolds on surfaces would produce geodesic networks. These are unions of geodesic segments (potentially with multiplicity), which can be singular at junctions \cite{Pitts} (see also \cite[Remark 1.1]{marques2016morse}). In recent work of Chodosh--Mantoulidis \cite{ChodoshMantoulidisWidths}, the authors use the particular Sine--Gordon regularization to show that Allen--Cahn min-max yields unions of closed geodesics on surfaces, building off work of Liu--Wei who constructed an integrable family of Allen--Cahn solutions on $\R^2$ \cite{liu2021classification}. Since this development, the regularity afforded by the Allen--Cahn theory on surfaces has been used to prove several results related to Gromov's $p$-widths, $\{\omega_p\}$. See \cite{MarxKuoSarnataroStryker, chodosh2025p, ChenGasparBerger, marx2025isospectral, marx2025p, marx2026p} for some recent developments. \nl 
\indent Given the parallels between the Almgren--Pitts min-max theory and the Allen--Cahn theory, it is natural to develop the parallels between $G$-equivariant min-max theory and the Allen--Cahn theory of $G$-equivariant solutions. In the results below, we show that much of the theory remains similar for actions with $\text{Cohom}(G) \geq 3$, though we overcome some novel difficulties due to the lack of robustness of Allen--Cahn solutions with finite \textit{$G$-equivariant index}. \nl 
\indent We provide a novel contribution in the setting of $\text{Cohom}(G) = 2$, by using equivariant Allen--Cahn min-max to show the existence of an immersed minimal hypersurface. For groups with $\text{Cohom}(G) = 2$, the quotient manifold, $M/ G$, is a smooth surface away from a small set. Using Allen--Cahn min-max, we can construct a set of $G$-equivariant solutions which descend to solutions of an Allen--Cahn equation \textit{with drift} on $M/G$ 
\[
\eps \Delta_{g} u + \eps \langle \n \ln(\mathcal{V}), \n u \rangle - \frac{W'(u)}{\eps} = 0.
\] 
\noindent We refer the reader to the discussion in \S \ref{section.cohom.two} for further details. \nl 
\indent Despite the analytic differences, the drift term disappears under blow-up and we are able to recover the regularity result of Chodosh--Mantoulidis \cite[Proposition 3.8]{ChodoshMantoulidisWidths} via an adaptation of the regularity estimates of Wang--Wei \cite{wang2019finite} (see also Mantoulidis \cite[Theorem 4.13]{mantoulidis2021allen}). Thus, in the limit that $\eps \to 0$, we obtain a union of closed geodesics for a weighted metric on $M/G$. These geodesics then lift to a union of immersed minimal hypersurfaces with optimal regularity on $M$. \nl 
\indent When $M/G$ has no singular points, this argument is robust, so much of our work handles singularities arising in the non-principal orbits. We remark that previous literature seems to focus on specific ambient manifolds, and our result seems to be the first general construction of Cohomogeneity $2$ equivariant minimal surfaces when the action has no exceptional orbits. 

\subsection{Statement of Results}
We will consider $(M^{n+1}, g)$ a closed Riemannian manifold and $G$, a Lie group acting via isometries on $M$, and the cohomogeneity of the action, denoted by $\cohom(G)$, satisfies $2 \leq \cohom(G) \leq 7$. We let $M^{reg} \subseteq M$ denote the dense, open subset consisting of the union of all principal orbits of the action (see \S \ref{section.background} and \S \ref{section.appendix} for formal definitions). Our first result parallels the work of Guaraco \cite{GuaracoMinmax} and is a general regularity theory showing the existence of a $G$-invariant minimal surface via solutions to the Allen--Cahn equation.
\begin{theorem} \label{thm.AC.regularity}
Let $\{u_{\eps_k}\} \subseteq W^{1,2}_G(M)$ be a sequence of Allen--Cahn solutions with $E_{\eps}(u_{\eps_k}) \leq C$ and $\text{Ind}_G(u_{\eps_k}) \leq N$ for fixed $C, N > 0$. Then up to subsequence, the solutions converge in a varifold sense to a minimal $G$-invariant hypersurface (possibly with integer multiplicities) which is smooth away from a closed singular set, $\Lambda$. If $\cohom(G) \geq 3$, the hypersurface is embedded, if $\cohom(G) = 2$ and the action has no exceptional orbits, it is smoothly immersed away from the singular set. In both cases, $\mathcal{H}^{n-7}(\Lambda) < \infty$, $\Lambda \subset M \backslash M^{reg}$.
\end{theorem}
\noindent Here, $\text{Ind}_G$ denotes the $G$-equivariant index (see Section \ref{section.background} for definition). In some sense, the regularity theory of Theorem \ref{thm.AC.regularity} is sharp, both in the dimension of the singular set and the allowance for immersed equivariant minimal surfaces when the cohomogeneity is $2$. For the former, we note that the Simon's Cone in $\R^8$ can be viewed as a cohomogeneity $2$ minimal hypersurface with respect to the natural $S^3 \times S^3$ action. Moreover, it has a singular set of dimension exactly $0 = n - 7$ (see details in \S \ref{section.examples}). For the latter, we remark that for any smooth compact Lie group $G$, we can consider $M = \Sigma \times G$ with the product metric, where $\Sigma$ admits a closed immersed geodesic. Then $G$ acts naturally on just the Lie group component of $M$ and the quotient is $\Sigma$. \nl 
\indent While Theorem \ref{thm.AC.regularity} is interesting on its own, we can apply it to a $G$-equivariant mountain pass construction of Allen--Cahn solutions to derive the following existence result:
\begin{corollary} \label{cor.AC.mountainpass}
For $(M^{n+1}, g)$ and $G$ as above, there exists a $G$-equivariant minimal surface which is smooth away from a closed singular set, $\Lambda$, satisfying $\mathcal{H}^{n-7}(\Lambda) < \infty$, $\Lambda \subset M \backslash M^{reg}$. In particular, a manifold with a $\cohom(G) = 2$ action and no exceptional orbits contains an immersed $G$-equivariant minimal hypersurface which is smooth away from said singular set.
\end{corollary}
\noindent To the authors knowledge, Corollary \ref{cor.AC.mountainpass} is novel for this level of generality of equivariant group actions. Our work uses the phase transition methods of \cite{ChodoshMantoulidisWidths}, and the $G$-equivariant minimal surface in Corollary \ref{cor.AC.mountainpass} is obtained as the limit interface of solutions of the Allen-Cahn equation whose energy equal the Allen--Cahn mountain pass width $\omega_{AC,\e}^G \in \R$ among $G$-invariant functions. In the limit, their energy converges to the equivariant Almgren--Pitts one parameter width (see Section \ref{mountain pass} for more details). 
\begin{remark} 
If one restricts to the case of smooth quotients, then one can show the existence of an immersed, smooth $G$-equivariant minimal surfaces on $M$ by using the existence of closed geodesics on $(M/G, \overline{g})$ (where $\overline{g}$ is the Hsiang--Lawson conformal metric, see \cite{HsiangLawson}). This follows via the Lyusternik--Fet Theorem \cite{fet1952variational}, though we note that from the min-max perspective, one could apply the work of Chodosh--Mantoulidis \cite{ChodoshMantoulidisWidths} directly on $(M/G, \overline{g})$ to find Allen--Cahn solutions which limit to unions of smooth geodesics. However, as we will see in section \ref{section.cohom.two}, such Allen--Cahn solutions on $(M/G, \overline{g})$ in general do not lift to $G$-equivariant Allen--Cahn solutions on $M$. Moreover, many quotients are not smooth, and the full strength of Corollary \ref{cor.AC.mountainpass} allows for singular orbits on $M/G$ by working with $G$-equivariant Allen--Cahn solutions on $M$.
\end{remark} 
\indent The \emph{Allen--Cahn $p$-widths} are a sequence of critical values for the Allen--Cahn energy in a compact Riemannian manifold introduced in \cite{GasparGuaracoClosed} in parallel to the volume spectrum of a Riemannian manifold \cite{MarquesNevesPositive}. These geometric invariants play a central role in the phase-transitions strategy to the existence of infinitely many minimal hypersurfaces, see e.g. \cite{GasparGuaracoWeyl, ChodoshMantoulidisMinimal} and also \cite{CFSS,FloritSimonWeyl} for a related nonlocal approach.

In the equivariant setting, we can also apply Theorem \ref{thm.AC.regularity} to the cohomological classes of sets of functions associated to the \emph{$G$-invariant Allen--Cahn $p$-widths}. These are min-max critical values for the Allen-Cahn energy in spaces of $G$-invariant functions and can be regarded as the $G$-equivariant versions of the widths defined in \cite{GasparGuaracoClosed}, providing a phase-transitions counterpart of Wang's \cite{WangDensity} equivariant volume spectrum. More precisely,
\begin{definition}
The Allen--Cahn $G$-invariant $p$-widths are defined as 
\[
\omega_{p, \eps}^G(M,g) = \inf_{\Phi \in \mathcal{F}_p^G} \sup_{x \in \text{Dom}(\Phi)} E_{\eps}(\Phi(x)),
\]
where $\mathcal{F}_p^G$ is the set of all images $\Phi(X) \subset W^{1,2}(M)$, for continuous $\Z_2$-equivariant maps $\Phi\colon X\to W^{1,2}(M)$ defined on a compact $\Z_2$-space with with $\Z_2$-cohomological index $\geq p$ and such that $\Phi(x)$ is a $G$-invariant function for all $x \in X$ (see Section \ref{phase transitions spectrum} for the detailed definition).
\end{definition}

\noindent Adding onto the mountain pass construction, we can use the Allen--Cahn $G$-invariant $p$-widths to show the existence of $G$-invariant minimal surfaces:
\begin{theorem} \label{thm.p.width.reg}
For each $p$, there exist a union of $G$-invariant minimal surfaces $\{\Sigma_i^p\}_{i = 1}^{N_p}$ and integer multiplicities $\{a_i^p\}$ such that 
\[
\omega_{p,\e}^G(M,g) = \sum_{i = 1}^{N_p} a_i^p \text{Area}(\Sigma_i^p).
\]
When $\text{Cohom}(G) = 2$ and the action has no exceptional orbits, these surfaces may intersect or be immersed. When $\cohom(G) = 3$, they are embedded and disjoint. In both cases, the surfaces are smooth up to a singular set of dimension at most $n-7$.
\end{theorem}
\noindent Adapting an argument of the second author \cite{gaspar2020second} we can also show weak index bounds of these minimal hypersurfaces. 
\begin{theorem} \label{thm.p.index}
Suppose that $\{u_{\eps_i}\}$ is a sequence of $G$-equivariant solutions to \eqref{ACEquation} with $G$-equivariant index bounded by $p$. Let $V = \sum_{i = 1}^{N_p} a_i^p \Sigma_i^p$ be the limit varifold of the $\{u_{\eps_i}\}$ where $\{\Sigma_i^p\}$ are immersed minimal hypersurfaces. Then
\[
\sum_{i = 1}^{N_p} \text{Ind}_G(\text{Reg}(\Sigma_i^p)) \leq p
\]
\end{theorem}
\noindent Here, $\text{Reg}(\Sigma)$ denotes the regular part of the hypersurface. Following the strategy by Dey \cite{dey2022comparison}, we also show that these equivariant Allen-Cahn widths coincide, in the limit as $\eps \to 0$, to the $G$-equivariant $p$-widths, $\{\omega_p^G\}$, as defined in Wang \cite{wang2022min}

\begin{theorem}
We have
\[
\frac{1}{2\sigma} \lim_{\eps \to 0} \omega_{p,\eps}^G(M,g) = \omega_p^G(M,g)
\]
\end{theorem}

\noindent Furthermore, we recover a multiplicity one result for the first Allen-Cahn width, in the positive Ricci setting:
\begin{theorem} \label{thm.mult.one.ricci}
When $(M^{n+1}, g)$ satisfies $\Ric_g > 0$ and $\cohom(G) \geq 3$, then 
\[
\lim_{\e \to 0}\omega_{AC,\e}^G(M,g) = \mathrm{Area}(\Sigma_1)
\]
for $\Sigma_1$ an embedded $G$-equivariant minimal surface. When $\cohom(G) = 2$ and $M/G$ is a smooth quotient with $\Ric_{\tilde g_{M/G}} > 0$, the same result holds.
\end{theorem}
\noindent We refer the reader to Theorem \ref{weyl} for the definition of $\tilde g_{M/G}$. Theorem \ref{thm.mult.one.ricci} is analogous to work of Bellettini \cite{bellettini2024multiplicity}, and we are able to adopt his argument to the $G$-invariant setting for $\cohom(G) \geq 3$. The same argument however fails in $\cohom(G) = 2$ (see remark \ref{codim.two.remark}), but we can recover multiplicity one and embeddedness when we have the further restriction of $M/G$ being a smooth quotient with positive Ricci curvature.
\subsection{An Explanation of no exceptional orbits in cohomogeneity two}
We remark that the requirement of \textit{no exceptional orbits} in cohomogeneity $2$ arises as follows: intuitively, to conclude the regularity theory of an immersed minimal surface in cohomogeneity two, we use the Sine--Gordon potential in the Allen--Cahn regularization of minimal surfaces. Applying either a mountain-pass or a $p$-parameter min-max construction will yield $G$-equivariant Allen--Cahn solutions, $\{u_{\eps}\}$, with $G$-equivariant index bounded by $p$. Following a classic covering argument, in the limit as $\eps \to 0$, the ``index accumulates" in at most $p$ orbits, $\{\mathcal{O}_i\}$. On the union of principal orbits (which is open and dense), any solution on $(M, g)$ projects to a solution of the Allen--Cahn equation \textit{with a drift term} on $(M/G, g_{M/G})$ (see equation \eqref{eqn.weak.one} and Theorem \ref{thm.drift.stability}). This complicates the analysis significantly, and we prove an adaptation of Wang--Wei's stability estimates \cite[Theorem 3.6]{wang2019finite} (see also \cite[Theorem 4.13]{mantoulidis2021allen} for the Riemannian adaptation) for the Allen--Cahn equation with a drift Laplacian. Nonetheless, if $\mathcal{O}_i$ is a principal orbit, then we can use the regularity theory coming from the quotient manifold, which is locally smooth near the quotient of a principal orbit. Applying the adapted stability estimates of Theorem \ref{thm.drift.stability} and the argument of Chodosh--Mantoulidis \cite[Theorem 3.1]{ChodoshMantoulidisWidths}, we conclude.  \nl 
\indent If $\mathcal{O}_i$ lies in the union of singular orbits, then these are necessarily at least codimension $3$ (and hence at least codimension $2$ when viewed as a subset of the minimal surface), by which the regularity theory and $\alpha$-structural hypothesis of Tonegawa--Wickramasekera \cite{TWStable} and Wickramasekera \cite{Wickramasekera} allow us to ``smooth" over $\mathcal{O}_i$ (see a similar argument by Hiesmayr \cite[\S 4.2]{hiesmayr2018spectrum}). Thus the core difficulty lies when the index accumulates in the exceptional orbits, which are codimension $2$ in the ambient manifold and hence codimension $1$ in the minimal surface itself. At this point, the machinery of Chodosh--Mantoulidis \cite{ChodoshMantoulidisWidths} may not apply as the quotient metric may be singular at points corresponding to quotients of exceptional orbits. The classic example of this is the $S^1$ action on $S^3$ as described in the appendix \S \ref{section.appendix} - here one is concerned that the index accumulates at either of the two conical points in the orbifold quotient of $S^2$, colloquially thought of as the tips of the ``American Football."

\subsection{Paper outline}
This paper is organized as follows:
\begin{enumerate}
\item In section \S \ref{section.background}, we introduce basic min-max objects with their $G$-invariant counterparts, the Allen--Cahn theory, and a short lemma on approximating Caccioppoli sets.

\item In \S \ref{section.min.max}, we construct $G$-equivariant solutions to the Allen--Cahn equation via mountain pass and higher $p$-parameter min-max methods. The latter leads to a definition of $G$-equivariant Allen--Cahn widths, $\{\omega_{p,\eps}^G\}$, and in \S \ref{section.comparison}, we demonstrate that these widths converge to the Almgren--Pitts $G$-equivariant $p$-widths, $\{\omega_p^G\}$, inspired by work of Wang \cite{WangDensity}, the second author and Guaraco \cite{GasparGuaracoWeyl}, and Dey \cite{dey2022comparison}. We also compute $\{\omega_p^G\}$ for several pairs of group actions and manifolds, using work of \cite{ChodoshMantoulidisWidths, marx2025p, marx2026p}.


\item In \S \ref{sec.regularity}, we show the regularity of stationary varifolds arising from limits of solutions to the Allen--Cahn equation. We employ tools from Tonegawa--Wickramasekera \cite{TWStable}, propagating the regularity of stable solutions to the regularity of $G$-stable solution in $G$-invariant annuli. When $\cohom(G) \geq 3$, the proof of regularity is quite short. Using classical ideas, any sequence of solutions $\{u_{\eps_i}\}$ with uniformly bounded index and energy will converge to a $G$-invariant smooth minimal hypersurface, away from a finite number of $G$-invariant orbits. When $\cohom(G) \geq 3$, each of these orbits will be at least codimension $3$ or more, for which the regularity theory of \cite{TWStable} allows us to smooth across the singularities up to a set of codimension $7$. When $\cohom(G) \leq 7$, we can further show that the singularities lie in $M \backslash M^{reg}$ by working on the quotient directly.

\item In \S \ref{section.cohom.two}, we show regularity for the limiting varifolds when $\cohom(G) = 2$. In this setting, we handle principal orbits which correspond to singularities in our minimal surface by passing to the quotient and adapting the local theory of \cite{ChodoshMantoulidisWidths}. We can ignore singular orbits via the same arguments as before, and we assume there are no exceptional orbits.

\item In \S \ref{section.index}, we prove index bounds, i.e. if $\{u_{\eps_i}\}$ is a sequence of solutions with bounded energy and index at most $p$, then the regular part of the limiting minimal $G$-equivariant hypersurfaces will also have $G$-index at most $p$. This adapts a theorem of the second author \cite[Theorem A]{gaspar2020second}.

\item In \S \ref{section.ricci}, we prove a multiplicity one result in the positive Ricci setting. This draws from work of Bellettini \cite{bellettini2024multiplicity}, though the same arguments do not apply for cohomogeneity $2$ actions.
\end{enumerate}
\subsection{Acknowledgments}
The authors are grateful to the Center for Mathematical Modeling in Santiago, Chile where part of this work was carried out. The authors would like to thank Tongrui Wang, Christos Mantoulidis, Juncheng Wei, Costante Bellettini, Akashdeep Dey, and Renato Bettiol for fruitful conversations. The first author is supported by Fondecyt grant number 11230872 and by Centro de Modelamiento Matemático (CMM) BASAL fund FB210005 for center of excellence from ANID-Chile. The second author was supported by ANID (Agencia Nacional de Investigación y Desarrollo, Chile) FONDECYT Iniciación grant number 11230874. The third author is supported by NSF Grant 23--603.
\section{Background} \label{section.background}
Throughout the text, we assume $(M,g)$ to be a closed Riemannian $(n+1)$-dimensional manifold, and assume $G$ to be a compact Lie group acting as isometries on $M$ of cohomogeneity $\cohom(G)=l+1\geq 2$. We refer to \S \ref{section.appendix} for a more detailed introduction to group actions. We borrow the following notations from \cite[\S 2]{wang2022min}, adding $G$- in front of objects meaning they are $G$-invariant:
  \begin{itemize}
    \item a $G$-varifold $V$ satisfies $g_{\#} V=V$ for all $g\in G$; 
    \item a $G$-vector field $X$ satisfies $g_{*} X=X$ for all $g\in G$; 
    \item a $G$-map $F$ satisfies $g^{-1}\circ F\circ g=F,\forall g\in G$, (i.e. $F$ is $G$-equivariant); 
    \item a $G$-set ($G$-neighborhood) is an (open) set which is a union of orbits. 
  \end{itemize}
  We will also sometimes add a subscript or superscript `$G$' to signify $G$-invariance: 
	\begin{itemize}
		\item $\pi$: the projection $\pi:M\mapsto M/G$ defined by $p \mapsto [p]$;
		\item $B_\rho^G(p),~\overline{B}_\rho^G(p)$:
		open and closed geodesic tubes with radius $\rho$ around $G\cdot p$;
		\item $\mathfrak{X}^G(M)$: the space of $G$-vector fields on $M$; 
		\item $\an(p,s,t)$: the open tube $B_t^G(p)\setminus \overline{B}^G_s(p)$; 
		\item $T_qG\cdot p$: the tangent space of the orbit $G\cdot p$ at some point $q\in G\cdot p$; 
		\item[$\bullet$] $M^{reg}$: the union of orbits with principal orbit type. 
	\end{itemize}
We recall that a $G$-varifold is $G$-stationary in $M$ if and only if it is stationary in $M$ \cite[Lemma 2.2]{liu2021existence}. We also recall the notion of $G$-index as the maximal dimension of subspaces of $G$-vector fields, $P$, for which $\delta^2V\Big|_{P \backslash \{0\}}$ is negative definite. We notate this maximal dimension as $\text{Ind}_G(V)$. \nl 
\indent Following Hsiang--Lawson \cite{HsiangLawson} and T. Wang \cite{WangDensity}, when $\ell + 1 = \cohom(G)$, we define the conformal metric on the quotient, $M/G$, as 
\begin{equation} \label{eqn.hsiang.lawson.metric}
\overline{g} = \mathcal{V}^{2/\ell} g_{M/G}
\end{equation}
where $g_{M/G}$ is the induced metric on $M/G$ coming from the projection map, and $\mathcal{V}(p) = \text{Vol}_g(\pi^{-1}(p))$ is the fiber volume function. We refer to \cite[Section 4]{HsiangLawson} for the key analytic properties of this function. 

\subsection{Allen-Cahn energy functional} \label{assumptions AC}
For $\e>0$, we consider the \emph{Allen-Cahn energy functional}:
\begin{equation} \label{eqn.AC.energy}
    E_\e(u) = \int_M\e \frac{|\nabla u|^2}{2} + \frac{W(u)}{\e},
\end{equation}
defined for Sobolev functions $u \in W^{1,2}(M)$. Here, the function $W$ is a ``double well potential" and we refer the reader to Gaspar--Guaraco \cite{GasparGuaracoClosed, GasparGuaracoWeyl} or Hutchison--Tonegawa \cite{HutchinsonTonegawa} for the classical assumptions on $W$. In practice, the exact choice of $W$ will not matter until \S \ref{section.cohom.two}, where we will define and specify the Sine-Gordon potential. \nl 
\indent Recall that critical points of $E_\e$ are precisely the weak solutions of the \emph{Allen-Cahn equation}
    \[-\e \Delta u + \frac{1}{\e}W'(u) =0.\]

If $u \in W^{1,2}(M)$ is a critical point of $E_\e$, we will denote by $\mathrm{Ind}_\e(u;\Omega)$ and by $n_\e(u;\Omega)$ its \emph{Morse index} and its \emph{nullity} on any open set $\Omega \subset M$, respectively. Concretely, consider the bilinear form given by the second variation of $E_\e$ at $u$, that is
    \begin{align*}
        D^2E_\e(u)\colon W^{1,2}(M) \times W^{1,2}(M) \to \R, \\[2pt] \quad D^2E_\e(u)[\psi,\rho] = \int_M \e \langle \nabla \psi,\nabla \rho\rangle  + \frac{W''(u)}{\e}\psi\rho
    \end{align*}
Then for any open set $\Omega \subset M$ we write
    \begin{align*}
        \mathrm{Ind}_\e(u;\Omega) & = \max\{ \dim V \colon V  \ \text{subspace of} \ W_0^{1,2}(\Omega) \ \text{s.t.} \ DE^2_\e(u)|_{V\times V} \ \text{is negative definite}\}\\
        n_\e(u;\Omega) & = \dim \ker D^2E_\e(u) = \dim \{ v \in W_0^{1,2}(\Omega) \colon D^2E_\e(u)[v,\rho] = 0 , \forall \rho \in W^{1,2}(\Omega)\}.
    \end{align*}
For $\Omega=M$, we write $\mathrm{Ind}_\e(u)=\mathrm{Ind}_\e(u;M)$ and $n_\e(u)=n_\e(u;\Omega)$. Moreover, we recall that these numbers correspond to the number of negative Dirichlet eigenvalues and to the nullity of the linearized Allen-Cahn (or Jacobi) operator
    \[
    J_{u_\eps}\phi = -\eps\Delta \phi + \frac{W''(u_\eps)}{\eps }\phi,
    \]
acting on smooth functions $\phi \in C_0^\infty(\Omega)$.\nl
\indent Given a critical point $u_\e$ of $E_\e$, we introduce its energy density measure
    \[
    d\mu_{\eps_i}=\frac{1}{2\sigma}\left(\frac{\eps_i|\nabla u_{\eps_i}|^2}{2}+\frac{W(u_{})}{\eps_i}\right)\,d\mathcal{H}^{n+1},
    \]
and associate to it a $n$-varifold $V(u_{\eps_i})$ on $M$:
    \[
    V(u_{\eps_i})\,(\phi) = \int_{M \cup \{|\nabla u_{\eps_i}|> 0\}} \phi\left(x, I - \frac{\nabla u_{\eps_i}}{|\nabla u_{\eps_i}|}\otimes \frac{\nabla u_{\eps_i}}{|\nabla u_{\eps_i}|}\right) \,d\mu_{\eps_i}, \quad \text{for} \ \phi \in C(G_n(M))
    \]
We recall that these varifolds were introduced in \cite{PadillaTonegawa,HutchinsonTonegawa,TWStable} to study the $\e$-limit interface associated to solutions of the Allen-Cahn equation.
    
\subsection{Symmetric smooth approximations for $G$-invariant Caccioppoli sets}

We will denote by $\mathcal{C}_G(M)$ the set of Caccioppoli sets in $M$ which are $G$-invariant. The following auxiliary results ensures that any $E \in \mathcal{C}_G(M)$ can be approximated by closed $G$-invariant subsets with smooth boundary. This can be seen as a $G$-invariant adaptation of \cite[Proposition 2.10]{dey2022comparison}.

\begin{lemma} \label{smooth_approx}
Let $E \in \mathcal{C}_G(M)$ and $F=M\setminus E$. Then there exist two sequences $\{E_j\}_j$ and $\{F_j\}_j$ of $G$-invariant closed subsets of $M$ with the following properties:

\begin{enumerate}
    \item[(i)] $E_j, F_j \in \mathcal{C}_G(M)$, $[\![ E_j ]\!] + [\![ F_j ]\!] = [\![ M ]\!]$ and $M= E_j \cup F_j$ for every $j$.
    \item[(ii)] $\mathcal{H}^{n+1}(E\Delta E_j) = \|\mathbf{1}_{E_j} - \mathbf{1}_E\|_{L^1(M)} \to 0$ and $\mathcal{H}^{n+1}(F\Delta F_j) = \|\mathbf{1}_{F_j} - \mathbf{1}F\|_{L^1(M)} \to 0$.
    \item[(iii)] $\partial [\![E_j ]\!] = [\![\partial ^* E_j ]\!] = [\![ \partial ^* F_j]\!] = \partial [\![F_j]\!]$ converges to $\partial [\![E]\!] = \partial [\![ F ]\!]$ in the $\mathbf{F}$ norm.  In particular, $\partial[\![ E_j ]\!]=\partial [\![F_j]\!] \to \partial [\![ E ]\!] = \partial [\![F]\!]$ in the flat metric, and the corresponding varifolds converge in the weak topology.
    \item[(iv)] For all $j$, the sets $E_j \cap F_j$ is a smooth, embedded closed hypersurface in $M$ which coincides with the topological and the reduced boundaries of both $E_j$ and $F_j$.
    \item[(v)] For all $j$ and all $p \in E_j\cap F_j$, there exists a $G$-invariant geodesic tube $U$ around $G\cdot p$ such that $U\setminus(E_j\cap F_j)$ is the union of two disjoint, connected, $G$-invariant open sets $\mathcal{O}_1 \subset E_j$ and $\mathcal{O}_2 \subset F_j$.
    \item[(vi)] For every $p \in M$ and every $R>0$, there exists a subsequence $\{j_s\}_{s=1}^\infty$ such that
        \[
        \mathcal{H}^{n}(\partial B_t^G(p) \cap (E_{j_s} \Delta E)) \to 0 \quad \text{and} \quad \mathcal{H}^{n}(\partial B_t^G(p) \cap (F_{j_s} \Delta F)) \to 0
        \]
    for a.e. $t \in (0,R)$.
\end{enumerate}
\end{lemma}

\begin{proof}
As in \cite{dey2022comparison}, using Miranda-Pallara-Paronetto-Preunkert \cite{MPPP}, we first approximate each $E_j$ by a sequence $\{u_{j,i}\}_i \subset C^\infty(M,[0,1])$ in the following sense: 
\[
u_{j,i} \to \mathbf{1}_{E_j} \ \text{in} \ L^1(M), \qquad \text{and} \qquad \|Du_{j,i}\|(M) \to \|D\mathbf{1}_{E_{j}}\|(M).
\]

For each $i$, consider the averaged function $v_{j,i} \colon M \to [0,1]$ given by
\[
v_{j,i}(x) = \frac{1}{|G|} \int_G \phi_{\#}u_{j,i}(x)\,dV_G(\phi) = v_{j,i}(x) = \frac{1}{|G|} \int_G u_{j,i}(\phi x)\,dV_G(\phi).
\]

Here we recall that $V_G$ is the (bi-invariant) Haar measure of $G$ and we write $|G|=V_G(G)$. Since $\mathbf{1}_{E_j}$ is $G$-invariant, we can write it as $\mathbf{1}_{E_j}(x) = \frac{1}{|G|} \int_G \phi_{\#} \mathbf{1}_{E_j}(x)\,dV_G(\phi)$, so that
\begin{align*}
\|v_{j,i}-\mathbf{1}_{E_j}\|_{L^1(M)} &\leq \frac{1}{|G|} \int_M \int_G|u_{j,i}(\phi x)-\mathbf{1}_{E_{j,i}}(\phi x)|\,dV_G(\phi) \,d\mathrm{vol}_g(x) \\
&= \frac{1}{|G|} \int_G \|\phi_{\#}u_{j,i} - \phi_{\#}\mathbf{1}_{E_j}\|_{L^1(M)}\,dV_G(\phi) \to 0.
\end{align*}
In addition, $\|Dv_{j,i}\|(M) \to \|D\mathbf{1}_{E_j}\|(M)$ as well. To see this, note that $u_{j,i}$ is smooth and $|\nabla (u_{j,i}\circ \phi)(x)| = |D\phi^{-1}(\phi x) \nabla u_{j,i}(\phi x)| = |\nabla u_{j,i}(\phi x)|$, so that

\[
\|Dv_{j,i}\|(M) = \int_M|\nabla v_{j,i} | \,d\mathrm{vol}_g \leq\frac{1}{|G|} \int_M \int_G |\nabla u_{j,i}(\phi x)|\,dV_G(\phi)\,d\mathrm{vol}_g(x) = \int_M|\nabla u_{j,i}|\,d\mathrm{vol}_g \to \|D\mathbf{1}_{E_j}\|(M).
\]

On the other hand, since $v_{j,i} \to \mathbf{1}_{E_j}$ in $L^1$, the lower semicontinuity of the total variation ensures $\liminf_i \|D v_{j,i}\|(M) \geq \|D\mathbf{1}_{E_j}\|(M).$

Now the construction of $E_j$ and $F_j$, and the proof of properties 1.-5. follow that of \cite[Proposition 2.10]{dey2022comparison} without any major changes, noting that the sets $E_j$ and $F_j$ can be constructed as regular super- and sub-level sets of the smooth, $G$-invariant approximations $v_{j,i}$, yielding $G$-invariant, closed, Caccioppoli sets.
\end{proof}

\section{Codimension 1 equivariant min-max for phase transitions} \label{section.min.max}

\indent We note that the $G$-action on $(M,g)$ induces a right group action on the space of real-valued functions $M\to \R$ by composition, which we will denote as $\phi_\# u(x) = u\circ \phi(x) = u(\phi x)$, for $\phi \in G$. We consider the subspace $W^{1,2}_G(M) \subset W^{1,2}(M)$ of $G$-invariant functions, that is, of all $u \in W^{k,p}(M)$ such that $\phi_\# u  = u$ for all $\phi \in G$. Since this is a closed subspace, it is also a Hilbert space when endowed with the restriction of the inner product from $W^{1,2}(M)$.
In this section, we will construct equivariant critical points of $E_\e$ by employing variational methods for the functional $E_\e^G = E_\e|_{W^{1,2}_G(M)\setminus \{0\}}$. Note that its domain is a free $\Z_2$-space with the antipodal action $u \mapsto (-u)$, and that it is an incomplete Hilbert manifold.

If $u \in W^{1,2}_G(M)$ is a $G$-invariant critical point of $E_\e$, we denote its $G$-\emph{equivariant index} and \emph{nullity} by $\mathrm{Ind}^G_\e(u)$ and $n_\e^G(u)$, which are computed considering $D^2E_\e(u)$ as a bilinear form in $W^{1,2}_G(M)$, that is, only among $G$-invariant functions. We will say that $u$ is \emph{$G$-stable} if $\mathrm{Ind}_\e^G(u)=0$, that is, if $DE_\e^2(u)$ is positive semidefinite among $G$-invariant functions.
    
For $c \in \R_+$ and $p \in \N$, we denote 
    \[
    K_c^G(p) = \left\{ u \in W^{1,2}_G(M)\setminus \{0\} \  \colon \ E_\e(u)=c, \ D\left(E_e|_{W^{1,2}_G(M)\setminus \{0\}}\right)(u)=0, \ \mathrm{Ind}_\e^G(u) \leq p \leq \mathrm{Ind}_\e^G(u) + n_\e^G(u)\right\}.
    \]

\subsection{Equivariant mountain pass for the Allen-Cahn energy} \label{mountain pass}

In our first existence result, we extend the construction of a mountain-pass type critical point for the Allen-Cahn energy carried out by Guaraco in \cite{GuaracoMinmax} to the equivariant setting:

\begin{proposition} \label{1 parameter min-max}
Let
    \[
    \Gamma_G = \{ \gamma \in C^{0}([-1,1], W^{1,2}_G(M)) \ \colon\ \gamma(\pm 1) = \pm1\}
    \]
and consider the mountain pass value for $E_\epsilon$ associated to $\Gamma$, that is
    \[
    \omega_{AC,\epsilon}^G(M,g) = \inf_{\gamma \in \Gamma_G} \max_{t \in [-1,1]} E_\epsilon(\gamma(t)).
    \]
Then $\omega_{AC,\epsilon}^G(M,g)>0$ for every $\e>0$ and there exists a $G$-equivariant solution $v \in C^\infty(M,g)$ of the Allen-Cahn equation on $(M,g)$ with $|v|\leq 1$ and such that
    \begin{enumerate}
        \item[(i)] $E_\epsilon(v) = \omega_{AC,\epsilon}^G(M,g)$
        \item[(ii)] The equivariant Morse index of $v$ satisfies $\mathrm{Ind}_\eps^G(v) \leq 1$.
    \end{enumerate}
In addition, we have
    \[
    0 < \liminf_{\epsilon \to 0} \omega^G_{AC,\epsilon}(M,g) \leq \limsup_{\epsilon \to 0} \omega^G_{AC,\epsilon}(M,g) <\infty
    \]
\end{proposition}

\begin{proof}
The proof follows the arguments of \cite[Proposition 4.5]{GuaracoMinmax} closely, and it mirrors the equivariant min-max construction of T. Wang \cite{WangJGA} in the Almgren-Pitts setting.

Since $W^{1,2}_G(M)$ is a closed subspace of $W^{1,2}(M)$ and $E_\e$ satisfies the Palais-Smale property (truncating the double-well potential outside a compact interval containing $[-1,1]$ in its interior, if necessary) along bounded sequences of functions in $W^{1,2}(M)$, so does its restriction to $W_G^{1,2}(M)$. In addition, the min-max values $\omega_{AC,\epsilon}(M,g)$ for $E_\e$ associated to the class of paths of (not necessarily equivariant) functions joining the constants $(-1)$ and $(+1)$ satisfy
    \begin{equation} \label{comparison equivariant values}
        \omega_{AC,\epsilon}(M,g) \leq \omega_{AC,\epsilon}^G(M,g).
    \end{equation}
This shows that $\omega_{AC,\epsilon}^G(M,g) >0 = E_\epsilon(\pm 1)$. Additionally, any critical point of $E_{\epsilon}|_{W^{1,2}_G(M)}$ is a critical point of the Allen-Cahn energy, by Palais’ principle of Symmetric Criticality \cite{Palais}, and its index is precisely the $G$-equivariant index. In fact, since the Lie group $G$ acts on $W^{1,2}(M)$ by isometries with respect to the Sobolev norm, the ($W^{1,2}-$)gradient of $E_\e$ is equivariant with respect to the $G$-action, which ensures that this gradient is tangent to $W^{1,2}_G(M)$. Hence, if $u \in W^{1,2}(M)$ is critical for the constrained energy $E_\e|_{X_G}$, then $DE_\e(u) = 0$ among all functions in $W^{1,2}(M)$.

Therefore, the existence of the solution $v$ follows from \cite[Corollary 10.5]{Ghoussoub} (see also \cite[Theorem 4.3]{GuaracoMinmax}). The lower bound for the $\epsilon$ limit of $\omega_{AC,\epsilon}^G(M,g)$ follows directly from \eqref{comparison equivariant values} and \cite[Proposition 5.2]{GuaracoMinmax}, in which $c_\epsilon=\omega_{AC,\epsilon}(M,g)$ is shown to have a positive lower bound independent of the parameter $\e$.

For the upper bound, it suffices to observe that we can use the path of functions constructed in \cite[Section 7.5]{GuaracoMinmax} using the heteroclinic solutions of the Allen-Cahn equation and the level sets of a Morse function. To see this, consider the level sets $\{\Sigma_t=f^{-1}(t)\}_{t \in [0,1]}$ of a $G$-equivariant Morse function $f \colon M \to  [0,1]$, in the sense of \cite{Wasserman}, and let $\gamma(t)$ be the composition of the heteroclinic, 1-dimensional solution of the $\epsilon$-Allen-Cahn equation with the following signed distance function to $\Sigma_t$:
    \[
    d_{t,\delta}(x) = \left\{ \begin{array}{cc} \mathrm{sign}(f(x)-t)\cdot \delta, & \text{if} \ \mathrm{dist}(x,\Sigma_t) \geq \delta,  \\ \mathrm{sign}(f(x)-t)\cdot \mathrm{dist}(x,\Sigma_t), & \text{if} \ \mathrm{dist}(x,\Sigma_t) \leq \delta,\end{array} \right.
    \]
where $\delta>0$. Then $\{\gamma(t)\}_{t \in [0,1]}$ is a path of $G$-equivariant Lipschitz functions and its energy satisfies
    \[E_\epsilon(\gamma(t)) \leq C_{\delta}\cdot \mathcal{H}^{n-1}(\Sigma_t) + o(1),\]
with respect to $\epsilon \to 0$. Here $C_\delta \to 2\sigma$ as $\delta \to 0$ and $\sigma>0$ depends only on the double-well potential $W$. By connecting $\gamma(0)$ and $\gamma(1)$ linearly to the constants $(-1)$ and $+1$ and re-parameterizing the resulting path, we construct a path in $\Gamma_G$ whose energy remains bounded from above by $2\sigma\cdot \sup_{t} \mathcal{H}^{n-1}(\Sigma_t)$, that is,
    \[
    \limsup_{\epsilon \to 0}\omega_{AC,\epsilon}^G(M,g) \leq 2\sigma\cdot \sup_{t \in [0,1]}\mathcal{H}^{n-1}(\Sigma_t). \qedhere
    \]
\end{proof}

We conclude this section with a brief discussion about the \emph{least positive energy} equivariant solutions, that is, solutions whose energy is given by
    \begin{equation} \label{least energy}
    a_\e^G=\min\{ E_\e(u) \colon u \in W_G^{1,2}(M), DE_\e(u)=0 \ \text{and} \ E_\e(u_\e)>0\}.
    \end{equation}
The existence of solutions at this energy level follows readily from the Palais-Smale property and the nondegeneracy of the double well-potential, by the same proof as in \cite[Theorem 2.1 (1)]{GasparGuaracoClosed}. It naturally implies $a_\e \leq \omega_{AC,\e}^G$. In addition, if $E_\e(v_\e)=a_\e^G$ for a $G$-unstable solution $v_\e$, then $v_\e$ can be obtained as an equivariant min-max solution. Even more so, under this assumption, there exists a continuous path $\gamma \colon [-1,1] \to W^{1,2}_G(M)$ with $\gamma(\pm 1) = \pm 1$ (the constant functions), $\gamma(0) = v_\e$ and such that
    \[\omega_{AC,\e}^G(M,g) \leq \sup_{t \in [-1,1]} E_\e(\gamma(t)) = E_\e(v_\e) = a^G_\e,\]
so we obtain $a^G_\e=\omega_{AC,\e}^G(M,g)$ in this case. The existence of such path $\gamma$ follows by the same argument of \cite[Theorem 2.1 (2)]{GasparGuaracoClosed}, using the first eigenfunction associated to $D^2E_\e(v_\e)$ in $W^{1,2}_G(M)$ and the parabolic Allen-Cahn equation, observing that its solutions remain $G$-invariant for all times provided the initial data is $G$-invariant, by the uniqueness of solutions \cite[Lemma 2.3]{GasparGuaracoClosed}. Finally, such a solution $v_\e$ must have $G$-invariant Morse index $1$, as otherwise we would be able to construct a path of $G$-invariant functions with energy strictly below $E_\e(v_\e)$, by using the eigenfunction associated to the second eigenvalue of $D^2E_\e(v_\e)$.

We summarize the discussion in the following

\begin{proposition}
For every $\e>0$, there exists a $G$-invariant solution of the Allen-Cahn equation with energy $a_\e^G$, as defined in \eqref{least energy}. Moreover, any such solution is either $G$-stable or it can be obtained via the $1$-parameter equivariant min-max, as described in Proposition \eqref{1 parameter min-max}. In the latter case, the solution has $G$-invariant Morse index $1$, and $a_\e=\omega_{AC,\e}^G(M,g)$.
\end{proposition}

\subsection{The equivariant phase transition spectrum} \label{phase transitions spectrum}

Our next goal is to extend the min-max construction to $p$-dimensional families of subsets of $W^{1,2}_G(M)$, in order to formulate a $G$-equivariant analogue of the phase transition spectrum defined in \cite{GasparGuaracoClosed,GasparGuaracoWeyl}.

Following the construction of \cite{GasparGuaracoClosed}, for each positive integer $p$, we let
    \[
    \cF^G_p(M) = \{ A \subset W^{1,2}_G(M)\setminus \{0\} \colon A \ \text{compact}, \ \text{symmetric}, \ \text{and} \ \Ind_{\Z_2}(A)\geq p+1\},
    \]
where $\Ind_{\Z_2}$ is the $\Z_2$-cohomological index of Fadell-Rabinowitz \cite{FadellRabinowitz} (see also Appendix B in \cite{GasparGuaracoClosed}). We define the \emph{$G$-equivariant phase transition spectrum of $(M,g)$} as the min-max sequence of critical values of $E_\e^G$ associated to the cohomological families $\cF_p$, that is the sequence $\{\omega^G_{p,\e}(M,g)\}_{p \in \N}$ defined by
    \[
    \omega^G_{p,\e}(M,g) = \inf_{A \in \cF^G_p(M)} \sup \{ E_\e(u) \colon u \in A\}. 
    \]

\begin{theorem}[Equivariant Min-Max Theorem for the Allen-Cahn energy] \phantom{m}\\
\noindent\emph{(1)} The $G$-equivariant phase transition spectrum satisfies 
    \[
    \omega^G_{p,\e}(M,g) \leq E_\e(0) = \frac{W(0)}{\e}\mathrm{vol}(M,g),\]
with equality for every sufficiently large $p$ depending on $(M,g)$, $W$ and $\e$.\smallskip

\noindent\emph{(2)} If $\omega^G_{p,\e}(M,g) < E_\e(0)$, then there exists a $G$-invariant, smooth solution of the Allen-Cahn equation $u \in C^\infty(M)$ with 
    \[
    |u|\leq 1, \quad E_\e(u) = \omega^G_{p,\e}(M,g), \quad \mathrm{Ind}^G_\e(u) \leq p \leq \mathrm{Ind}^G_\e(u) + n^G_\e(u).
    \]
For each fixed $p \in \N$, this holds true provided $\e$ is sufficiently small.\smallskip

\noindent\emph{(3)} If $\omega_{p,\e}^G(M,g) = \omega_{p+k,\e}^G(M,g)$ for some $p,k\in \N$, then there are infinitely many critical points of $E_\e$ at this energy level. More precisely,
    \[
    \Ind_{\Z_2}\left(\, K^G_{\omega_{p,\e}^G(M,g)}(p+k) \, \right) \geq k+1.
    \]
\end{theorem}

\begin{proof}
The proof follows closely the proof of Theorem 3.3 in \cite{GasparGuaracoClosed}, by adapting the techniques from \cite[Chapters 9 and 10]{Ghoussoub}, together with the following straightforward observations:

\begin{itemize}
    \item If $u \in W^{1,2}_G(M)\setminus \{0\}$ is a critical point of $E_\e|_{W^{1,2}_G(M)\setminus \{0\}}$, then it is a critical point of the unconstrained energy functional $E_\e$, and hence a solution of the Allen-Cahn equation. This is also a consequence of Palais' Principle of Symmetric Criticality \cite{Palais}. 
    \item If $A \in \cF_p^G$, then $\lambda A = \{ \lambda u \colon u \in A\} \in \cF_p^G$ for any $\lambda \in \R_+$. This ensures, as in \cite{GasparGuaracoClosed}, that $\omega_{p,\e}^G(M,g) \leq E_\e(0)$.  
    \item Recall that $c_\e(p)=\omega_{p,\e}^{\mathrm{id}}(M,g)$ is the min-max critical value of $E_\e$ associated to $\Z_2$-symmetric families of compact subsets of $W^{1,2}(M)\setminus\{0\}$ with $\Ind_{\Z_2} \geq p+1$, \cite{GasparGuaracoClosed}. Trivially, we have $c_\e(p)  \leq \omega^G_{p,\e}(M,g)$, hence (1) follows from the corresponding statement in \cite{GasparGuaracoClosed}.
    \item The proof that $\omega_{p,\e}^G(M,g) < E_\e(0)$ for sufficiently small $\e>0$ can be carried out as in \cite{GasparGuaracoClosed}, replacing the Laplace spectrum of $(M,g)$ by its equivariant counterpart, but it also follows from the ($\e$-independent) upper bounds for $\omega^G_{\e,p}(M,g)$ shown below.\qedhere
\end{itemize}
\end{proof}

\subsection{Asymptotics of the equivariant phase transition spectrum}

In this section we study the asymptotic growth of the $\e$-limits of $\omega_{p,\e}(M,g)$, akin to Theorems 10 and 13 in \cite{WangJGA}. First, we give a short proof of the sublinear growth (with respect to $p$) of these limits, with an exponent determined by the cohomogeneity $(l+1)$ of the $G$-action:

\begin{theorem}
There exists $C=C(M,g,G,W)$ such that
    \[C^{-1} p^{\frac{1}{l+1}} \leq \liminf_{\e \to 0^+} \omega^G_{p,\e}(M,g) \leq \limsup_{\e \to 0^+} \omega^G_{p,\e}(M,g) \leq Cp^{\frac{1}{l+1}}.\]
\end{theorem}

\begin{proof}
We first prove the sublinear upper bound by lifting the piecewise linear sweepout constructed in \cite{GasparGuaracoClosed} from the quotient $M/G$ to $M$.

As observed by Wang in \cite{WangJGA}, by the work of Verona \cite{Verona} and Illman \cite{Illman1, Illman2} on equivariant triangulations, there exists an $l$-dimensional cubical subcomplex $K$ of some $I^m$ and a bi-Lipschitz map $f \colon M/G \to K$. Therefore, by the construction of \cite[Section 4]{GasparGuaracoClosed}, which can be applied to any cubical subcomplex of $I^m$, there exists a family $\{h_a\}_{a \in S^p}$ of (nonzero) Lipschitz functions $K \to \R$ such that
\begin{enumerate}
    \item[(i)] $h_{(-a)} = -h_a$ for all $a \in S^p$;
    \item[(ii)] If $a_j$ is a sequence in $S^p$ with $a_j \to a \in S^p$, then $\mathbb{H}_\e \circ h_{a_j} \to \mathbb{H}_\e \circ h_{a}$ and $\nabla(\mathbb{H}_\e \circ h_{a_j}) \to \nabla(\mathbb{H}_\e \circ h_{a})$ a.e. on $K$.
    \item[(iii)] For every $a \in S^p$, the sum of the Allen-Cahn energies of $\mathbb{H}\circ h_a$ over all $l$-dimensional cells of $K$ is bounded from above by $Cp^{1/(l+1)}$, where $C=C(M,g,G,W)>0$ does not depend on $a$ nor on $\e>0$.
\end{enumerate}
Since the projection $\pi \colon M \to M/G$ is a $1$-Lipschitz map (see \cite[Section 8.12]{EellsFuglede}), it follows that the set $\{\mathbb{H}_\e \circ h_a \circ f \circ  \pi\  \colon \ a \in S^p\} \subset W^{1,2}_G(M)$ is in $\cF_G^p(M)$ and 
    \[
    E_\e(\mathbb{H}_\e \circ h_a \circ f \circ \pi) \leq C_1\cdot Cp^{1/l}, \ \text{for all} \ a \in S^p,
    \]
where we used \cite[Lemma A.2]{GasparGuaracoClosed} and $C_1$ depends on the Lipschitz constants of $f$ and $\pi$. Therefore $\omega_{p,\e}^G(M,g) \leq C_1C \ p^{1/(l+1)}$.

To prove the lower bound, let $\tilde \Omega$ be a compact, regular domain contained in $M^{reg}/G$ with smooth boundary, and let $\Omega = \pi^{-1}(\tilde \Omega)$. Note that $\Omega$ is $G$-invariant and that the restriction map $u \mapsto u|_{\Omega}$ defines a linear continuous map $W^{1,2}_G(M) \to W^{1,2}_G(\Omega)$. Consequently, if $A \in \cF^G_p(M)$, then $\{u|_\Omega \colon u \in A\} \in \cF^G_p(\Omega)$. Since $E_\e(u) \geq E_\e(u|_\Omega)$, this shows that $\omega_{p,\e}^G(\Omega,g) \leq \omega_{p,\e}^G(M,g)$. 

We can endow the quotient $M^{reg}/G$ with a Riemannian metric $g_{M/G}$ which makes the quotient map $\pi \colon (M^{reg},g) \to (M^{reg}/G,g_{M/G})$ a Riemannian submersion. We claim that there exists a constant $C$ depending only on $(M,g)$ and $G$ (and on the choice of $\Omega$, but independent of $(p,\e)$) such that
    \[
    \omega_{p,\e}(\tilde \Omega,g_{M/G}) \leq C\cdot \omega_{p,\e}^G(\Omega,g),
    \]
compare with Theorem 4.4 in \cite{WangDensity}. Since $\dim \tilde \Omega = l+1$, by the sublinear growth of the volume spectrum of $\omega_{p,\e}^G(\tilde \Omega,g_{M/G})$, see e.g. \cite{GasparGuaracoWeyl}, this inequality readily implies the claimed lower bound. For this purpose, we observe that for any $u \in W^{1,2}_G(\Omega)$ there exists a unique $\tilde u \in W^{1,2}(\tilde \Omega)$ such that $\tilde u \circ \pi = u$, and that the corresponding map $W^{1,2}_G(\Omega) \to W^{1,2}(\tilde \Omega)$ is linear and continuous. In fact, the existence and uniqueness of the function $\tilde u$ follows from the equivariance of $u$. The linearity is straightforward, and the continuity is then a consequence of the inequalities
    \begin{align*}
        \|u\|_{L^2(\Omega,g)}^2 &\leq \left(\sup_{y \in \tilde \Omega} \mathrm{vol}(G\cdot y,g)\right)\cdot \|\tilde u\|^2_{L^2(\tilde \Omega,g_{M/G})}, \\
        \quad \|\nabla^g u\|_{L^2(\Omega,g)}^2 &\leq \left(\sup_{y \in \tilde \Omega} \mathrm{vol}(G\cdot y, g)\right)\cdot \| \nabla^{g_{M/G}} \tilde u\|^2_{L^2(\tilde \Omega, g_{M/G})}
    \end{align*}
see e.g. \cite[Lemma 7]{ChenGasparBerger} for a proof. This proves that $u\mapsto \tilde u$ is an odd continuous map, and ensures that if $A \in \cF^G_p(\Omega)$, then $\{\tilde u \colon u \in A\} \in \cF^G_p(\tilde \Omega)$

Furthermore, the second inequality above, together with the Fubini Theorem for Riemannian submersions \cite[Chapter II, Theorem 5.6]{Sakai}, shows that
    \[E_\e(\tilde u; \tilde \Omega, g_{M/G}) \leq C\cdot E_\e(u; \Omega,g),\]
where $C$ is again the maximum volume of a $G$-orbit of a point in $\Omega$. Therefore, 
    \begin{align*}
        \omega_{p,\e}^G(\tilde \Omega, g_{M/G}) & = \inf_{A \in \cF^G_p(\tilde \Omega)} \sup_{v \in A} E_\e(v; \tilde \Omega, g_{M/G}) \leq \inf_{A \in \cF^G_p(\Omega)} \sup_{u \in A} E_\e(\tilde u; \tilde \Omega, g_{M/G})\\
        & \leq \inf_{A \in \cF^G_p(\Omega)} \sup_{u \in A} \left( C\cdot E_\e(u;\Omega, g)\right) \leq C \cdot \omega_{p,\e}^G(\Omega,g).\qedhere
    \end{align*}
\end{proof}

To conclude this section, we will compare the equivariant Allen-Cahn widths of $(M,g)$ and the widths of the quotient space $M/G$, assuming there are no nonprincipal orbits. Let
    \begin{equation} \label{volume orbits}
        \mathcal{V}\colon M^{reg}/G \to \R_{>0}, \qquad \mathcal{V}(\pi(p)) = \mathrm{vol}^g_{n-\ell}(\pi^{-1}(\pi(p)) = \mathcal{H}^{n-\ell}(G\cdot p).
    \end{equation}
For $\Omega \subset M$ we denote by $E_\e^{(\Omega,g)}$ the restriction of the Allen-Cahn energy functional to $W^{1,2}(\Omega)$, and similarly for domains $D \subset M^{reg}/G$. Inspired by \cite[Theorem 4.4]{WangDensity}, we show:

\begin{lemma} \label{lem.energy.compare}
If $\Omega\subset \! \subset M^{reg}$ is a $G$-invariant domain with (piecewise) smooth boundary and which contains principal orbits only and $\tilde g_{M/G}$ is the conformal metric $\tilde g_{M/G}=\mathcal{V}^{2/\ell}\cdot g_{M/G}$ on $M^{reg}/G$, then
    \[E_{c\epsilon}^{(\Omega/G,\tilde g_{M/G})}(u) - E_{\epsilon}^{(\Omega,g)}(u\circ \pi) = \int_{\Omega/G}\left[ \frac{\epsilon}{2} \left| \nabla^{g_{M/G}}u \right|^2 \cdot \left( \frac{c}{\mathcal V^{1/\ell}} - 1 \right) + \frac{W(u)}{\epsilon}\left( \frac{\mathcal V^{1/\ell}}{c}-1 \right)\right] \mathcal V \, d\mu_{g_{M/G}}\]
for every $u \in W^{1,2}(\Omega/G)$ and every $c>0$.
\end{lemma}

\begin{proof}
Note that
    \[|\nabla^{\tilde g_{M/G}}u|^2 = \frac{1}{\mathcal V^{2/\ell}}|\nabla^{g_{M/G}}u|^2, \qquad d\mu_{\tilde g_{M/G}} = \mathcal V^{1+\frac{1}{\ell}}d\mu_{g_{M/G}}, \qquad |\nabla^g (u\circ \pi)|^2 = |\nabla ^{g_{M/G}}u|^2\]
Using Fubini’s Theorem for the Riemannian submersion $\pi \colon (\Omega,g) \to (\Omega/G,g_{M/G})$, we compute
    \begin{align*}
        E_{c\epsilon}^{(\Omega/G,\tilde g_{M/G})}(u) - E_{\epsilon}^{(\Omega,g)}(u\circ \pi) & = \int_{\Omega/G} \left[\frac{c\epsilon}{2}\frac{|\nabla^{g_{M/G}}u|^2}{\mathcal V^{2/\ell}} + \frac{W(u)}{c\epsilon}\right]{\mathcal V}^{1+\frac{1}{\ell}}d\mu_{g_{M/G}} \\
        & \qquad \qquad \qquad - \int_{\Omega/G} \left[\frac{\epsilon}{2}|\nabla^{g_{M/G}}u|^2 + \frac{W(u)}{\epsilon}\right] \mathcal V\,d\mu_{g_{M/G}}\\
& = \int_{\Omega/G} \left[\frac{\epsilon}{2}{|\nabla^{g_{M/G}}u|^2}\cdot\frac{c}{\mathcal V^{1/\ell}} + \frac{W(u)}{\epsilon}\cdot\frac{1}{c/\mathcal V^{1/\ell}}\right]\mathcal Vd\mu_{g_{M/G}} \\ &\qquad \qquad \qquad - \int_{\Omega/G} \left[\frac{\epsilon}{2}|\nabla^{g_{M/G}}u|^2 + \frac{W(u)}{\epsilon}\right] \mathcal V\,d\mu_{g_{M/G}}
    \end{align*}
By collecting the terms we obtain the claimed result.
\end{proof}

\begin{corollary}\label{const.vol}
If $\Omega\subset M$ is a $G$-invariant domain with (piecewise) smooth boundary and which contains principal orbits only, then
    \[|\omega_{p,AC}{(\Omega/G,\tilde g_{M/G})} - \omega_{p,AC}^G(\Omega,g)| \leq 2\left(\frac{\max_\Omega \mathcal V^{1/\ell}}{\min_\Omega \mathcal V^{1/\ell}}-1\right)\cdot  \omega_{p,AC}^G(\Omega,g)\]
In particular, if the orbits of $G$ have constant volume, then $\omega_{p,AC}{(\Omega/G,\tilde g_{M/G})} = \omega_{p,AC}^G(\Omega,g)$
\end{corollary}

\input{comparison}
\input{Regularity}
\input{Index_Bounds}

\input{Ricci_Pos}
\input{Appendix}
\bibliographystyle{amsalpha}
\bibliography{references}

\end{document}

%% file: comparison.tex
\section{Comparison with $G$-invariant min-max for minimal hypersurfaces} \label{section.comparison}

The goal of this section is to compare the $\e$ limits of the phase transition spectrum $\omega_{p,\e}^G(M,g)$ and the \emph{$G$-equivariant volume spectrum of} $(M,g)$, as studied in recent work by T. Wang \cite{WangDensity} and extending the results by Marques-Neves \cite{MarquesNevesWillmore,MarquesNevesPositive,marques2016morse}, based on Almgren-Pitts min-max theory for minimal hypersurfaces, to the equivariant setting. Concretely, we will show:

\begin{theorem} \label{comparison}
Under the conditions described above, the equivariant Allen-Cahn phase transition spectrum of $(M,g)$ satisfies
    \[
    \frac{1}{2\sigma} \lim_{\epsilon \to 0^+}\omega_{p,\e}^G(M,g) = \omega^G_{p}(M,g),
    \]
where $\{\omega_p^G(M,g)\}_{p \in \N}$ is the equivariant volume spectrum of $(M,g)$, cf. \cite[Section 3]{WangDensity}.
\end{theorem}

This comparison extends the previous work by the second named author and Guaraco \cite{GasparGuaracoClosed} and by A. Dey \cite{dey2022comparison}. In particular, the latter established this convergence result for a trivial action of $G=\{\mathrm{id}\}$ and provided a strong connection between the Allen-Cahn and Almgren-Pitts min-max approaches. Ultimately, this was used in concrete examples to compute certain $p$-widths using fundamental results about solutions of the Allen-Cahn equation, for instance in \cite{ChodoshMantoulidisWidths}. \nl 

\indent As a corollary of the comparison between the equivariant phase transitions and volume spectra and \cite[Theorem 4.8]{WangDensity}, we obtain the following Weyl-type law in the spirit of \cite{LiokumovichMarquesNeves,GasparGuaracoWeyl} (see also \cite{GuthLiokumovich}). Recall that the $G$ action on $(M,g)$ has cohomogeneity $(l+1) = \dim(M^{reg}/G)$.

\begin{theorem} \label{weyl}

The equivariant phase transition spectrum $\omega_{p,AC}^G(M,g) := \lim_{\epsilon \to 0}\omega^G_{p,\epsilon}(M,g)$ of $(M,g)$ satisfies:
    \[
        \lim_{p \to \infty} p^{-\frac{1}{\ell+1}}\omega_{p,AC}^G(M,g) = \tau(\ell) \cdot {\mathrm{vol}(M^{reg}/G,\tilde g_{M/G})^{\frac{1}{\ell+1}}}
    \]
where $\tau(\ell)$ is a dimensional constant.
\end{theorem}

Along this section, we will use the notation introduced in introduction of \cite{GasparGuaracoClosed} regarding symmetric cubical complexes. In particular, we recall that $Q^m=[-1,1]^m$ and consider the $m$-dimensional cubical complex $Q(m,k)$ obtained by equal subdivision of each interval $[-1,1]$ by $2\cdot 3^k$ subintervals of equal length, for each $k\in \Z_{>0}$. We will denote by $\mathcal{C}_p$ the set of all connected, compact cubical complexes $\tilde X$ that have a free $\Z_2$ action $T \colon \tilde X \to \tilde X$, which have $\mathrm{ind}_{\Z_2}(\tilde X)\geq p+1$, and whose corresponding orbit space is a subcomplex of $Q(m,k)$, for some $m,z\in \Z_{>0}$. 

We can associate to each such $\tilde X \in \mathcal{C}_p$ a min-max value for the $G$-equivariant Allen-Cahn energy by letting
    \begin{equation} \label{homotopy min max AC}
        \mathbf{L}_\e(\tilde X; M,g; G) = \inf_{h \in \Gamma_G(X)}\sup_{x \in \tilde X} E_\e(h(x)),
    \end{equation}
where $\Gamma_G(\tilde X)$ is the space of all continuous, $\Z_2$-equivariant functions $h\colon \tilde X \to W^{1,2}_G(M)\setminus\{0\}$. We note that the proof of \cite[Lemma 6.2]{GasparGuaracoClosed} (see also p. 1025 in \cite{dey2022comparison}) carries over to equivariant function spaces to ensure that
    \begin{equation} \label{complexes}
        \omega_{p,\e}^G(M,g) = \inf_{\tilde X\in \mathcal{C}_p}  \mathbf{L}_\e(\tilde X; M,g; G) = \inf_{\tilde X \in \mathcal{C}_p \colon \dim(\tilde X) = p} \mathbf{L}_\e(\tilde X; M,g; G).
    \end{equation}

\subsection{From sweepouts by invariant functions to sweepouts by invariant cycles}

We will first follow closely the strategy of \cite{GasparGuaracoClosed} to prove:

\begin{proposition} \label{comparison bound 1}
    The equivariant $p$-widths for Allen-Cahn satisfy:
    \begin{equation}
        \frac{1}{2\sigma}\liminf_{\epsilon \to 0^+} \omega_{p,\e}^G(M,g) \geq \omega^G_{p}(M,g),
    \end{equation}
\end{proposition}

For small $\e>0$, by \eqref{complexes}, there exist a $p$-dimensional cubical complex $\tilde X\in \mathcal{C}_p$ and a continuous odd map $h\colon \tilde X \to W_G^{1,2}(M)\setminus \{0\}$ with
    \[\sup_{x \in X} E_\e(h(x)) \leq \omega_{p,\e}^G(M,g) + \epsilon.\]
We would like to produce a \emph{$(G,p)$-sweepout} $\phi \colon \tilde X \to \mathcal{Z}_{n}^G(M;\Z_2)$, according to \cite[Definition 3.4]{WangDensity} -- namely, a continuous map with respect to the flat metric, that detects the $p$-th power of the fundamental class of $\mathcal{Z}_{n}^G(M;\Z_2)$ (see \cite[Theorem 3.1]{WangDensity}) and which satisfies the \emph{no concentration of mass on orbits condition} (as we discuss below) with $\sup_{x \in \tilde X} \mathbf{M}(\phi(x))$ bounded from above by $\sup_{x \in \tilde X} E_\epsilon(h_x)+o(1)$ with respect to $\e$. Here we write, for the sake of notation, $h_x:=h(x)$, for each $x \in \tilde X$. \nl 
\indent Following \cite{GasparGuaracoClosed}, we select suitable level sets for the renormalized functions $\tilde h_x :=\Psi \circ h_x$, where
    \[
    \Psi(t) = \int_0^t\sqrt{W(s)/2}\,ds.
    \]
We recall that the rescaling introduced by the function $\Psi$ appears in Modica’s work \cite[p 494]{Modica} and it is used to obtain a $BV$ limit, as $\epsilon \downarrow 0$, for functions with uniformly bounded energy. Concretely, if $|h_x|\leq 1$, then $\tilde h_x$ satisfies $|\tilde h_x| \leq \sigma/2$, where $\sigma = \int_{-1}^1\sqrt{W(s)/2}\,ds$, and
    \[
    \int_U |\nabla \tilde h_x| \, d\mathrm{vol}_g \leq \frac{1}{2}\int_U \left( \frac{\epsilon}{2} |\nabla h_x|^2 + \frac{1}{\epsilon}W(h_x) \right) \, d\mathrm{vol}_g
    \]
On the other hand, if we consider the measure-theoretic level sets of $\tilde h_x$, namely the total variation measures $\|\partial \{\tilde h_x > s\}\|=\|D\mathbf{1}_{\{\tilde h_x > s\}}\|$, which are well-defined for a.e. $s \in [-\sigma/2,\sigma/2]$, then by coarea formula,
    \[
    \int_{-s_0}^{s_0}\|\partial \{\tilde h_x>s\}\|(U)\,ds \leq \int_U |\nabla \tilde h_x|\leq \frac{1}{2}\int_U \left( \frac{\epsilon}{2} |\nabla h_x|^2 + \frac{1}{\epsilon}W(h_x) \right) \, d\mathrm{vol}_g = E_\e^{(U,g)}(h_x)/2
    \]
for any $s_0>0$. 

Fix $\tilde \sigma \in (0,\sigma/2)$. The argument above shows that, for every $x \in X$, there exists a $\tilde s_x \in [-\tilde \sigma, \tilde \sigma]$ for which
    \[
    2\tilde \sigma \|\partial \{\tilde h_x > \tilde s_x\}\|(M) \leq \int_{-\tilde \sigma}^{\tilde \sigma}\|\partial \{\tilde h_x>s\}\|(M)\,ds \leq \frac{1}{2} E^{(M,g)}_\epsilon(h_x) \leq \frac{1}{2}(\omega_{p,\e}^G(M,g) +\e).
    \]
Since $h_x$ is assumed to be $\Z_2$-equivariant, with respect to the action $T\colon \tilde X \to \tilde X$ (and hence $\{h_{T(x)}>s\} = \{-h_x > s\} = \{h_x < -s\}$), and the corresponding level sets have zero $\mathcal{H}^n$-measure for a.e. $s$, we can choose $\tilde s_x$ so that $\tilde s_{T(x)} = -\tilde s_x$. Moreover, since the functions $h_x$ and $\tilde h_x$ are $G$-invariant, so are their (super/sub-)level sets. Thus, we obtain a $\Z_2$-invariant map $\tilde\phi_0\colon \tilde X \to \mathcal{Z}^G_{n}(M;\Z_2)$ with mass $\leq (\omega_{p,\e}^G(M,g)+\e)/4\tilde\sigma$ given by
    \[
    \tilde\phi_0(x)=\partial [\![ h_x > s_x]\!],
    \]
where $s_x = \Psi^{-1}(\tilde s_x) \in (-1,1)$, which then descends to a map $\phi_0 \colon X \to \mathcal{Z}_n^G(M,\Z_2)$ from the orbit space of $\tilde X$. This is important in \cite[Section 6.9]{GasparGuaracoClosed} to ensure the non-triviality of the resulting map $X \to \mathcal{Z}_n^G(M,\Z_2)$.

Since the choice of $s_x$ is quite arbitrary, the map $\phi_0$ might fail to be continuous, so that the actual sweepout is constructed by discretizing the cubical complex $\tilde X$ into finer cubical divisions and by using interpolation theorems from min-max theory for the area functional, e.g. from \cite{MarquesNevesWillmore}, for its restriction to the set $X[j]_0$ of vertices of a sufficiently fine subdivision $X[j]$ of $X$.

The $G$-equivariant versions of such interpolation theorems already appear in Wang’s work, for instance in \cite{WangJGA,WangFB}. On the other hand, these results often assume the \emph{no concentration of mass on orbits} technical condition. This means that the mass of such equivariant sweepouts cannot accumulate of small $G$-invariant tubes. More precisely, we say that a map $\phi \colon A \to \mathcal{Z}_n^G(M,\Z_2)$ satisfies this condition if:
    \[
    \lim_{r \to 0} \sup \{\|\phi(x)\|(B^G_r(p)) \colon x \in A, p \in M\} = 0.
    \]

To ensure we can use the $G$-equivariant interpolation results, we will show that, for a family of finitely many Sobolev functions, almost every level set satisfies the no concentration of mass on orbits condition:

\begin{lemma}
Let $K \in W_G^{1,1}(M;[-1,1])$ be a finite set and write
    \[
    \mathcal{B}(K) = \left\{ t \in [-1,1] \colon \liminf_{r \downarrow 0}\sup \{ \|\partial \{w>t\}\|(B^G_r(p)) \colon  w \in K, p \in M \}>0 \right\}
    \]
for the set of bad level sets, namely those for which the corresponding level sets do not satisfy the no concentration of mass on orbits condition. Then $\mathcal{B}(K)$ has zero (Lebesgue) measure in $[-1,1]$.
\end{lemma}

\begin{proof}
Let $w \in K$. It follows from \cite[Section 5.5]{EvansGariepy} that $\mathcal{G}(w):=\{ t \in [-1,1] \colon \{w>t\} \in \mathcal{C}(M)\}$ has full Lebesgue measure. For each $t \in \mathcal{G}(w)$, it follows that the total variation measure $\|\partial\{w>t\}\|$ is a Radon measure given by the restriction of the Hausdorff measure $\mathcal{H}^n$ to the reduced boundary of the superlevel set $\{w>t\}$. This ensures that, for every $p \in M$, we have:
    \[
    \|\partial \{w>t\}\|(G\cdot p)= \lim_{r \downarrow 0}\|\partial \{w>t\}\|(B^G_r(p)),
    \]
for $\|\partial \{w>t\}\|$ is an outer regular measure. Additionally, $\mathcal{H}^n(G\cdot p)=0$ for all $p\in M$, as the $G$ action has cohomogeneity $\geq 2$, and hence $\|\partial \{w>t\}\|(G\cdot p)=0$. By the compactness of $M$, we conclude that
    \[
    \sup_{p \in M} \|\partial\{w>t\}\|(B^G_r(p)) \to 0, \qquad \text{as} \ r \downarrow 0.
    \]

Now let $\mathcal{G} = \bigcap_{w \in K}\mathcal{G}(w)$. Since $K$ is finite, this set has full Lebesgue measure in $[-1,1]$, so it suffices to show that $\mathcal{G} \subset [-1,1] \setminus \mathcal{B}(K)$. Let $t \in \mathcal{G}$. The argument above shows that for every $\epsilon>0$ and every $w \in K$, there exists a $r(w)>0$ such that
    \[
        \sup_{p \in M} \|\partial\{w>t\}\|(B^G_r(p)) < \epsilon, \qquad \forall r \in (0,r(w)).
    \]
Therefore, $r_0 = \min_{w\in K}r(w)>0$ satisfies
    \[
    \sup\{ \|\partial\{w>t\}\|(B^G_r(p)) \colon w \in K, p \in M\} < \epsilon, \qquad \forall r \in (0,r_0).
    \]
This shows that
    \[
        \lim_{r \downarrow 0}\sup\{ \|\partial\{w>t\}\|(B^G_r(p)) \colon w \in K, p \in M\}=0,
    \]
that is $t \notin \mathcal{B}(K)$, and concludes the proof.
\end{proof}

We can now adapt the construction from \cite{GasparGuaracoClosed} to the equivariant setting. We point out the main changes in the proof:
    \begin{itemize}
        \item The discrete-to-continuous interpolation theorem \cite[Theorem 6.4]{GasparGuaracoClosed} extracted from \cite[Theorem 14.1]{MarquesNevesWillmore} can be replaced by its equivariant counterpart, \cite[Theorem 3]{WangJGA}.
        \item The discussion in \cite[Section 6.5]{GasparGuaracoClosed} can be directly translated to $G$-invariant cycles based on Wang's extension of Almgren's isomorphism for relative $G$-cycles, \cite[Theorems 3.1 and 3.3]{WangDensity}.
        \item The criterion given by \cite[Proposition 6.9]{GasparGuaracoClosed} that describes whether an even map $\tilde X\to \mathcal{Z}_n(M,\Z/2)$ defined on a $\tilde X \in \mathcal{C}_p$ is a $p$-sweepout remains true for maps taking values in the space of $G$-invariant cycles. In particular, the proof of Lemma 6.8 in \cite{GasparGuaracoClosed} on the local path connectedness of $\mathcal{Z}_{n}(M;\Z/2)$ carries over to $\mathcal{Z}_{n}(M;\Z/2)$, replacing the homotopy construction \cite[Theorem 8.2]{Almgren} by \cite[Theorem 3.2]{WangDensity}.
        \item The interpolation result \cite[Theorem 6.12]{GasparGuaracoClosed} can be replaced by the equivariant interpolation constructions underlying the proof of \cite[Theorem 4.11]{WangFB}, which translates the results from \cite[Section 13]{MarquesNevesWillmore} to the equivariant setting, using a combinatorial argument with $G$-tubes instead of geodesic balls.
    \end{itemize}

We use the latter interpolation result for $\phi_0$ inductively on each cell of $X[j]$ for sufficiently large $j$, as in \cite[Section 6.8]{GasparGuaracoClosed}. This allows us to conclude that, for all small $\delta_1>0$, there exists $\e^*=\e^*(\delta_1)>0$ with the following property: if $\e \in (0,\e^*)$, then we can find a discrete map $\phi \colon X[j+k]_0 \to \mathcal{Z}_n^G(M,\Z_2)$ such that
    \begin{enumerate}
        \item[(1)] $\sup_{x,y \in X[j+k]_0} \mathcal{F}( \phi(x),  \phi(y)) \leq \xi\left( \sup_{x,y \in X[j]_0} \mathcal{F}(\phi_0(x),\phi_0(y) \right)$, where $\xi$ is a positive function with $\xi(s) \to 0$ as $s \to 0^+$.
        \item[(2)] $\sup_{x \in X[j+k]_0}\mathbf{M}(\phi(x)) \leq \frac{\omega_{p,\e}^G(M,g) + \e}{4\tilde\sigma}+p\cdot \delta_1$.
        \item[(3)] $\phi$ lifts to a map $\tilde \phi\colon \tilde X[j+k]_0 \to \mathcal{C}_G(M)$, namely $\phi(\pi(x)) = \partial \tilde \phi(x)$ for all $x \in \tilde X[j+k]_0$.
        \item[(4)] $\phi$ has fineness $\leq C_p \cdot \delta_1$, with respect to the mass norm, for some $C_p>0$ depending only on $p$ (but not on $\delta_1$ nor $\e$).
    \end{enumerate}
Therefore, by letting $\delta_1>0$ be sufficiently small, we conclude that the Almgren $G$-extension $\Phi\colon X \to \mathcal{Z}_{n}^G(M; \mathbf{M};\Z_2)$ of $\phi$, \cite[Theorem 3]{WangJGA}, is a $(G,p)$-sweepout with 
    \[
        \mathbf{M}(\Phi(x)-\Phi(y)) \leq C(M,g,G,p)\cdot \delta_1
    \]
for any $x,y \in X$ in a common $p$-cell of $X[j+k]$, and consequently
    \[
        \omega_p^G(M,g) \leq \sup_{x \in X} \mathbf{M}(\Phi(x)) \leq \sup_{x \in X[j+k]_0}\mathbf{M}(\Phi(x))+C(p,M,g,G) \cdot \delta_1 \leq \frac{\omega_{p,\e}^G(M,g) + \e}{4\tilde\sigma}+(C(p,M,g,G)+p)\cdot \delta_1
    \]
By letting $\e\downarrow 0$, then $\delta_1 \downarrow 0$, and finally $\tilde \sigma \uparrow \sigma/2$, we conclude the proof of Proposition \ref{comparison bound 1}.

\subsection{From sweepouts by invariant cycles to sweepouts by functions}

We now turn to the reverse inequality between the equivariant phase transitions spectrum and equivariant volume spectrum, by adapting the techniques of \cite{dey2022comparison}. We will also consider the equivariant min-max values and minimal hypersurfaces studied  by T. Wang in \cite{WangJGA}. Concretely, we consider a  $G$-homotopy class $\Pi \in [X,\mathcal{Z}_n^G(M;\mathbf{F};\mathbb{Z}_2)]$ defined on a cubical complex $X$ with nontrivial, cubical double cover $\tilde X \to X$, for some cubical subcomplex $\tilde X$ of $Q(m,k)$. Denote by $\mathbf{L}(\Pi)$ the corresponding min-max value for the area functional, namely:
\[
    \mathbf{L}(\Pi) = \inf_{\Phi \in \Pi} \max_{x \in X} \mathbf{M}(\Phi(x))
\]

The main goal of this section is to show the following inequality between min-max values for $\mathbf{M}$ and the min-max values for the Allen-Cahn energy defined in \eqref{homotopy min max AC}:

\begin{proposition} \label{comparison bound 2}
With the notation above, for all $\tilde X \in \mathcal{C}_p$, we have
    \[
    \frac{1}{2\sigma} \limsup_{\e \downarrow 0} \mathbf{L}_\e(\tilde X; M,g; G) \leq \mathbf{L}(\Pi).
    \]
\end{proposition}

By combining Propositions \ref{comparison bound 1} and \ref{comparison bound 2}, we can prove Theorem \ref{comparison} exactly as in \cite[Section 3.5]{dey2022comparison}. We sketch the proof here for completeness.

\begin{proof}[Proof of Theorem \ref{comparison}]
By Proposition \ref{comparison bound 1}, it suffices to show that
    \[
        \frac{1}{2\sigma}\limsup_{\e \downarrow 0} \omega_{p,\e}^G(M,g) \leq \omega_p^G(M,g).
    \]
Let $\eta>0$. There exists a map $\Phi\colon X \to \mathcal{Z}_n(M;\mathbf{F};\Z_2)$ defined on a connected $p$-dimensional cubical complex $X$ such that
    \[
    \Phi^*\colon H^p(\mathcal{Z}_n^G(M;\mathcal{F};\Z_2);\Z_2)\simeq\Z_2 \to H^p(X,\Z_2) \ \text{is nontrivial} \quad \text{and} \quad  \sup_{x \in X}\mathbf{M}(\Phi(x)) \leq \omega_p^G(M,g) + \eta.
    \]
By arguing as in \cite[p. 1025]{dey2022comparison}, we see that $X$ has a nontrivial double cover $\pi \colon \tilde X \to X$ with $\tilde X \in \mathcal{C}_p$. Therefore, by \eqref{homotopy min max AC},
    \[
        \omega_{p,\e}^G(M,g) \leq \mathbf{L}_\e(\tilde X;M,g;G).
    \]
On the other hand, if $\Pi$ is the $\mathbf{F}$-homotopy class of $\Phi$, then
    \[
        \mathbf{L}(\Pi) \leq \sup_{x \in X} \mathbf{M}(\Phi(x)) \leq \omega_p^G(M,g) + \eta.
    \]
Combining the two previous inequalities and Proposition \ref{comparison bound 2}, we conclude
    \[
        \frac{1}{2\sigma}\limsup_{\e \downarrow 0} \omega_{p,\e}^G(M,g) \leq \frac{1}{2\sigma}\limsup_{\e\downarrow 0} \mathbf{L}_\e(\tilde X;M,g;G) \leq \mathbf{L}(\Pi) \leq \omega_p^G(M,g) + \eta.
    \]
Since $\eta>0$ is arbitrary, this concludes the proof.
\end{proof}

The proof of Proposition \ref{comparison bound 2} relies on the construction of almost smooth (and almost optimal) discrete $(G,p)$-sweepouts. This is described in the next Lemma, which is a $G$-equivariant adaptation of \cite[Proposition 3.6]{dey2022comparison}, and whose proof is described in the next section.

\begin{proposition} \label{almost smooth sweepout}
For each $\eta>0$, there exist $N,I \in \Z_{>0}$, a closed $G$-invariant set $\mathcal{S} \subset M$ and a $\Z_2$-equivariant map $\Psi \colon \tilde X[NI]_0 \to \mathcal{C}_G(M)$ with the following properties. Here we write
    \[\Psi(x) = \Psi(T(x)):=\tilde\Psi(x) \cap \tilde\Psi(T(x))\]
for each $x \in \tilde X[NI]_0$.
    \begin{enumerate}
        \item[(0)] Each $\tilde\Psi(x)$ is a closed subset of $M$ with $\tilde\Psi(x) \cup \tilde\Psi(T(x))=M$.
        \item[(1)] The set $\mathcal{S}$ is the union of finitely many smooth codimension $2$ submanifolds of $M$. In particular, there exist $C_{\mathcal{S}}>0$ and $\rho_0>0$ such that 
            \[\mathcal{H}^n\left(\{x \in M \colon \mathrm{dist}(x,\mathcal{S})\}=\rho\right)\leq C_{\mathcal{S}}\cdot \rho , \qquad \text{for all} \quad \rho \in (0,\rho_0).\]
        \item[(2)] For each $x \in \tilde X[NI]_0$, the set $\Psi(x)\setminus \mathcal{S}$ is the union of finitely many $G$-invariant, embedded hypersurfaces. Moreover, for every $p \in \Psi(x)\setminus \mathcal{S}$, there exists a $G$-tube $U$ such that $U\cap \Psi(x)$ is a smooth hypersurface and $U\setminus\Psi(x) = \mathcal{O}_1 \cup \mathcal{O}_2$, for open sets $\mathcal{O}_1\subset \tilde\Psi(x)$ and $\mathcal{O}_2\subset \tilde\Psi(T(x))$.
        \item[(3)] $\mathcal{H}^n(\Psi(x)) \leq \mathbf{L}(\Pi) + O(\eta)$, for all $x \in \tilde X[NI]_0$.
        \item[(4)] If $x,y$ lie in a common $1$-cell $\alpha \in \tilde X[N]$, then the pairs $(\tilde\Psi(x),\tilde\Psi(y))$ and $(\tilde\Psi(T(x)),\tilde\Psi(T(y)))$ coincide outside of a tube $B^G_{r^*}(p_\alpha)$ for which $\mathcal{H}^n((\Psi(x)\cup\Psi(y))\cap B^G_{r^*}(p_\alpha)) \leq O(\eta)$.
    \end{enumerate}
\end{proposition}

By composing the signed distance functions to the $G$-invariant, piecewise smooth hypersurfaces $\Psi(x)$, that is
    \[
        d_x\colon M \to \R, \qquad d_x(p) = \left\{ \begin{array}{cl} \mathrm{dist}(p,\Psi(x)), & p \in \tilde\Psi(x), \\ -\mathrm{dist}(p,\Psi(x)), & p \in \tilde\Psi(T(x)), \end{array} \right.
    \]
with a real function $\tilde{\mathbb{H}}_\e$ obtained by truncating the heteroclinic solution $\mathbb{H}_\e$ of the Allen-Cahn solution in $\R$ to a constant away from $[-\sqrt{\e},\sqrt{\e}]$, we obtain a $\Z_2$-equivariant map $\vartheta_\e \colon \tilde X[NI]_0 \to W^{1,2}_G(M)\setminus \{0\}$, $\vartheta_\e(x) = \tilde{\mathbb{H}}_\e\circ d_x$. For all $x \in \tilde X[NI]_0$, a standard computation via the coarea formula and normal exponential coordinates for $\Psi(x)\setminus \mathcal{S}$ shows that the energy $E_\e(\vartheta_\e)$ can be bounded from above by 
    \[2\sigma\cdot \sup_{\tau\in [-\sqrt\e, \sqrt\e]}\mathcal{H}^n(\{d_x = \tau\}) \leq 2\sigma\cdot  C_{\mathcal{S}}\cdot \sqrt{\e} + 2\sigma\cdot (1+C_1\sqrt\e)\sup_{x \in\tilde X[NI]} \mathcal{H}^n(\Psi(x)),\]
where the second inequality follows from parts (1) and (3) from Proposition \ref{almost smooth sweepout}. We refer to \cite[Section 3.4]{dey2022comparison} for the detailed computation, and collect the properties of $\vartheta_\e$ in the proposition below, obtained by letting $\e>0$ be sufficiently small:

\begin{proposition}
    There exist $\e^*>0$ with $\e^* < K \eta$ and $C_0>0$ depending on $(M,g)$, $G$, and also on the map $\Psi$ given by Proposition \ref{almost smooth sweepout} with the following property. For each $\e \in (0,\e^*)$, there exists a $\Z_2$-equivariant map $\vartheta_\e \colon \tilde X[NI]_0 \to W^{1,2}_G(M)\setminus\{0\}$ with
        \[E_\e(\vartheta_\e(x)) \leq 2\sigma\cdot (\mathbf{L}(\Pi) + C_0\eta)\]
\end{proposition}

\noindent The proof of Proposition \ref{comparison bound 2} then follows from the interpolation argument in \cite[Proposition 3.11]{dey2022comparison}. It consists of defining a continuous, $\Z_2$-equivariant map $\zeta_\e \colon \tilde X \to W^{1,2}(M)\setminus\{0\}$ inductively over the $j$-dimensional cells of $\tilde X[NI]_0$, for $j=1,\ldots, \dim\tilde X$ using the continuous deformation map
    \begin{align*}
        \theta&\colon W^{1,2}(M) \times W^{1,2}(M) \times W^{1,2}(M) \to W^{1,2}(M), \\
        \theta(u,v,w) &= \min\{\max\{u_0,-w\},\max\{u_1,w\}\}^+ + \max\{\min\{u_0,w\},\min\{u_1,-w\}\}^-
    \end{align*}
The main properties of $\theta$ are described in \cite[Proposition 3.10]{dey2022comparison}. We note here that $\theta(u,v,w) \in W^{1,2}_G(M)$ provide $u,v,w \in W^{1,2}_G(M)$, and the energy of $\theta(u,v,w)$ can be bounded from above by:
    \begin{equation} \label{energy theta}
        E^{(U,g)}_\e(\theta(u,v,w)) \leq E_\e^{(U,g)}(u) + E_\e^{(U,g)}(v) + E_\e^{(U,g)}(w)
    \end{equation}
for any open set $U \subset M$. 

To illustrate this interpolation argument, we observe that, in each $1$-cell $\alpha =\{v_t \colon t \in [0,1]\}$ of $\tilde X[NI]_1$ with vertices $v_0,v_1$, the functions $\zeta_\e$ have the form
    \[\zeta_\e(v_t) = \theta(\vartheta_\e(v_0),\vartheta_\e(v_1),w_\e(t)),\]
for $w_\e(t) = \tilde{\mathbb{H}}_\e\circ d_t$ and a signed distance function $d_t$ to the level sets $f^{-1}(t)$ of a Morse function $f\colon M \to [1/3,2/3]$ with no local (non-global) extrema, interpolated linearly to $\pm 1$ over $[0,1/3]\cup[2/3,1]$. Again, we can choose $f$ to be a $G$-equivariant Morse function \cite{Wasserman}, and the rest of the proof from \cite[Proposition 3.11]{dey2022comparison} remains unchanged. In particular, we can assume that the level sets of $f$ have area bounded by $O(\eta)$ in the $G$-tubes given by property (4) in Proposition \ref{almost smooth sweepout}, see Proposition 3.8 in \cite{dey2022comparison} and Lemma \ref{sweep_annulus} below. Therefore, for each $p$-cell $\alpha$ of $\tilde X[NI]$ and each $x \in \alpha$, there exists $v \in \tilde X[NI]_0$ and a $G$-tube $B_{r^*}^G(p_\alpha)$ such that the energy of $\zeta_\e(x)$ is bounded from above by $O(\eta)$ in $B_{r^*}^G(p_\alpha)$, and hence
    \[E_\e(\zeta_\e(x)) \leq E_\e(\vartheta_\e(v)) + E_\e^{(B^G_{r^*}(p_\alpha),g)}(\zeta_\e(x)) \leq 2\sigma\cdot(\mathbf{L}(\Pi)+C\eta),\]

using the fact that $\Psi(v)$ and $f^{-1}(t)$ have area $\leq O(\eta)$ in each of these geodesic tubes $B_{r^*}^G(p_\alpha)$ and \eqref{energy theta} to bound $E_\e^{(B^G_{r^*}(p_\alpha),g)}(\vartheta_\e(x))$. This can be ensured by the construction of the almost sweepout, as we describe in the next subsection, and finishes the proof of Proposition \ref{comparison bound 2}, as it implies that
    \[
    \frac{1}{2\sigma}\mathbf{L}_\e(\tilde X;M,g;G) \leq \frac{1}{2\sigma} \sup_{x \in \tilde X} E_\e(\zeta_\e(x))\leq \mathbf{L}(\Pi) + C\eta, \qquad \text{for all} \ \eta>0.
    \]

\subsection{Almost smooth sweepouts by $G$-tubes}

In this section, we briefly describe the construction of the almost smooth sweepout described in Proposition \ref{almost smooth sweepout} in the context of $G$-invariant functions and sets, following \cite[Section 3.3]{dey2022comparison} closely. The central idea is to consider cover $M$ by finitely many sufficiently thin, pairwise transverse, $G$-invariant tubes $\{B_{r^*}^G(p_i)\}_{i=1}^I$ with smooth boundaries, in which the areas of an almost optimal $(G,p)$-sweepout $\Phi\colon X \to \mathcal{Z}^G_n(M;\Z_2)$ and the volume of its lift $\tilde \Phi\colon \tilde X \to \mathcal{C}_G(M)$ remain controlled, as well as the areas of the level sets of the equivariant Morse function $f \colon M \to [1/3,2/3]$. Next, we approximate each $\tilde\Phi(x)$, for each vertex $x$ of a fine subdivision $\tilde X[N]$ of $\tilde X$, by smooth $G$-invariant sets $\{\tilde\Phi_j(x)\}_{j \in \N}$, using Lemma \ref{smooth_approx}. The desired almost-smooth sweepout can be roughly described as obtained by iteratively interpolating between $\tilde\Phi_j(x)$ and $\tilde\Phi_j(x')$, for adjacent vertices in $\tilde X[N]_0$, by pieces of the $G$-invariant tubes $B_{r^*}^G(p_i)$ over each cell in $\tilde X[N]$.

As a general remark, since all the building blocks in this construction are $G$-invariant, we can adapt Dey's construction to the equivariant setting following the same underlying principles. In the rest of this section we describe some of the central steps, in order to emphasize that the entire process preserves the $G$-invariance. \medskip

For each $\eta>0$, it follows from the equivariant interpolation results of \cite[Section 4.1]{WangJGA} that there exists a continuous map $\Phi\colon X \to\mathcal{Z}^G_n(M;\mathbf{M};\mathbb{Z}_2)$ with $\Phi \in \Pi$ such that
\[
\sup_{x \in X} \mathbf{M}(\Phi(x)) < \mathbf{L}(\Pi)+\eta 
\]
Then, by \cite[Theorem 9]{WangJGA} (see also \cite[Theorem 3.1]{WangDensity}), there exists a $\mathbb{Z}_2$-equivariant lift $\tilde\Phi \colon X \to \mathcal{C}_G(M)$ of $\Phi$. This means that: if $T\colon \tilde X \to \tilde X$ is the $\Z_2$ action on $\tilde X$ (by the deck transformation of the double cover $\pi\colon \tilde X \to X$), then
\[
[\![\tilde \Phi(T(x))]\!] + [\![\tilde \Phi(x)]\!] = [\![M]\!], \qquad \text{for all} \ x \in \tilde X.
\]\smallskip

\noindent\textbf{Preliminary constructions.} Let $D \subset M^{reg}/G$ be a countable set which is dense in $M^{reg}/G$, and denote by $\Theta = \{ \mathrm{dist}(G\cdot p, M\setminus M^{reg})\ \colon \ G\cdot p \in D\}$ the set of all distances between the orbits in $D$ and the union of all nonprincipal orbits. We consider the countable collection of $G$-tubes
\[
\mathcal{B}=\{B^G_t(p) \colon G\cdot p \in D, \ t \in (0,\mathrm{inj}(M,g))\cap \mathbb{Q}\setminus \Theta\}
\]
Note that the last condition on $t$ ensures that $\partial B \subset M^{reg}$ for any $B \in \mathcal{B}$.  In particular, each such $\partial B$ is a smooth $G$-equivariant submanifold, with smooth quotient $\partial B / G$ of dimension $\ell$.

As in \cite{dey2022comparison}, the construction of an almost smooth $p$-sweepout uses a $1$-sweepout by the level sets of a suitable Morse function. Since the space of equivariant Morse functions is generic among equivariant smooth functions, see e.g. \cite{Wasserman}, there exists a $G$-equivariant Morse function $f \colon M \to [1/3,2/3]$ such that the restriction of $f$ to each $B\in \mathcal{B}$ and each corresponding $\partial B$ is a $G$-equivariant Morse function.  We can also assume that the equivariant Morse function $f$ above has no orbits that are local but not global maxima or minima of $f$. This is achieved by a local smooth perturbation around critical orbits, relying on the equivariant Morse Lemma \cite[Lemma 4.1]{Wasserman}. We refer to \cite[p.1001]{dey2022comparison} for further details.

Note that $\tilde \Phi$ and the map $t \in [1/3,2/3] \mapsto \partial [\![ f_B > t]\!] \in \mathcal{C}^G(M)$ have the \emph{no concentration of mass on orbits} property, by \cite[Lemma 3.5]{WangDensity} and \cite[Lemma 14]{WangJGA}) (see also Proposition 2.1 in \cite{LiuCVPDE}). Hence, there exists $r_0 \in (0,\mathrm{inj}(M,g)) \cap \mathbb Q$  with $t_0 \notin \Theta$ such that
\begin{equation} \label{small mass in tubes}
\|\Phi(x)\| (\overline B_{r_0}^G(p))< \eta\qquad \text{and} \qquad \mathcal{H}^n(f^{-1}(t)\cap \overline B_{r_0}^G(p)) < \eta     
\end{equation}
for all $p \in M$, all $x \in X$ and all $t \in [1/3,2/3]$. In addition, for all $p \in M$, we have $B_{r_0/8}^G(p) \cap M^{reg} \neq\emptyset$, since $M\setminus M^{reg}$ has empty interior. Hence, by the density of $D$ in $ M^{reg}/G$, there is some orbit $G\cdot p_i$ in $D$ contained in $B_{r_0}^G(p)$, which then implies $G\cdot p \subset B_{r_0/4}^G(p_i)$. This shows that there exists $I \in \Z_{>0}$ and a finite subset $\{G\cdot p_i\}_{i=1}^I \subset D$ such that $\{B_{r_0/4}^G(p_i)\}_{i=1}^I$ covers $M$. We let $\overline{\mathbf{B}}_i^0 = B_{r_0}^G(p_i)$.

\begin{lemma} \label{sweep_annulus}
There exist $r_1 \in(r_0/2, 3r_0/4) \cap \mathbb{Q}\setminus \Theta$ and $\delta\in (0,r_0/8)$ such that
    \[
    \|\Phi(x)\|\left(\overline{\mathrm{An}}^G(p_i,r_1-2\delta,r_1+\delta)\right) < \frac{\eta}{I} \qquad \text{and} \qquad \mathcal{H}^n\left(f^{-1}(t)\cap \overline{\mathrm{An}}^G(p_i,r_1-2\delta,r_1+\delta)\right)<\frac{\eta}{I}
    \]
for all $x \in X$, $t \in [1/3,2/3]$ and all $i=1,\ldots, I$.
\end{lemma}

\begin{proof}
The proof follows from the same arguments in \cite[Lemma 3.2 and Lemma 3.3]{dey2022comparison}, by observing that the Radon measure valued maps  $x \in X \mapsto \|\Phi(x)\|$ and $t \in [1/3,2/3] \mapsto \|\partial [\![ f>t]\!] \|=\mathcal{H}^n \lfloor  f^{-1}(t)$ are continuous. 
\end{proof}
\noindent In order to prove Proposition \ref{almost smooth sweepout}, we will first replace the continuous map $\Phi$ by a discrete map into $\mathcal{Z}_n^G(M,\Z_2)$ that is sufficiently fine with respect to the mass norm. By the continuity of $\Phi$ with respect to $\mathbf{M}$ and the combinatorial structure of the complex $X$, we can find a sufficiently fine cubical subdivision of $X$, namely $X[N]$ for sufficiently large $N \in \Z_{>0}$, such that
\begin{itemize}
    \item The double cover $\tilde X$ inherits a cell complex structure $\tilde X[N]$ such that the restriction of the quotient map $\pi \colon \tilde X \to X$ to each cell of $\tilde X[N]$ is a homeomorphism onto its image.
    \item If $x,y$ are vertices in the same cell of $X[N]$, then $\mathbf{M}(\Phi(x)-\Phi(y)) < \eta$.
    \item If $x,y$ are vertices in the same cell of $\tilde X[N]$, then $\mathbf{M}([\![\tilde\Phi(x)]\!]-[\![\tilde\Phi(y)]\!]) = \mathcal{H}^{n+1}(\tilde\Phi(x) \Delta \tilde\Phi(y)) < \delta\eta/I$, where $\delta>0$ is given by the previous lemma.

\end{itemize}
We may then order the cells in $\tilde X[N]$ as $\{e_q,f_q \colon q =1,\ldots Q\}$, with nondecreasing dimension with respect to $q$, and such that $T(e_q) = f_q$ and $\pi|_{e_q}$ and $\pi|_{f_q}$ are homeomorphisms onto their images.\bigskip

\noindent\textbf{Approximation by smooth sets.} By applying Lemma \ref{smooth_approx} to each pair $\{\tilde\Phi(x), \tilde\Phi(T(x))\}$ with $x \in \tilde X[N]_0$, we obtain:

\begin{lemma} \label{smooth sweep}
There exists a sequence of discrete maps $\tilde\Phi_j\colon X[N]_0 \to \mathcal{C}_G(M)$, $j \in \Z_{>0}$ with the following properties: for all $x \in X$,
\begin{enumerate}
    \item[(i)] $\tilde \Phi_j(x)$ are closed subsets of $M$ such that  $[\![\tilde\Phi_j(x)]\!] + [\![\tilde\Phi_j(T(x))]\!] = [\![M]\!]$ and  $\tilde \Phi_j(x) \cup \tilde\Phi_j(T(x))=M$, for all $j$.
    \item[(ii)] If $\Phi_j(\pi(x)):=\tilde\Phi_j(x) \cap \tilde\Phi(T(x))$, then each $\Phi_j(\pi(x))$ is a smooth hypersurface in $M$ which coincides with the topological and the reduced boundary of both $\tilde\Phi_j(x)$ and $\tilde\Phi_j(T(x))$.
    \item[(iii)] $\|\mathbf{1}_{\Phi_j(x)} - \mathbf{1}_{\tilde\Phi(x)}\|_{L^1(M)} \to 0$ and $\mathbf{F}\left([\![\Phi_j(\pi(x))]\!], \Phi(\pi(x))\right) \to 0$.
    \item[(iv)] For all $j$ and all $p \in \Phi_j(x)$, there exists a $G$-invariant tube $U$ around $G\cdot p$ such that $U \setminus \Phi_j(x)$ is the union of two disjoint, connected, $G$-invariant open sets $\mathcal{O}_1 \subset \tilde\Phi_j(x)$ and $\mathcal{O}_2\subset \tilde\Phi_j(T(x))$.
    \item[(v)] If $r_1>0$ is as in Lemma \ref{sweep_annulus} then, for all $i=1,\ldots, I$,
        \[
        \mathcal{H}^n\left(\partial B^G_t(p_i) \cap(\tilde\Phi_j(x)\Delta \tilde\Phi(x))\right) \to 0, \qquad \text{for almost every} \ t \in (0,r_1).
        \]
\end{enumerate}
\end{lemma}

By induction, for each pair of cells $\{e_q,f_q\}$ in $\tilde X[N]$ (where $q =1,\ldots, Q$) and each orbit $G\cdot p_i$, where $i=1,\ldots I$, we can choose a radius $r_i(q) \in (r_1-\delta,r_1)\setminus \Theta$ such that the boundary of the tubes $\partial B_{r_i(q)}^G(p_i)$ are smooth, $G$-invariant, compact hypersurfaces with 

\begin{enumerate}
    \item[(1)] $\|\Phi(\pi(x))\|(\partial B^G_{r_i(q)}(p_i))=0$.
    \item[(2)] $\partial B_{r_i(q)}^G(p_i)$ is transverse to the smooth hypersurfaces: $\Phi_j(x)$, for all $j \in \N$ and all $x \in (e_q)_0\cup (f_q)_0$, to $\partial B^G_{r_s(q)}(p_s)$ for all $1\leq s \leq i-1$, and to $\partial B^G_{r_k({q'})}(p_k)$ for all previous cells $1\leq q'\leq q-1$ and all orbits $1 \leq k \leq I$.
    \item[(3)] For all vertices $x \in (e_q)_0 \cup (f_q)_0$ (using part (v) in Proposition \ref{smooth sweep}),
        \[
        \lim_j \mathcal{H}^n\left(\partial B_{r_i(q)}^G(p_i) \cap (\tilde \Phi_j(x) \Delta \tilde\Phi(x)) \right) = 0
        \]
    \item[(4)] If $x$ and $x'$ are both vertices in the same cell $c=e_q$ or $c=f_q$ in $\tilde X[N]$, then
        \[
            \mathcal{H}^n\left(\partial B_{r_i(q)}^G(p_i) \cap (\tilde \Phi(x) \Delta \tilde\Phi(x')) \right) \leq \frac{2^{2 \dim c}\eta}{I}
        \]
\end{enumerate}
We note that the last property can be achieved by the argument used to find $r_1$ in the proof of Lemma \ref{sweep_annulus}, together with the fineness bounds $\mathcal{H}^{n+1}(\tilde\Phi(x) \Delta \tilde\Phi(x')) < \eta$, from the choice of the subdivision $\tilde X[N]$ of $\tilde X$.

Finally, we can pick $r^*>0$ with $\max_{i,q} r_i(q) < r^* < r_1$ such that each boundary $\partial B_{r^*}^G(p_i)$ is transverse to the smooth hypersurfaces $\Phi_j(x)$, for all large $j \in \N$ and all $x \in \tilde X[N]_0$, and such that
    \[
    \|\Phi(\pi(x))\|(\partial B_{r^*}^G(p_i)) = 0, \quad \text{for all}  \ i=1,\ldots, I \ \text{and} \ x \in \tilde X[N]_0.
    \]
With those choices, we can verify that each collection of $G$-tubes $\{B_{r_i(q)}^G(p_i)\}_{i=1}^I$, for $q=1,\ldots, Q$, and the collection $\mathcal{B}^* :=\{B_{r^*}^G(p_i)\}_{i=1}^I$ cover $M$. As in \cite[p. 1005]{dey2022comparison}, we consider the collection $\mathscr{R}$ of the $G$-invariant open subsets of $M$ obtained by finite intersections and finite unions of the $G$-invariant sets in $\mathcal{B}^* \cup \{M \setminus \overline B \colon B \in \mathcal{B}^*\}$. The choice of the radius $r^*$ ensures that
\[
\|\Phi(\pi(x))\|(\partial R) \leq \sum_{i=1}^I \|\Phi(\pi(x))\|(\partial B_{r^*}^G(p_i)) = 0, \quad \text{for all} \  x \in \tilde X[N]_0 \ \text{and all} \ R \in \mathscr{R}. 
\]
Therefore, by the $\mathbf{F}$-convergence given by Lemma \ref{smooth sweep} part (iii), the equivalence between $\mathbf{F}$-convergence and weak convergence and finally \cite[Thm 1.40]{EvansGariepy} (adapted to manifolds), we have
    \begin{equation} \label{F-convergence}
    \|\Phi_j(x)\|(R) \to \|\Phi(\pi(x))\|(R) \quad \text{for all} \ x\in \tilde X[N]_0 \ \text{and} \ R \in \mathscr{R}.
    \end{equation}
\noindent The following Proposition summarizes some useful properties of the approximations $\tilde \Phi_j$ and provides the suitable smooth approximation for the discretized $(G,p)$-sweepout $\Phi$ on $\tilde X[N]_0$. Its content and its proof follows closely \cite[Prop 3.5]{dey2022comparison}

\begin{proposition}
For sufficiently large $j \in \N$, the smooth, $G$-invariant hypersurfaces $\{\tilde \Phi_j(x) \colon x \in \tilde X[N]_0\}$ satisfy:
    \begin{enumerate}
        \item[(i)] $\left| \|\Phi_j(x)\|(R) - \|\Phi(\pi(x))\|(R) \right| < \eta$, for all $R \in \mathscr{R}$.
        \item[(ii)] $\mathcal{H}^n(\Phi_j(x)) < \mathbf{L}(\Pi) + 2\eta$ 
        \item[(iii)] $\mathcal{H}^n(\Phi_j(x) \cap \overline B_{r_0}^G(p_i)) < \eta$ , for all $i =1,\ldots I$.
        \item[(iv)] $\mathcal{H}^n\left(\Phi_j(x) \cap \bigcup_{i=1}^I \overline{\mathrm{An}}^G(p_i,r_1-2\delta,r_1+\delta)\right) < \eta$
        \item[(v)] If $x$ and $x'$ are both vertices in the same cell $c=e_q$ or $c=f_q$ in $\tilde X[N]$, then
        \[
            \mathcal{H}^n\left(\partial B_{r_i(q)}^G(p_i) \cap (\tilde \Phi_j(x) \Delta \tilde\Phi_(x')) \right) \leq \frac{2^{2 \dim c}\eta}{I}
        \]
    \end{enumerate}
\end{proposition}

\begin{proof}
Conclusion (i) follows readily from \eqref{F-convergence}. Using (i), $M =\bigcup \mathcal{B} \in \mathscr{R}$ and the upper bound for the mass of $\Phi$, we get
    \[\mathcal{H}^n(\Phi_j(x)) = \|\Phi_j(x)\|(M) < \|\Phi(\pi(x))\|(M) + \eta =\mathbf{M}(\Phi(\pi(x)) + \eta <  \mathbf{L}(\Pi) + 2\eta,\]
thus proving (ii). In order to show (iii), we note that
    \begin{align*}
    \limsup_{j \to \infty} \mathcal{H}^n(\Phi_j(x) \cap \overline B_{r_0}^G(p_i)) & = \limsup_{j \to \infty} \|\Phi_j(x)\|(\overline B_{r_0}^G(p_i)) \\
    & \leq \limsup_{j \to \infty}\left( \|\Phi(\pi(x))\| (\overline B_{r_0}^G(p_i)) + \mathbf{F}([\![ \Phi_j(x) ]\!], \Phi(\pi(x))) \right) < \eta
    \end{align*}

where we used Lemma \ref{smooth sweep}, part (iii), and \eqref{small mass in tubes}. A similar computation, using Lemma \ref{sweep_annulus}, proves (iv). Finally, (v) follows from properties (3) and (4) in the inductive choice of the radii $r_i(q)$.
\end{proof}
\noindent From now on, we fix a large $j \in \N$, as given by the previous proposition. Consider the set of all intersections between the surfaces $\Phi_j(x)$ and the boundaries of the $G$-tubes associated to the cells of $\tilde X[N]$ constructed above:
    \begin{align*}
        \mathscr{S} = &\left\{\Phi_j(x) \cap \partial B_{r_i(q)}^G(p_i) \ \colon \ q=1,\ldots, Q, \ i=1,\ldots, I, \ x \in \tilde X[N]_0\right\} \\
            &\qquad  \cup \ \left\{ \partial B^G_{r_i(q)}(p_i) \cap \partial B^G_{r_{i'}(q')}(p_{i'}) \ \colon \ q,q'=1,\ldots, Q, \ i,i'=1,\ldots, I \right\}
    \end{align*}
By the choice of the radii $r_i(q)$, it follows that each element of $\mathscr{S}$ is a closed, smooth, $G$-invariant, codimension 2 submanifold of $M$. Therefore, if $\mathcal{S} = \bigcup_{\Gamma \in \mathscr{S}}\Gamma$, then 
    \[
        \mathcal{H}^n\left( \{x \in M \colon \mathrm{dist}(x,\mathcal{S})=\rho\} \right) \leq C\rho
    \]
for all sufficiently small $\rho$, where $C=C_\mathcal{S}>0$. 

The final technical tool necessary in the proof of Proposition \ref{almost smooth sweepout} is the map between Caccioppoli sets:
    \[\Omega(K,K'; B) := (K\cap \overline{M \setminus B}) \cup (K' \cap \overline B),\]
for $K,K'\in \mathcal{C}(M)$ and a $G$-invariant, normal tube $B$ with finite perimeter. Intuitively, this map replaces the part of $K$ inside $B$ by the corresponding piece of $K'$. It follows from its definition that $\Omega(K,K';B')$ is \emph{closed} (respectively, $G$-invariant) provided the sets $K, K'$ are also closed (respectively, $G$-invariant).

The proof of Proposition \ref{almost smooth sweepout} now follows precisely as in the Proof of 3.6, pages 1009-1006 in \cite{dey2022comparison}, taking into account that all the sets that appear in the construction are $G$-invariant. 

\subsection{Computing equivariant widths}

In this section, we compute $\omega_p^G(M, g)$ for specific manifolds $(M, g)$ and cohomogeneity $2$ group actions, $G$. We rely on the computations of \cite{ChodoshMantoulidisWidths, marx2025p, marx2026p} who demonstrated that 
\begin{theorem}[\cite{ChodoshMantoulidisWidths}] \label{thm.CM.sphere}
For $(S^2, g_{round})$ the standard metric on the sphere, we have 
\[
\omega_p(S^2, g_{round}) = 2 \pi \lfloor \sqrt{p} \rfloor
\]
\end{theorem}
\begin{theorem}[\cite{marx2025p}] \label{thm.MK.rptwo}
For $(\R \P^2, g_{round})$ the standard metric on the real projective plane, we have 
\[
\omega_p(\R \P^2, g_{round}) = 2 \pi \left\lfloor \frac{1}{4} \left(1 + \sqrt{1 + 8p} \right) \right\rfloor
\]
\end{theorem}
\begin{theorem}[\cite{marx2026p}] \label{thm.MK.hemisphere}
For $(S^2_+, g_{round})$ the round metric on the hemisphere, we have 
\[
\omega_p(S^2_+, g_{round}) =  \pi \left\lfloor \frac{1}{2} \left(-1 + \sqrt{1 + 8p} \right) \right\rfloor
\]
\end{theorem}

\begin{example}
    Consider the classical Hopf fibration $S^{1}\curvearrowright S^{3}$, and let $\pi: S^3 \rightarrow S^2(1 / 2)$ be the map defined as 
\[
\pi\left(z_1, z_2\right)=\left(z_1 \bar{z}_2, \frac{\left|z_1\right|^2-\left|z_2\right|^2}{2}\right)
\]
onto the $2$ -sphere of radius $1/2$. Then $\pi$ is a Riemannian submersion with respect to the standard round metric and its fibers are closed geodesics in $(S^3,g_{round})$ of length $2 \pi$.

Since all the fibers have the same volume, by Corollary \ref{const.vol}, we conclude that 
\[\omega_{p,AC}^{S^{1}}(S^{3},g_{S^{3}})=\omega_{p,AC}(S^{2}(1/2),\tilde{g}_{S^{2}(1/2)}),\]
where the conformal metric $\tilde{g}_{S^{2}(1/2)}= \mathcal{V}^{2/\ell}\cdot g_{S^{2}(1/2)}=(2\pi)^2 g_{S^2(1/2)}$. Consequently, using Theorem \ref{thm.CM.sphere} we obtain that
\[\omega_{p,AC}^{S^{1}}(S^{3},g_{S^{3}}) = \omega_{p,AC}(S^{2}(1/2),(2\pi)^{2}g_{S^{2}(1/2)}) = 2\pi \omega_{p,AC}(S^{2}(1/2),g_{S^{2}(1/2)}) = 2 \pi^2 \lfloor\sqrt{p}\rfloor .\]
\end{example}

\begin{example}
Consider the action $O(2) \curvearrowright SO(3)$, where $SO(3)$ is endowed with its bi-invariant metric. Then the standard projection
$\pi: SO(3) \to SO(3)/O(2) \cong \mathbb{RP}^2,$
given by
$$\pi(A) = [A e_3],$$
where $[A e_3]$ denotes the unoriented line determined by the third column of $A$, is a Riemannian submersion. Here 
$SO(3)/O(2)$ is equipped with the induced homogeneous quotient metric.
Moreover, all fibers of this map have constant volume $4\pi$.

Similarly to the previous example, applying Corollary \ref{const.vol}, we have that 
\[\omega_{p,AC}^{O(2)}(SO(3),g_{SO(3)})=\omega_{p,AC}(\mathbb{RP}^{2},\tilde{g}_{\mathbb{RP}^{2}}),\]
where $\tilde{g}_{\mathbb{RP}^{2}} = (4\pi)^{2}g_{\mathbb{RP}^{2}}$. Therefore, by Theorem \ref{thm.MK.rptwo}, we conclude that for every $p \in \N^{+}$,
\[\omega_{p,AC}^{O(2)}(SO(3),g_{SO(3)})= \omega_{p,AC}(\mathbb{RP}^{2},(4\pi)^{2}g_{\mathbb{RP}^{2}})= 4\pi\omega_{p,AC}(\mathbb{RP}^{2},g_{\mathbb{RP}^{2}}) = 8\pi^{2} \cdot\left\lfloor\frac{1}{4}(1+\sqrt{1+8 p})\right\rfloor.\]
\end{example}
\begin{example}
Any trivial principal $G$-bundle over $S^2$, $S^{2}_{+}$, or $\R \P^2$ will have $G$-equivariant widths equal to the $p$-widths of the corresponding surface times the volume of $G$.
\end{example}

\begin{remark}
We observe that the examples of $3$-dimensional $S^1$-bundles and $O(2)$-bundles over the orbit spaces $B=S^2$ or $B=\R\P^2$ described above can be characterized more broadly as Seifert fibrations $\pi\colon (M^3,g) \to (B,g_{round})$ endowed with a Riemannian metric $g$ for which $\pi$ is a submersion onto those surfaces and such that the fibers $\pi^{-1}(b)$ have constant length $=\ell$. Assuming $g_{round}$ has constant curvature $=1$ and that $S^1$ acts on $M^3$ isometrically on the fibers, we have (here $G=S^1$ or $G=O(2) = S^1 \dot \cup S^1$)
    \[\omega_p^{G}(M,g) = \ell \cdot \omega_p(B,g_0) = \left \{ \begin{array}{cl} 2\pi\ell \cdot \left\lfloor \sqrt p\right\rfloor,  & \text{if} \ B=S^2, \\ 2 \pi \ell \cdot \left\lfloor \frac{1}{4} \left(1 + \sqrt{1 + 8p} \right) \right\rfloor,& \text{if} \ B=\R\P^2.   \end{array} \right. \]
We refer to \cite{orlik2006seifert} and \cite{scott1983geometries} for the general classification of these spaces.
\end{remark}

%% file: Regularity.tex
\section{Regularity} \label{sec.regularity}
The goal of this section is to prove Theorem \ref{thm.AC.regularity} concerning the regularity of $G$-equivariant solutions to \eqref{ACEquation} with bounded index and energy 
\begin{theorem*} 
Let $\{u_{\eps_i}\} \subseteq W^{1,2}_G(M, g)$ be a sequence of Allen--Cahn solutions with $\eps_i \to 0^+$, and $E_{\eps_i}(u_{\eps_i}) \leq \Lambda$ and $\mathrm{Ind}_{\eps_i}^G(u_{\eps_i}) \leq N$ for fixed $\Lambda, N > 0$. 
\begin{itemize}
\item If $\cohom(G) \geq 3$, then up to subsequence, the corresponding varifolds $V(u_{\eps_i})$ converge (in varifold sense) to a disjoint union of embedded, minimal, $G$-invariant hypersurface (possibly having integer multiplicities) which are smooth up to a closed singular set.
\item If $\cohom(G) = 2$, further assume that the action has no exceptional orbits. Then the corresponding varifolds $V(u_{\eps_i})$ converge to a union of immersed, minimal $G$-equivariant minimal hypersurface (possibly having integer multiplicities) which are smooth up to a closed singular set.
\end{itemize}
In both cases, the singular set is dimension at most $n-7$ and contained in the union of non-principal orbits.
\end{theorem*}
\noindent One of our main tools is the following regularity theory of Tonegawa--Wickramasekera
\begin{theorem}[Theorem 2.1 \cite{TWStable}, Theorem 3.8 \cite{GuaracoMinmax}] \label{thm.TW.stable}
Let $\{u_{\eps_i}\} \subseteq W^{1,2}(M^{n+1})$ be a sequence of stable Allen--Cahn solutions with $\e_i \to 0^+$ and $E_{\eps_i}(u_{\eps_i}) \leq \Lambda$ for fixed $\Lambda > 0$. Then up to subsequence, the corresponding varifolds $V(u_{\eps_i})$ converge to a stationary, integral $n$-varifold $V$. Moreover, if $\mathrm{Ind}_{\e_i}^G(u_{\eps_i})\leq N$, for some $N>0$ independent of $i$, then $\mathrm{supp}\|V\|$ is a smooth, embedded, stable, minimal hypersurface away from a singular set $\mathrm{sing} \; V$ with Hausdorff dimension $\leq n-7$. 
\end{theorem}
\noindent We begin with a proposition inspired by work of Wang \cite[Lemma 2.8]{WangJGA}:
\begin{proposition}
Suppose $u_{\eps}$ is a G-stable Allen--Cahn solution in $\Omega \subseteq M$, a $G$-invariant set. Then $u_{\eps}$ is stable in $\Omega$.
\end{proposition}
\begin{proof}
Note that stability is equivalent to showing that the first eigenfunction of the second variation corresponds to a positive eigenvalue. Let $v_1(x) \in C_c^1(\Omega)$ be the first eigenfunction for the second variation of $u_{\eps}$ corresponding to $\lambda_1$. Since this is the first eigenfunction, it is strictly positive everywhere, and hence 
\[
v_G(x) = \frac{1}{|G|} \int_G v(g \cdot x) dV_g
\]
is non-zero. Moreover, by an elementary computation, we have that 
\[
J_{u_{\eps}} v_G(x) = \lambda_1 v_G(x)
\]
and hence $\lambda_1 < 0$ by $G$-stability, and hence $u_{\eps}$ is stable in $\Omega$. 
\end{proof}
We can also propagate finite $G$-equivariant morse index to regularity of the hypersurface outside of finitely many points
\begin{proposition} \label{prop.index.converge}
Let $\{u_{\eps_i}\} \subset W^{1.2}_G(M)$ be a sequence of Allen--Cahn solutions with $\eps_i \to 0^+$ and with uniformly bounded energy and Morse index, $\mathrm{Ind}_{\eps_i}^G(u_{\eps_i})\leq p$. Let $V$ be a (subsequential) limiting stationary varifold obtained from the sequence $V(\epsilon_i)$, and $\Sigma = \mathrm{supp}(V)$. Then $\Sigma$ is an embedded minimal hypersurface, smooth up to a set of dimension at most $n-7$, away from finitely many orbits of the form $\Q_i = G \cdot q_i$ where $q_i \in M$.
\end{proposition}
\noindent Note that in the above proposition,  we do not distinguish between principal, exceptional, and singular orbits.
\begin{proof}
This follows from a standard argument, as if we choose any $(p+1)$ disjoint, $G$-invariant tubes around orbits $\{G\cdot q_i\}_{i = 1}^{p+1}$ for $q_i \in M$, then by nature of having bounded $G$-equivariant index by $p$, $u_\eps$ must be $G$-equivariantly stable and hence stable at least one such tube. By Theorem \ref{thm.TW.stable}, $\Sigma$ is necessarily regular in that tube. Choosing a finite covering of $\Sigma$ with balls of smaller and smaller radius, we conclude the theorem. 
\end{proof}
\noindent We can also upgrade finite index and regularity away from finitely many points to being fully smooth when the principal orbits are sufficiently small. We note that in the original setting, the index of Allen--Cahn solutions may accumulate around single \textit{points} in the limiting varifold, which can be smoothened out using the condition of being stable in annuli (see e.g. \cite[Theorem A, Theorem 3.8]{GuaracoMinmax}). In our setting, we must use $G$-equivariant annuli, which are much larger  when $\cohom(G) \neq n+1$. However, we circumvent this difficulty by relying on Wickramasekera's $\alpha$-structural hypothesis, which allows stable varifolds with singular sets, $Z$, such that $\mathcal{H}^{n-1}(Z) = 0$ to be actually smooth. We summarize this as follows 
\begin{corollary}[Theorem 3.1 \cite{Wickramasekera}] \label{cor.wick.regular}
Suppose $V$ is a stationary, stable varifold in an open set $U \cap \text{reg}(V)$ and $\mathcal{H}^{n-1}(\mathrm{sing}(V) \cap U) = 0$, then $V$ is supported along a minimal hypersurface in $U$, which is smooth up to a set of dimension at most $n-7$.
\end{corollary}
The condition of $\mathcal{H}(\text{sing}(V) \cap U) = 0$ means that the ``$\alpha$-structural hypothesis" of Wickramasekera is automatically true (see also the discussion at the beginning of \cite[\S 3]{TWStable}). We can now prove an upgraded regularity statement:
\begin{proposition} \label{proposition.cohom.three}
Suppose in the above that $\text{Cohom}(G) \geq 3$, then $\Sigma^{reg} = \Sigma \cap M^{reg}$ is embedded and smooth up to a set of dimension at most $n-7$. 
\end{proposition}
\begin{proof}
In the previous proof, if $\text{Cohom}(G) \geq 3$, then we have that $\dim(G \cdot q_i) \leq n - 2$, and hence we see that $\Sigma^{reg}$ is a stable, smooth minimal hypersurface outside of a set of dimension $n-2$. By Corollary \ref{cor.wick.regular}, we conclude that $\Sigma^{reg}$ is smooth up to a set of dimension at most $n-7$. 
\end{proof}
We can also show that the potentially non-smooth orbits from Proposition \ref{prop.index.converge} are contained in the union of all non-principal orbits of the action of $G$ on $M$
\begin{proposition} \label{prop.singularity.localize}
Suppose that $3 \leq \cohom(G) \leq 7$. Let $\{\Q_i = G \cdot q_i\}$ be the finite set of orbits contained in $\Sigma$ from Proposition \ref{prop.index.converge}. If for some $i$, $\Q_i \subseteq M^{reg}$, then $\Sigma$ is smooth and embedded in an open neighborhood of $\Q_i$.
\end{proposition}
\begin{proof}
Let $U$ be a small open neighborhood of $\Q_i$. Let $\tilde{V}$ denote the corresponding quotient varifold on $U/G$, which has support $\tilde{\Sigma}: = (\Sigma \cap U) / G$ lying in $U/G \subseteq M^{reg} / G$. $\tilde{V}$ is now a stationary varifold on $(U/G, \ov{g})$ with finite index. By potentially taking $U$ smaller, we conclude $\tilde{V}$ is stable on $(U \backslash \Q_i) / G = (U/G) \backslash \{q_i\}$ where $q_i = \Pi(\Q_i)$. Let $m = \dim(U/G)$, so that by assumption $3 \leq m \leq 7$. Since $\cohom(G) \geq 3$, we see that $\mathcal{H}^{m-2}\{q_i\} = 0$ Applying Corollary \ref{cor.wick.regular} to $\tilde{V}$ on $U/G$, we conclude that $\tilde{\Sigma}$ is smooth and embedded everywhere on $U/G$. Lifting back to $M$, we conclude that $\Sigma$ is smooth and embedded in $U$.
\end{proof}
Combining propositions \ref{prop.index.converge}, \ref{proposition.cohom.three}, and \ref{prop.singularity.localize} finishes the proof of Theorem \ref{thm.AC.regularity} when $\cohom(G) \geq 3$.

\section{Regularity for Cohomogeneity $2$} \label{section.cohom.two}
We now show that if there are no exceptional orbits then we can conclude regularity everywhere for actions with Cohomogeneity $2$. Note that from Montgomery \cite{montgomery1956exceptional}, the union of the singular orbits is \textit{at most} $n-1$ dimensional, so that a priori, the union may correspond to a set of codimension $1$ within a minimal hypersurface may. However, the \textit{individual} non-principal orbits are still at most $(n-2)$-dimensional in this setting (see \S \ref{section.appendix.examples} for examples). \nl 
\indent Recall that in their seminal work \cite{ChodoshMantoulidisWidths}, the authors used the \emph{Sine--Gordon potential}
\[
W(t) = \frac{1 + \cos(\pi t)}{\pi^2}
\]
to define the energy functional \eqref{eqn.AC.energy}. Using this potential, they showed the following:
\begin{theorem}[Theorem 3.1, Chodosh--Mantoulidis] \label{thm.cho.man.reg}
Let $(U, g)$ be a relatively compact open submanifold of a two dimensional Riemannian manifold. Consider a sequence $\{u_{\eps_i}\}$ of solutions to equation \eqref{eqn.AC.energy} with the Sine--Gordon potential such that 
\[
\|u_{\eps_i}\|_{\infty}+ \mathrm{Ind}_{\eps_i}(u_{\eps_i}; U) + E_{\eps_i}(u_{\eps_i}, U) \leq \Lambda
\]
for $\Lambda > 0$ independent of $i$. Let $V$ be a (subsequential) limiting stationary varifold obtained from the sequence $V(\epsilon_i)$ and assume $\|V\|(U)>0$. Then
\[
V \Big|_U  = \sum_{j = 1}^N V_{\sigma_j}
\]
where each $\sigma_j$ is a smoothly immersed geodesic on $U$, and $V_{\sigma_j}$ is the varifold induced by $\sigma_j$ (we allow for repetition among the $\{\sigma_j\}$).
\end{theorem}
\indent We are unfortunately unable to apply Theorem \ref{thm.cho.man.reg} directly to $U = M / G$, as $M / G$ is not a closed surface due to the presence of exceptional and singular orbits. Moreover, even in the case when $M/G$ is a smooth closed surface, it is not true that $G$-equivariant solutions descend to solutions to the Allen--Cahn equation on $(M/G, g)$ unless the fibers are constant volume. \nl 
\indent Indeed, let $u_{\eps}$ be a $G$-equivariant solution on $M$ and let $\ov{u}_{\eps}$ denote the corresponding projected function on $U/G$. Suppose $\overline{v} \in C^{\infty}_c(U/G)$ and let $v$ denote its lift to $C^{\infty}_c(U)$, then we see that weakly )
\begin{align*}
\frac{d}{dt}\bigg|_{t=0} E_{\eps}(\overline{u}_{\eps} + t\overline{v}, U/G, g_{M/G}) &= \int_{U/G} \left( \eps \langle \n^{g_{M/G}} \overline{u}_{\eps}, \n^{g_{M/G}} \overline{v} \rangle + \frac{W'(\ov{u}_{\eps})}{\eps} \ov{v} \right) dV_{g_{M/G}} \\
&= \int_{U/G} \int_{\pi^{-1}([p])} \left( \eps \langle \n^{g_{M/G}} \overline{u}_{\eps}, \n^{g_{M/G}} \overline{v} \rangle + \frac{W'(\ov{u}_{\eps})}{\eps} \right) \mathcal{V}^{-1}([p]) dV_{g_{M/G}} dV_{\pi^{-1}([p])}\\
&= \int_{U} \left( \eps \langle \n^{g} u_{\eps}, \n^{g} v \rangle + \frac{W'(u_{\eps})}{\eps} \right) \mathcal{V}^{-1} dV_{g}\\
&= \int_{U} \eps \langle \n^g u_{\eps}, \n \ln(\mathcal{V}) \rangle v \mathcal{V}^{-1} dV_g \\
&= \int_{U/G} \int_{\pi^{-1}([p])} \eps \langle \n^g u_{\eps}, \n \ln(\mathcal{V}) \rangle v \mathcal{V}^{-1} dV_{g_{M/G}} dV_{\pi^{-1}([p])} \\
&= \int_{U/G} \eps \langle \n^{g_{M/G}} \overline{u}_{\eps}, \n^{g_{M/G}} \ln(\mathcal{V}) \rangle \overline{v} dV_{g_{M/G}}
\end{align*}
where we recall that $\pi^{-1}([p])$ is the fiber of the action corresponding to a point $[p] \in M/G$, which has volume $\mathcal{V}([p])$. This means that $\overline{u}_{\eps}$ weakly solves 
\begin{equation} \label{eqn.weak.one}
-\eps \Delta_g \overline{u}_{\eps} + \frac{W'(u)}{\eps} - \eps \langle \n^g \overline{u}_{\eps}, \n \ln(\mathcal{V}) \rangle = 0
\end{equation}
We may refer to equation \eqref{eqn.weak.one} as an Allen--Cahn with drift equation, and we remark that solutions to equation \eqref{eqn.weak.one} are critical points of a weighted Allen--Cahn energy functional (see equation \eqref{eqn.energy.weighted}). While equation \eqref{eqn.weak.one} differs from the original Allen--Cahn equation \eqref{ACEquation}, solutions to equation \eqref{eqn.weak.one} blow up to $\eps = 1$ solutions to equation \eqref{ACEquation} on the plane. Thus one may expect similar regularity results for solutions to the drift equation as with solutions to the original Allen--Cahn equation. \nl 
\indent Our first step to proving the analogous result of Theorem \ref{thm.cho.man.reg} for equation \eqref{eqn.weak.one}  is to establish regularity for equation \eqref{eqn.weak.one}.
\begin{theorem} \label{thm.drift.stability}
Let $(U, g)$ be a smooth open two dimensional manifold with boundary. Suppose that $u_{\eps_i}$ are a sequence of stable solutions in $U$ with respect to the energy functional 
\begin{equation} \label{eqn.energy.weighted}
E_{\eps}(u, U, \mathcal{V}, g): =\int_U \mathcal{V} \left( \frac{\eps}{2} |\n^g u|^2 + \frac{W(u)}{\eps} \right) dV_g
\end{equation}
let $B = \frac{|\n^2 u|^2 - |\n |\n u||^2}{|\n u|^2} $ denote the enhanced second fundamental form and fix $0 < \beta < 1$. There exists $C, \eps_0 > 0$ such that if $\eps < \eps_0$ and 
\[
|B(x)| \leq C \qquad \forall x \in U \cap \{|u| \leq 1 - \beta \}
\]
then for all $U' \subset \subset U$
\begin{equation} \label{eqn.stability.estimate}
|B(x)| \leq C_* \eps^{1/7} \qquad \forall x \in U' \cap \{|u| \leq 1 - \beta\}
\end{equation}
where $C_*$, $\eps_0$ depend on $g$, $\text{dist}(\p U, U')$ but not $\eps$.
\end{theorem}
We delay the proof of this result (and related ones) to section \S \ref{section.WW.drift}. We now prove our main cohomogeneity $2$ regularity result:
\begin{proposition} \label{prop.principal.immersion}
Suppose $\cohom(G) = 2$, then Proposition \ref{prop.index.converge} yields $\Sigma$, such that $\Sigma$ is embedded away from finitely orbits and $\Sigma \cap M^{reg}$ is a smoothly immersed minimal hypersurface.
\end{proposition}
\begin{proof}
Suppose that $\{u_{\eps_i}\}$ are $G$-invariant Allen--Cahn solutions with bounded energy and $G$-equivariant index. Applying Proposition \ref{prop.index.converge}, we obtain a limit varifold, $V$, $\Sigma := \text{supp}(V)$, and finitely many orbits of the form $\{Q_j = G \cdot q_j\}$. Away from these orbits, our limiting stationary varifold is supported a hypersurface which is embedded and smooth up to a closed singular set of dimension at most $n-7$. \nl 
\indent Let $p \in \text{Sing}(V) \backslash \cup_{i} \Q_i$. If $p \in M^{reg}$, then noting that the singular set is $G$-invariant, we have that $V$ is singular all along $G \cdot p$, which is a set of dimension $n-2$, a contradiction to the singular set being at most dimension $n-7$. Thus any singular points either lie inside $\cup_i \Q_i$ or inside the union of the non-principal orbits of the action of $G$ on $M$. \nl 
\indent Let $\overline{u}_{\eps}: M / G \to \R$ be the corresponding function on the quotient via $\overline{u}_{\eps}(\overline{x}) = u_{\eps}(x)$ where $x$ is any element of the orbit $\pi^{-1}(\overline{x})=G\cdot x$. For any precompact $G$-equivariant open subset $U \subset\!\subset  M^{reg}$, we have that $u_{\eps_i}$ being uniformly bounded in energy and index implies that $\overline{u}_{\e_i}$ is uniformly bounded in energy and index on $U/G$ as follows:

By the same argument of Fubini's theorem as in lemma \ref{lem.energy.compare}, we have 
\begin{align*}
E_{\eps}(u, U) & = \int_U \frac{\eps}{2} |\nabla^g u|^2 + \frac{W(u)}{\eps} dV_g \\
&= \int_{U/G} \mathcal{V} \left( \frac{\eps}{2} |\nabla^g \ov{u}|^2 + \frac{W(\ov{u})}{\eps}  \right) dV_{U/G}.
\end{align*}
Noting that $||\mathcal{V}||_{\infty}$ is finite, we conclude that  equation \eqref{eqn.energy.weighted} is bounded in $U$. Moreover, a simple integration by parts shows that the finite $G$-equivariant index of $u_{\eps}$ implies that $\ov{u}_{\eps}$ has finite index on $U$ with respect to \eqref{eqn.energy.weighted}. \nl 
\indent Now, the analogous statement of \cite[Theorem 3.1, Proposition 3.8]{ChodoshMantoulidisWidths} holds via the improved estimates of stable solutions to the Allen--Cahn with drift equation \ref{thm.drift.stability}, and the following remarks
\begin{itemize}
\item Chodosh--Mantoulidis \cite[Lemma C.6]{ChodoshMantoulidisWidths} holds for $E_{\eps}(u,U, \mathcal{V}, g)$ as above (see lemma \ref{lem.linear.energy.growth}).

\item The usage of \cite[Proposition C.1]{ChodoshMantoulidisWidths} on a two dimensional set can be replaced by using the same proposition on $U \cdot G$, concluding that one obtains a stationary $G$-equivariant varifold on $U \cdot G$, which then projects to a stationary $1$-varifold on $U$ with respect to the Hsiang--Lawson conformal metric, $\tilde{g} = \mathcal{V} g$.
\end{itemize}
The rest of the proof \cite[Proposition 3.8]{ChodoshMantoulidisWidths} now adopts to the drift setting, and hence we conclude that $\Sigma$ is smoothly immersed at any $\Q_i \subseteq M^{reg}$.
\end{proof}
\noindent We now improve the regularity under the assumptions listed at the beginning of this section:
\begin{proposition} \label{smallNonPrincipal}
Suppose that all of the non-principal orbits have dimension at most $n-2$. Let $\{u_{\eps_i}\}$ be a sequence of $G$-invariant Allen-Cahn solutions with uniformly bounded energy and $G$-equivariant index bounded by $p$. Let $V$ the limiting varifold of $V(u_{\eps_i})$, and $\Sigma = \text{supp}(V)$. Then $\Sigma$ is a $G$-invariant minimal hypersurface, smooth up to a singular set of dimension at most $n-7$, which lies in the union of all non-principal orbits.
\end{proposition}
\begin{proof}
\indent For any points $q_i \in M^{reg}$, the argument remains the same as Proposition \ref{prop.principal.immersion}, and we can conclude that $\Sigma \cap M^{reg}$ is a union of smooth (potentially immersed) geodesics on $M^{reg}/G$. If $q_i \in (M \backslash M^{reg})$, then $Q_i = G \cdot q_i$ is necessarily dimension $n-2$ or smaller, meaning that we can apply Tonegawa--Wickramasekera to smooth across $Q_i$ up to a singular set of dimension at most $n-7$.
\end{proof}
\begin{remark}
We emphasize that in the hypothesis of Proposition \ref{smallNonPrincipal}, we require that the dimension of any \textit{individual} non-principal orbit has dimension $\leq n - 2$. We compare this with Wang who has established regularity under different conditions: his initial work \cite{wang2022min} assumed that the union of the non-principal orbits is at most $n-2$ dimensional, while  later work \cite{WangJGA} \cite{WangDensity} \cite{WangIndex} removed this assumption and established regularity when all orbits have codimension between $3$ and $7$ (in the closed setting). See also the work of Liu \cite{LiuCVPDE}. In \S \ref{section.appendix}, we provide examples in which the union of the non-principal orbits is $n-1$ dimensional, though each non-principal orbit is $\leq n-2$ dimensional.  \nl 
\indent We also note that because the action has $Cohom(G) = 2$, the requirement of non-principal orbits having dimension at most $n-2$ is equivalent to saying \textit{there are no exceptional orbits} (see \S \ref{section.appendix} for definitions).
\end{remark}
Combining propositions \ref{prop.index.converge}, \ref{prop.principal.immersion}, and \ref{smallNonPrincipal}, we conclude the proof of Theorem \ref{thm.AC.regularity} when $\cohom(G) = 2$.

%% file: Index_Bounds.tex
\section{Index Bounds} \label{section.index}
We consider $\{u_{\eps_i}\}$ a sequence of G-equivariant Allen--Cahn solutions with $|u_{\eps}| \leq 1$, $E_{\eps}(u_{\eps}) \leq \Lambda_0$, and $\text{Ind}_{G}(u_{\eps}) \leq p$. Let $V(u_{\eps})$ be the corresponding measures for the these solutions and 
\[
V = \lim_{i \to \infty} V(u_{\eps_i})
\]
the stationary integral varifold which arises as a limit of a subsequence of these measures. The goal of this section is to prove the following theorem:
\begin{theorem} \label{thm.index.bound}
For $V$ as above, suppose that the support of $V$ is induced by a union of minimal surfaces (potentially with multiplicity), which are smooth, disjoint, and embedded away from a set of at most dimension $n-7$, i.e. $\text{supp}(V) = \cup_{i = 1}^M \Sigma_i$, then 
\[
\sum_{i = 1}^M \text{Ind}_G(\text{Reg}(\Sigma_i)) \leq p
\]
\end{theorem}
\noindent Recall that $\text{Reg}(\Sigma_i)$ denotes the regular part of $\Sigma_i$. Much of the work of the second author \cite{gaspar2020second} now translates to the equivariant setting, including his computations of the Allen--Cahn first and second \emph{inner variations}, namely $\delta E_\e(u,X) = DE_\e(u)(\langle \nabla u,X\rangle)$ and $\delta^2E_\e(u,X) = D^2E_\e(u)[\langle \nabla u,X\rangle,\langle \nabla u,X\rangle]$ with a geometric description:
\begin{proposition}[Prop 3.2, Gaspar \cite{gaspar2020second}] \label{prop.gaspar.var}
It holds that 
\[
\delta E_{\eps}(u, X) = \int_M \left[ \frac{\eps |\n u|^2}{2} + \frac{W(u)}{\eps}  \div(X) - \eps \langle \n_{\n u} X, \n u \rangle \right] 
\]
and 
\begin{align*}
\delta^2 E_{\eps}(u, X) &= \int_M \Big(\div (\n_X X) - \Ric(X, X) + \tr_g S_x - \frac{1}{2} |h_X|^2 + (\div X)^2 \Big) d e_{\eps} \\
& + \eps \int_M \Big( T_X(\n u, \n u) + 2 \langle \n_{\n_{\n_u} X} X, \n u \rangle - \langle \n_{\n u} \n_X X, \n u \rangle \\
& \qquad - 2 \langle \n_{\n u} X, \n u \rangle \div X + R(X, \n u, X, \n u) \Big) 
\end{align*}
where 
\begin{align*}
S_X(Y_1, Y_2) &=  \langle \n_{Y_1} X, \n_{Y_2} X\rangle  \\
h_X(Y_1, Y_2) &= \langle \n_{Y_1} X, Y_2 \rangle + \langle Y_1, \n_{Y_2} X \rangle \\
T_X(Y_1, Y_2) &= \tr_g((Z_1, Z_2) \mapsto \langle \n_{Z_1} X, Y_1\rangle \cdot \langle \n_{Z_2} X, Y_2 \rangle )
\end{align*}
\end{proposition}
We will be interested in apply Proposition \ref{prop.gaspar.var} when $X$ is a G-equivariant vector field.  We now assume that the limit varifold, $V$, takes the form 
\begin{equation} \label{eqn.varifold.decomp}
V = \sigma \sum_{i = 1}^N v(\Gamma_j, m_j)
\end{equation}
where $\Gamma_j$ are connected minimal surfaces, $m_j \in \Z^+$ and $\sigma = \int_{-1}^1\sqrt{W(s)/2} ds$ is a normalizing constant. We now conclude a $G$-equivariant version of \cite[Prop 2.2]{gaspar2020second}:
\begin{proposition} \label{prop.normal.converge}
Let $\{u_{\eps_k}\}$ be a sequence of G-equivariant Allen--Cahn solutions satisfying equation \eqref{eqn.varifold.decomp}. Then up to subsequence
\[
\lim_{k \to \infty}\eps_k \int_M T(\n u_{\eps_k}, \n u_{\eps_k}) = 2 \sigma \sum_{j = 1}^N m_j \int_{\Gamma_j} T(n_j, n_j)
\]
where $T$ is any $G$-equivariant $(0,2)$ tensor on $M$. Here $n_j$ denotes a measure choice of a unit normal vector field defined on $\Gamma$.
\end{proposition}
\begin{proof}
The same proof as in \cite{gaspar2020second} holds.
\end{proof}
With propositions \ref{prop.gaspar.var} and \ref{prop.normal.converge}, we can conclude the geometric convergence of the Allen--Cahn second inner variation to the variation of the limiting varifold plus an error term
\begin{proposition} \label{prop.second.var.converge}
Let $\{u_{\eps_k}\}$ our $G$-equivariant sequence of solutions as above. We have that up to subsequence
\[
\frac{1}{2 \sigma} \lim_{k \to \infty} \delta^2 E_{\eps_k}(u_{\eps_k}, X) = \delta^2 V(X) + \sum_{j = 1}^N m_j \int_{\Gamma_j} \Big( \langle \n_{n_j} X, n_j \rangle^2 + R(X, n_j, X, n_j) \Big)
\]
for any G-equivariant vector field $X$.
\end{proposition}
\noindent The same proof of Theorem A in \cite[\S 4]{gaspar2020second} now gives the conclusion of Theorem \ref{thm.index.bound}.

%% file: Ricci_Pos.tex
\section{Positive Ricci for embedded $G$-invariant minimal surface of cohomogeneity $\geq 3$} \label{section.ricci}
The purpose of this section is to prove the following result:
\begin{theorem} \label{thm.Ric.pos.mult.one}
For $(M^{n+1}, g)$, $3 \leq n+1 \leq 7$ and $G$ as above, assume that $\text{Cohom}(G) \geq 3$ and $\Ric_g > 0$. Then 
\[
\omega_{1,AC}^G(M,g) = \text{Area}(\Sigma_1)
\]
i.e. the first $G$-invariant width occurs with multiplicity one.
\end{theorem}
\noindent The idea behind Theorem \ref{thm.Ric.pos.mult.one} is an adaptation of the following result of Bellettini: \cite{bellettini2024multiplicity}
\begin{theorem} \label{thm.belletini}
Let $N$ be a compact Riemannian manifold of dimension $3 \leq n+1$ with positive Ricci curvature. Let $M \subseteq N$ be any smooth minimal hypersurface such that for every $x \in M$ there exists a geodesic ball in $N$ centered at $x$ in which $M$ is stable. Then the mountain pass Allen--Cahn min-max value, $c_{\eps}$ satisfies
\[
\limsup_{\eps \to 0} c_{\eps} < 2 \mathcal{H}^n(M)
\]
\end{theorem}
\noindent Since the idea of theorem \ref{thm.Ric.pos.mult.one} is a short adaptation of Theorem \ref{thm.belletini}, we only sketch the details below for the following analogous theorem, where we impose G-equivariance and require an upper bound on the dimension of the ambient manifold.
\begin{theorem} \label{thm.belletini.G.invar}
Let $N$ be a compact Riemannian manifold of dimension $3 \leq n+1 \leq 7$ with positive Ricci curvature and isometric action given by $G$. Let $M \subseteq N$ be any embedded smooth $G$-invariant minimal hypersurface. Then the $G$-invariant mountain pass Allen--Cahn min-max value defined in Proposition \ref{1 parameter min-max}, $\omega^G_{AC,\eps}$, satisfies
\[
\frac{1}{2\sigma}\limsup_{\eps \to 0} \omega_{AC,\e}^G < 2 \mathcal{H}^n(M)
\]
\end{theorem}
\noindent We remark that Bellettini's condition of ``for every $x \in M$ there exists a geodesic ball in $N$ centered at $x$ in which $M$ is stable" is automatically true for smooth embedded minimal surfaces. Moreover, by the dimension restriction of $n+1 \leq 7$, any embedded minimal surface is automatically smooth, and hence this condition is redundant. From Theorem \ref{thm.belletini.G.invar}, Theorem \ref{thm.Ric.pos.mult.one} follows, so we sketch the proof of Theorem \ref{thm.belletini.G.invar}. 
\begin{proof}[Proof of Theorem \ref{thm.belletini.G.invar}]
Bellettini constructs a one parameter family of functions, $\varphi: [0,1] \to H^1(M)$ such that $\varphi(0) = -1$, $\varphi(1) = +1$, and 
\begin{equation} \label{eqn.varphi}
\frac{1}{2\sigma}\sup_{t \in [0,1]} E_{\eps}(\varphi(t)) \leq 2 \mathcal{H}^n(M) - \delta
\end{equation}
$\varphi\Big|_{[0,1/2]}$ is constructed via explicit functional constructions which rely on the signed and unsigned distances to $M$ (after rescaling the time domain), as well as an excised small ball $B \subseteq M$. When $M$ happens to be a $G$-invariant minimal hypersurface, these distance functions are also $G$-invariant. Furthermore, we replace $B$ with its $G$-invariant counterpart $G \cdot B$ and find his analogous perturbing function as follows:
\begin{lemma} \label{lem.unstable.region}
	There exists a geodesic, $G$-invariant, tube $D \subset M$ and $\tilde{\phi} \in C^2(\tilde{M})$ with $\tilde{\phi} \geq 0$ such that the support of $\tilde{\phi}$ is contained in $\tilde{M} \backslash \iota^{-1}(D)$ and 
	\[
	\int_{\tilde{M}} |\n \tilde{\phi}|^2 - \int_{\tilde{M}} \tilde{\phi}^2 (|A|^2 + \Ric_N(\nu,\nu)) < 0
	\]
\end{lemma}
\begin{proof}
The construction is identical to that of \cite[\S 5.1]{bellettini2024multiplicity}, though we note that we are now constructing a $G$-invariant tube about a point, instead of a small ball about a point. Let $p \in N^{reg} \cap M$, which must exist since the union of all non-regular points of the action has finite $n-1$-Hausdorff dimension (and $M$ has non-zero $n$-Hausdorff dimension). Because the cohomogeneity of the action is at least $3$, we see that for any $p \in N$, $G \cdot p$ has finite $\mathcal{H}^{n-2}$ measure and hence there exists $\rho$, a $G$-invariant function which is $1$ on the open $G$-invariant tube of size $\delta$ about $G \cdot p$, $0$ outside of the $G$-invariant tube of size $2 \delta$, $D_{2\delta}$, about $G \cdot p$, satisfies $\int_N |\nabla \rho|^2 < K \delta$ for some $K$ independent of $\delta$. In Bellettini's language, the $2$-capacity of $D_{2\delta}$ is finite and tending to $0$ with $\delta \to 0$. The function $\tilde{\phi}(q) = 1 - \rho(\iota(q))$ now works for $\delta$ sufficiently small.
\end{proof}
\noindent The above lemma implies that the constructions in Bellettini \S 3 - \S 7.4 yield $G$-invariant functions. These sections are used to define (again, up to a reparameterization of the time interval $[0, t_0+1] \to [0, 1/2]$) $\varphi\Big|_{[0,1/2]}$ and hence $\varphi\Big|_{[0,1/2]}$ provides a map into $X_G(M)$. For $t > 1/2$, $\varphi(t)$ is defined via a regularization of the parabolic Allen--Cahn equation, given by 
\begin{equation} \label{eqn.perturbed.flow}
\frac{\partial u}{\partial t} = \eps \Delta_g u - \frac{W'(u)}{\eps} + \mu_{\eps}
\end{equation}
where $\mu_{\eps} > 0$ is a constant tending to $0$ as $\eps \to 0$. Hence it suffices to show that this flow preserves $G$-invariant initial data, as this means that $\varphi\Big|_{(1/2,1]}$ will also be $G$-invariant. 
\begin{lemma} \label{lem.stays.G.invar}
	Suppose $(N^{n+1}, g)$ is a closed Riemannian manifold with smooth isometric action from a Lie group, $G$. Consider the perturbed Allen--Cahn flow
	\begin{equation} \label{eqn.perturbed.ac.flow}
		\frac{\partial u}{\partial t} = \eps \Delta_g u - \frac{W'(u)}{\eps} + \mu_{\eps}
	\end{equation}
	corresponding to the gradient flow of the functional, $\mathcal{F}_{\eps, \mu_{\eps}}(u) = E_{\eps}(u) - \mu_{\eps} \int_N u$. For smooth initial condition $u_0 \in C^{\infty}(N)$, the corresponding flow, $u(x,t)$ exists for all time $t$. Moreover, if $u_0$ is $G$-invariant, then $u(x,t)$ is $G$-invariant for all $t \in [0, \infty)$
\end{lemma}
\begin{proof}
	For any $g \in G$, consider $u^g(t,x) = u(t, g \cdot x)$. Then both $u(t,x)$ and $u^g(t,x)$ solve equation \eqref{eqn.perturbed.ac.flow} with the same initial data $u(0,x) = u_0(x) = u^g(0,x)$ as $\Delta_g$ is also $G$-invariant. Hence by uniqueness of the flow, $u(t,x) = u^g(t,x)$ for all $(x,t)$.
\end{proof}
\noindent We remark that Bellettini employs various smoothing operators before applying his perturbed parabolic Allen--Cahn flow \eqref{eqn.perturbed.flow}. This is needed because in his setting, $M$ is not smooth a priori and has a singular set of codimension at least $7$. In our case, we assume that the ambient manifold is at most $7$ dimensional, meaning that the corresponding minimal hypersurfaces are smooth. Thus the constructions in \S 7.5 used to define $\varphi\Big|_{(1/2,1]}$ (with the smoothing operators omitted) now give $G$-invariant functions. \nl 
\indent Noting that the entire path, $\varphi(t)$, is now $G$-invariant but the inequality of equation \eqref{eqn.varphi} holds, we conclude the proof.
\end{proof}
\begin{remark} \label{codim.two.remark}
We remark that in cohomogeneity $2$, the minimal surfaces to consider can a priori be immersed, so that we lose regularity of the distance function in a tubular neighborhood of our hypersurface, preventing us from directly applying Bellettini's work \cite{bellettini2024multiplicity}. We also note that Lemma \ref{lem.unstable.region} does not hold in cohomogeneity $2$, as a $G$-invariant tube $D_{2\delta} \subseteq M$ has positive $1$-capacity inside of $M$ and hence one cannot construct such a $\tilde{\phi}$ with compact support in Lemma \ref{lem.unstable.region} via the same method. As of now, the authors conjecture that the multiplicity one result should still hold.
\end{remark}
\subsection{Smooth quotients with Cohomegeneity $2$ and $\Ric_{M/G} > 0$}
In this section, we make some simple yet fruitful remarks for the situation in which 
\begin{enumerate}
\item $M / G$ is itself a smooth manifold (i.e. no non-principal orbits).
\item The Hsiang--Lawson conformal metric (see equation \eqref{eqn.hsiang.lawson.metric}) on the quotient, $\overline g :=\overline{g}_{M/G}$, satisfies $\Ric_{\overline{g}} > 0$.
\end{enumerate}
As an example, any Berger metric on $S^3$ with its $S^1$ action will have a quotient as $S^2$ with positive Gaussian curvature.
\begin{theorem} \label{thm.cohom.2.pos}
Suppose that $Cohom(G) = 2$ and $K_{\overline{g}} > 0$. Then
\[
\omega_{1}^G(M,g)  = \omega_1(M/G,\bar g)= \text{Area}(\Sigma_1),
\]
where $\Sigma_1$ is an embedded $G$-equivariant minimal hypersurface in $(M,g)$.
\end{theorem}
\begin{remark}
Note that Theorem \ref{thm.cohom.2.pos} asserts both multiplicity one and also embeddedness for the first width in the cohomogeneity $2$. Neither of these conditions should hold in general for $\omega_{p, G}$. 
\end{remark}
\begin{proof}[Proof of Theorem \ref{thm.cohom.2.pos}]
We know that $(M/G, \overline{g})$ is a smooth topological sphere with $K_{\overline{g}} > 0$, and henceforth we refer to this as $(S^2, \overline{g})$. Applying the work of Calabi--Cao \cite[Thm 3.1]{CalabiCao}, we have directly that 
\begin{equation} \label{eqn.calabi.cao}
\Lambda_{AP, \overline{g}} = \ell(\gamma_0)
\end{equation}
where $\Lambda_{AP, \overline{g}}$ denotes the Almgren--Pitts one parameter width and $\gamma_0$ is the shortest geodesic on $(S^2, \overline{g})$, which happens to be embedded because of the $K_{\overline{g}} > 0$ condition and \cite[Theorem D]{CalabiCao}.
Note however, that the regularity theory of Chodosh-Mantoulidis \cite[Theorem 3.1]{ChodoshMantoulidisWidths} gives that 
\[
\omega_{1}(M/G,\overline{g}) = \sum_{i = 1}^{\tilde{N}} \tilde{m}_i \ell(\tilde{\gamma}_i)
\]
where $\tilde{m}_i \in \Z^+$ and $\tilde{\gamma}_i$ are also closed primitive geodesics. Thus, trivially we have $\omega_{1}(M/G,\overline{g}) \geq \ell(\gamma_0)$, and from the general inequality of $\omega_{1}(M/G,\overline{g}) \leq \Lambda_{AP, \overline{g}}$, we conclude that 
\[
\ell(\gamma_0) = \Lambda_{AP, \bar g} = \omega_{1}(M/G,\overline{g})
\]
\end{proof}

%% file: Appendix.tex
\section{Appendix} \label{section.appendix}
\subsection{Introduction to Group Actions and Examples} \label{section.appendix.examples}
In this section, we recall some basic information about Lie Group actions on closed manifolds. Let $G$ be a compact Lie group acting isometrically on $(M^{n+1}, g)$, a closed Riemannian manifold. \nl 
\indent For a closed subgroup $H \leq G$, we let $\langle H \rangle$ denote the conjugacy class of $H$ in $G$. Under this notation, $p \in M$ has orbit type $\langle H \rangle$ if $\langle G_p \rangle = \langle H \rangle$, for $G_p = \{g \in G \; | \; g \cdot p = p \}$ the stabilizer of $p$ (also known as the isotropy group of $p$). We let 
\[
M_{\langle H \rangle} = \{ p \in M \; | \; \langle G_p \rangle = \langle H \rangle \}
\]
to be the union of all points with orbit type $\langle H \rangle$. This is known to be a union of smooth embedded submanifolds of $M$ (see Wall \cite{wall2016differential} for more details), and there are only finitely many different orbit types on $M$. Recall that there are $3$ types of orbits:
\begin{itemize}
    \item Principal orbits, $P$, corresponding to a minimal conjugacy class so that the union of all principal orbits forms a dense open neighborhood of $M$. We define the cohomogeneity of $G$ as the codimension of a principal orbits, i.e. 
    \[
    \mathrm{Cohom}(G) = \text{codim}(P)
    \]
    \item Exceptional orbits, $Q$, which are not principal but $\text{codim}(Q) = \mathrm{Cohom}(G)$. 
    
    \item Singular orbits, $S$, which satisfy $\text{codim}(S) > \mathrm{Cohom}(G)$.
\end{itemize}
The orbits are themselves manifolds of potentially different dimensions leading to a stratification of the ambient manifold $M$.
\subsubsection{Examples of Group actions} \label{section.examples}
It is interesting to pose the following question: 
\begin{question}

Given a group action with $\mathrm{Cohom}(G) = \ell \geq 2$, let $k$ denote the Hausdorff dimension of $M \backslash M^{reg}$. In general, it is known that $k \leq n-1$, but do there exist group actions for every pair of $(\ell, k) \in \{2, \dots, n\} \times \{0, \dots, n-1\}$? 
\end{question}
\noindent Here are some examples
\begin{enumerate}
\item $SO(d) \curvearrowright S^{n+1}$ via acting on the first $d$ coordinates in $S^{n+1} = \{x_1^2 + \dots + x_{n+2}^2 = 1\} \subseteq \R^{n+2}$. Then the non-principal orbits are simply fixed points of the action of the form of the form 
\[
x = (0, \dots, 0, x_{d+1}, \dots, x_{n+2}) \; \st \; x_{d+1}^2 + \dots + x_{n+2}^2 = 1
\]
And hence the dimension of the non-principal orbits is $0$ but the union of the non-principal orbits is $k = n+1-d$. Moreover, the cohomogeneity of the action is $\ell = n+1 - (d-1) = n+2 - d$, so this demonstrates $\ell = k+1$. 
\item Let $S^1$ act on $S^{2n-1}$ via complex multiplication on the first $d$ components, i.e. 
\[
S^{2n-1} = \{|z_1|^2 + \dots + |z_n|^2 = 1\} \subseteq \C^n
\]
and 
\[
e^{i\theta} \cdot (z_1, \dots, z_n) = (e^{i \theta} z_1, \dots, e^{i \theta} z_d, z_{d+1}, \dots, z_n)
\]
the principal orbits are all one dimensional, and the non-principal orbits are all singular. In particular, they are points corresponding to when $z_1 = z_2 = \dots = z_d = 0$. In this case, the union of the singular orbits gives 
\[
\{(0,\dots, 0, z_{d+1}, \dots, z_n) \; | \; |z_{d+1}|^2 + \dots + |z_n|^2 = 1\}
\]
which is of dimension $2(n-d)-1$. This provides an example of $\ell = 2n - 2$ and $k = 2(n-d)-1$.

\item Let $G$ be a Lie group and $M$ an arbitrary manifold and consider the action of 
\[
G \curvearrowright G \times M, \qquad g \cdot (h, p) = (gh, p)
\]
i.e. trivial action on the product. If $G$ acts freely on itself (i.e. only one orbit) then each principal orbit is of the form $G \times \{p\}$ and there are no singular orbits. In this case, the cohomogeneity is $\ell = \dim(M)$ and the union of the singular orbits has dimension $k = 0$.

\item Let $\Z_2 \times S^1$ act on $\R^3$ via $(\pm 1, e^{i\theta}) \cdot (x,y, z) = (x \cos \theta - y \sin \theta, x \sin \theta + y \cos \theta, \pm z)$, i.e. $\Z_2$ acts via reflection about the $xy$ plane and $S^1$ acts via rotation about the $z$-axis. \nl \nl  
The cohomogeneity is $2$ and the principal orbits consist of any points (in polar coordinates) of the form $(r, \theta, z)$ where $r$ and $z$ are non-zero. The principal orbits are unions of circles at $(r,S^1, \pm z)$. There is an exceptional orbit at $z = 0$ and $r \neq 0$. And there are singular orbits at $r = 0$. If $r = 0$ and $z \neq 0$, the singular orbits are two points. If $r = 0 = z$, then the singular orbit is just the origin. \nl \nl 
This example shows an interesting stratification of the singular set. The fundamental domain is an orbifold, namely the closed upper right hand quadrant (see figure \ref{fig:action.example}). The boundary consists of a union of a point and two open rays (the $r$ and $z$ axis). Within the stratification of the quotient manifold, the rays correspond to the one dimensional component of the boundary and the origin corresponds to the $0$-dimensional part. \nl \nl 
The Hsiang--Lawson metric on the fundamental domain is given by 
\[
\overline{g} = V g_{euc} = 4 \pi r g_{euc} = 4 \pi r (dr^2 + dz^2)
\]
from which one can compute that the only geodesics are given by  $z = c$ (lifting to planes in $\R^3$) or the projection of the catenoid in the form $z = f(r)$. 
\begin{figure}[h!]
\centering
\includegraphics[scale=1]{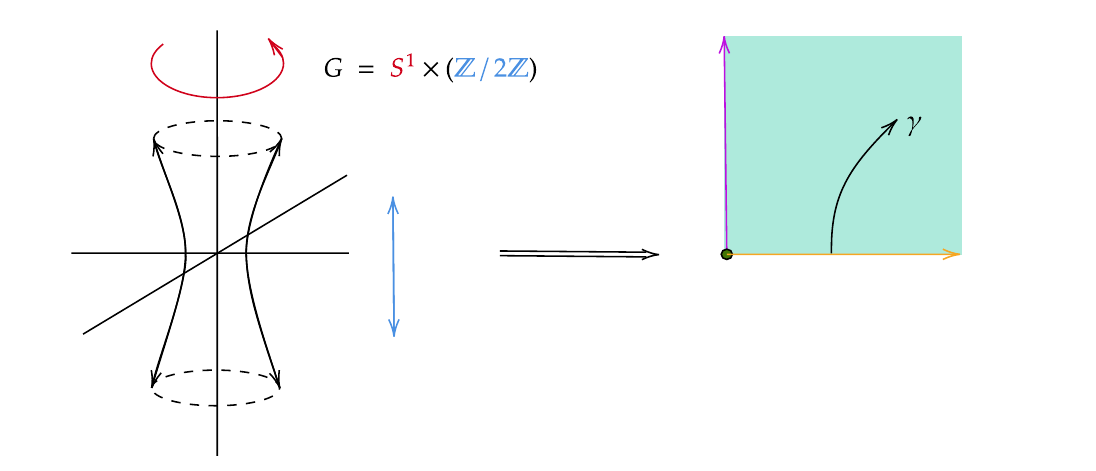}
\caption{Example of $S^1 \times (\Z/2\Z) \curvearrowright \R^3$ and $\gamma$, a geodesic on the fundamental domain corresponding to a catenoid in $\R^3$.}
\label{fig:action.example}
\end{figure}

\item Consider the following $S^1$ action on $S^3 = \{|z|^2 + |w|^2 = 1\} \subseteq \C^2$
\[
(z,w) \to (e^{i p \theta} z, e^{iq \theta} w)
\]
where $p,q > 1$ and $\text{gcd}(p,q) = 1$. The action is free except for points of the form $(z, 0)$ and $(0,w)$, which form two exceptional orbits, corresponding to orbifold points in the resulting $S^2$ quotient. See figure \ref{fig:spindle} below and we refer to Orlik \cite{orlik2006seifert}, Scott \cite{scott1983geometries} for the details of this construction.
\begin{figure}[h!]
\centering
\includegraphics[scale=0.75]{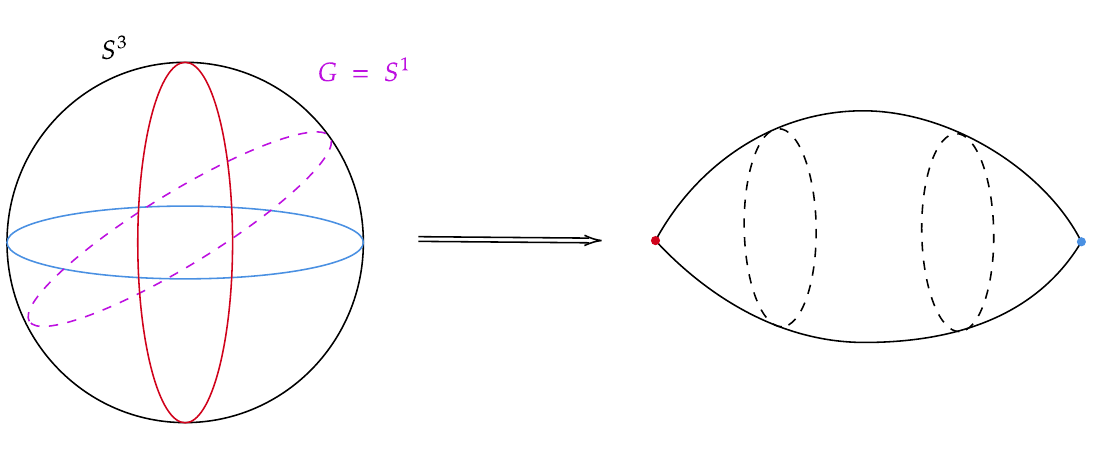}
\caption{Example of $S^1 \curvearrowright S^3$ to form an orbifold quotient $S^2$.}
\label{fig:spindle}
\end{figure}

\item Consider  $\R^8 = \R^4 \times \R^4 = \{(r, \theta), (\rho, \phi) \; | \; r, \rho \in \R^{\geq 0}, \;\; \theta, \phi \in S^3\}$ and the natural $S^3 \times S^3$ action given by 
\begin{align*}
(\theta_0, \phi_0) & \in S^3 \times S^3 \\
(\theta_0, \phi_0) \cdot ((r, \theta), (\rho, \phi)) &= ((r, (\theta_0 \cdot \theta)), (\rho, (\phi_0 \cdot \phi))
\end{align*}
where $\theta_0 \cdot \theta$ and $\phi_0 \cdot \phi$ are computed with the standard Lie group actions on $S^3$. The principal orbits correspond to $(r, \rho)$ such that $r, \rho > 0$. Singular orbits have at least one of $r = 0$ or $\rho = 0$, and there are no exceptional orbits. \nl \nl
With this action, $\R^8 / (S^3 \times S^3) \cong \R^{\geq 0} \times \R^{\geq 0}$, i.e. the quotient is isomorphic to the closed upper right quadrant. Moreover, the Hsiang--Lawson metric is given by $\overline{g} = 4 \pi^4 r^3 \rho^3 (dr^2 + d \rho^2)$. The Simon's cone is the lift of the straight line geodesic given by $r = \rho$ in this quotient space, and hence a cohomogeneity $2$ minimal hypersurface, with singularity set of dimension exactly equal to $n - 7 = 0$ at the origin. See also \cite[Ex 1.7]{HsiangLawson} which discusses the Simon's Cone as an equivariant minimal surface, with the related action of $G = SO(4) \times SO(4)$ on $\R^8$.

\end{enumerate}

\subsection{Stable Estimates for Allen--Cahn with drift} \label{section.WW.drift}
In this section, we prove estimates for the enhanced second fundamental form for stable solutions to the Allen--Cahn equation with drift (see Theorem \ref{thm.drift.stability}).
\begin{theorem*} 
Let $(U, g)$ be a smooth open two dimensional manifold with boundary. Suppose that $u_{\eps_i}$ are a sequence of stable solutions in $U$ with respect to the energy functional 
\begin{equation} \label{eqn.energy.weighted}
E_{\eps}(u, U, \mathcal{V}, g): =\int_U \mathcal{V} \left( \frac{\eps}{2} |\n^g u|^2 + \frac{W(u)}{\eps} \right) dV_g
\end{equation}
let $B = \frac{|\n^2 u|^2 - |\n |\n u||^2}{|\n u|^2} $ denote the enhanced second fundamental form and fix $0 < \beta < 1$. There exists $C, \eps_0 > 0$ such that if $\eps < \eps_0$ and 
\[
|B(x)| \leq C \qquad \forall x \in U \cap \{|u| \leq 1 - \beta \}
\]
then for all $U' \subset \subset U$
\begin{equation} \label{eqn.stability.estimate}
|B(x)| \leq C_* \eps^{1/7} \qquad \forall x \in U' \cap \{|u| \leq 1 - \beta\}
\end{equation}
where $C_*$, $\eps_0$ depend on $g$, $\text{dist}(\p U, U')$ but not $\eps$.
\end{theorem*}
To prove Theorem \ref{thm.drift.stability}, we make a slight adaptation of the original proof of Wang--Wei \cite{wang2019finite}[Theorem 3.7]. We note that the adaptation of said theorem to the Riemannian setting was done by Mantoulidis \cite{mantoulidis2021allen}[Theorem 4.13], in which he carefully checks that the change to the Riemannian setting adds higher order error terms which do not affect the estimate \eqref{eqn.stability.estimate}. We follow the same strategy by showing that the drift terms are higher order as well.
\begin{proof}
We note the necessary adjustments section by section of Wang--Wei \cite{wang2019finite}. Throughout this section, we will use $A \lesssim B$ to mean that $A \leq K \cdot B$ where $K$ is a constant independent of any other relevant parameters in context, unless noted otherwise.
\begin{center}
\textbf{Adjustments to \cite{wang2019finite}[\S 7]} 
\end{center}
The results of this section are identical as they rely on blow up arguments, and as remarked in the introduction \S \ref{section.introduction}, a sequence of solutions, $\{\ov{u}_{\eps_i}\}$, to equation \eqref{eqn.weak.one} blow up to a solution to the Sine--Gordon equation on $\R^2$.
\begin{center}
\textbf{Adjustments to \cite{wang2019finite}[\S 8]} 
\end{center}
The same adjustments from Mantoulidis \cite{mantoulidis2021allen}[\S Appendix C], ``Adjustments to Section 8" are sufficient. We note that under rescaling of $(B_r(0), g) \to (B_{\eps^{-1} r}(0), g_{\eps})$ and the pulled back functions, $u = \ov{u}(\eps x)$, and $R = \ln(\mathcal{V}(\eps x))$. 
\begin{equation} \label{eqn.rescaled.AC.drfit}
\Delta_{g_{\eps}} u + \langle \n^{g_{\eps}} u, \n^{g_{\eps}} R \rangle - W'(u) = 0
\end{equation}
By abuse of notation, we will label $g_{\eps}$ as $g$ from here on, and in fermi coordinates about some curve $\gamma$, we let $z$ denote the normal coordinate and $y$ the tangential coordinate. In these coordinates, we have 
\[
g = \begin{pmatrix}
1 & 0 \\
0 & g^{yy}
\end{pmatrix}
\]
See \cite[Equation C.2]{mantoulidis2021allen} for corresponding curvature estimates. We record for future reference that in these rescaled coordinates,
\begin{align} \label{eqn.R.bounds}
||\p_{\alpha} R||_{\infty} &\lesssim \eps^{|\alpha|} \\ \nonumber
||\p_{\alpha} g^{yy}||_{\infty} & \lesssim \eps^{|\alpha|}
\end{align}
where $\p_{\alpha}$ denotes any combination of $|\alpha|$ copies of $\p_z$ and $\p_y$.
\begin{center}
\textbf{Adjustments to \cite{wang2019finite}[\S 9]} 
\end{center}
Following the notation of Wang--Wei \cite{wang2019finite}[\S 9] exactly, we plug in the ansatz of $\phi:= u - g_*$, where $u$ now solves equation \eqref{eqn.rescaled.AC.drfit}, and compute in Fermi coordinates with respect to $\Gamma_{\alpha}$. As in Wang--Wei \cite{wang2019finite}[\S 9],
\begin{align*}
\Delta g_{\alpha} &= (\p_z^2 - H_z \p_z + \Delta_z) g_{\alpha} \\
&= g_{\alpha}'' - (-1)^{\alpha - 1} g_{\alpha}' H^{\alpha} - (-1)^{\alpha - 1} g_{\alpha}' \Delta_z h_{\alpha} + g_{\alpha}'' |\n h_{\alpha}|^2 
\end{align*}
where we have notated $\Delta_z = g^{ss}(s,z) \p_s^2$. We now compute the drift term as 
\begin{align*}
\langle \n R, \n g_{\alpha} \rangle &= (-1)^{\alpha-1}\left( \p_z(R)  - g^{yy} \p_y(R) \p_y h_{\alpha} \right) g_{\alpha}' \\
&= (-1)^{\alpha-1} \p_z(R) g_{\alpha}' + (-1)^{\alpha}\langle \n^z R, \n^z h \rangle g_{\alpha}' 
\end{align*}
where $\n^z$ denotes the tangential gradient along the level set of $\text{dist} = z$. We now conclude the following revised version of \cite{wang2017some}[Equation 9.4]:
\begin{align} \label{eqn.phi.prelim.a}
\Delta_g \phi + \langle \n R, \n \phi \rangle &= W''(g_*)\phi + \mathcal{R}(\phi) + \left[W'(g_*) - \sum_{\beta = 1}W'(g_{\beta}) \right] \\ \nonumber
& \quad (-1)^{\alpha} g_{\alpha}' [H^{\alpha} + \Delta_{z,R} h_{\alpha} - \p_z(R)] - g_{\alpha}''|\nabla_z h_{\alpha}|^2 \\ \nonumber 
& \quad - \sum_{\beta \neq \alpha} \left[ (-1)^{\beta}g_{\beta}' \mathcal{R}_{\beta,1} + g_{\beta}'' \mathcal{R}_{\beta,2} \right] - \sum_{\beta} \xi_{\beta}
\end{align}
where 
\begin{align*}
\langle \n R, \n \phi \rangle &= g^{yy} R_y \phi_y + R_z \phi_z \\
\Delta_{z,R} h &:= \Delta_z h - \langle \n^z R, \n^z h \rangle \\
\mathcal{R}_{\beta,1} &:= (H^{\beta} - \p_{z^{\beta}}(R)) + \Delta_{z,R} h_{\beta} \\
\mathcal{R}_{\beta,2} &:= |\nabla_z h_{\beta}|^2 \\
R(\phi) &:= W'(g_* + \phi) - W'(g_*) - W''(g_*)\phi = O(\phi^2)
\end{align*}
where $z^{\beta}$ denotes the fermi coordinate with respect to $\Gamma_{\beta}$. Sometimes we will suppress the $\beta$ notation when it is implicit. Note that $\Delta_{z,R}$ denotes the surface drift Laplacian on $\gamma_z = \{p \in M \; | \; \text{dist}(p, \gamma) = z\}$ with respect to the drift function, $R \Big|_{\gamma_z}$. \nl 
\indent We remark that for each $\alpha$,
\[
H^{\alpha}(s) - \p_{z}(R)(y,0) =  H^{\alpha, \tilde{g}}
\]
where $\tilde{g} = e^{R(y,z) - R(y,0)} g = \frac{\mathcal{V}^{\alpha}(\eps y, \eps z)}{\mathcal{V}^{\alpha}(\eps y, 0)} g$ and $\mathcal{V}^{\alpha}(\tilde{y}, \tilde{z})$ denotes the fiber volume function in fermi coordinates about $\Gamma_{\alpha}$. We also remark that by equation Wang--Wei \cite[Equation 8.1]{wang2019finite} (cf. Mantoulidis \cite[Lemma 4.5]{mantoulidis2021allen} after rescaling):
\begin{equation} \label{eqn.mc.easy.bound}
H^{\alpha, \tilde{g}} \lesssim \eps \qquad \forall \alpha	
\end{equation}
And derivatives in the $y$ direction will decrease the upper bound in terms of powers of $\epsilon$. It is thus natural to carry out the analysis with respect to $\tilde{g}$ - however, the conformal invariance of the Laplacian in $2$ dimensions leaves the analysis unchanged, and we will not use this alternative perspective on $H^{\alpha} - \p_z R$ until the adjustments to \S 20. Thus, we opt to work with the metric $g$, and recreate the analysis of Wang--Wei with 
\begin{align} \label{eqn.MC.replace}
H^{\alpha} + \Delta_z h_{\alpha} &\to (H^{\alpha} - \p_{z^{\alpha}}(R)) + \Delta_{z,R} h \\ \label{eqn.laplace.replace}
\Delta_g \phi &\to \Delta_g \phi + \langle \n R, \n \phi \rangle
\end{align}
We remark that from here on, when we say "the analogous equation" or "appropriate modifications," this means symbolically making the changes denoted in equations \eqref{eqn.MC.replace} \eqref{eqn.laplace.replace} to the estimates and equations that Wang--Wei originally produced.
\begin{center}
\textbf{Adjustments to \cite{wang2019finite}[\S 10]} 
\end{center}
Most of the computations in this section remain the same (see \cite[Appendix C]{mantoulidis2021allen} for the adaptation to the Riemannian setting without drift), so we simply include the bounds on the extra error terms. The analogue of \cite[Equation 10.1]{wang2019finite} is now
\begin{align} \label{eqn.new.10.1}
\int_{-\delta R}^{\delta R} g_{\alpha}'&\left( \Delta_z \phi - H_z \p_z \phi + \p_z^2 \phi \right) \\ \nonumber
&= \int_{-\delta R}^{\delta R} g_{\alpha}' \left[W'(g_* + \phi) - \sum_{\beta = 1}W'(g_{\beta}) \right] \\ \nonumber
& + \int_{-\delta R}^{\delta R}(-1)^{\alpha} (g_{\alpha}')^2 [H^{\alpha} - \p_z(R) + \Delta_{z,R} h_{\alpha}] \\ \nonumber
& + \int_{-\delta R}^{\delta R} - g_{\alpha}' g_{\alpha}''|\nabla_z h_{\alpha}|^2 \\  \nonumber
& \quad - \sum_{\beta \neq \alpha} \int_{-\delta R}^{\delta R} g_{\alpha}'\left[ (-1)^{\beta}g_{\beta}' \mathcal{R}_{\beta,1} + g_{\beta}'' \mathcal{R}_{\beta,2} \right] - \sum_{\beta} \int_{-\delta R}^{\delta R} \xi_{\beta} g_{\alpha}' \\ \nonumber
& - \int_{-\delta R}^{\delta R} g_{\alpha}' \langle \n R, \n \phi \rangle 
\end{align}
In this section, Wang--Wei isolate the term on the third line of equation \eqref{eqn.new.10.1} and produce a coarse upper bound for $H^{\alpha}(y,0) + \Delta_{0} h_{\alpha}$ using orthogonality relations. We aim to produce the analogous upper bound for $(H^{\alpha}(y,0) - \p_z(R)(y,0)) + \Delta_{0,R} h$. As such, it suffices to bound
\begin{align*}
I + II &= \int_{-\delta R}^{\delta R} (-1)^{\alpha} (g_{\alpha}')^2 [- (\p_z(R)(y,z) - \p_z(R)(y,0)) + (\langle \n_z R, \n_z h \rangle - \langle \n_0 R, \n_0 h \rangle)] \\
III &= \int_{-\delta R}^{\delta R} g_{\alpha}' \langle \n R, \n \phi \rangle
\end{align*}
We the first two terms as:
\begin{align*}
|I| & \lesssim \int_{-\delta R}^{\delta R} (g_{\alpha}')^2 z||\p_z^2 R||_{\infty}\\ 
& \lesssim ||\p_z^2 R||_{\infty} \int_{-\delta R}^{\delta R} (g_{\alpha}')^2 z \\
& \lesssim \eps^2 \\
|II| &\lesssim \int_{-\delta R}^{\delta R} (g_{\alpha}')^2 \left( |\langle \nabla_z R - \nabla_0 R, \nabla_z h \rangle| + |\langle \nabla_0 R, \nabla_z h - \nabla_0 h \rangle|\right) \\
&\lesssim ||\p_z R|| \cdot \sup_{y} ||\nabla_0 h||  \\
& \lesssim \eps^2 + \sup_{y}||\nabla^0 h||^2 
\end{align*}
having used equation \eqref{eqn.R.bounds} to bound the appropriate derivatives of $R$ and the AM-GM inequality. For $III$, we have 
\begin{align*}
III &= \int_{-\delta R}^{\delta R} g_{\alpha}' (R_z \phi_z + g^{yy} R_y \phi_y)\\
&= \int_{-\delta R}^{\delta R} g_{\alpha}' (- R_{zz} \phi + g^{yy} R_y \phi_y) \\
|III| & \lesssim \int g_{\alpha}' (\eps^2 + |\nabla_y \phi|^2) \\
& \lesssim \eps^2 + \sup_{(-6 |\log \eps|, 6 |\log \eps|)} |\nabla_y \phi(y,z)|^2
\end{align*}
We similarly bound these differences in the error terms, i.e. 
\begin{align*}
I_{\beta} + II_{\beta} &= \int_{-\delta R}^{\delta R} g_{\alpha}' g_{\beta}'[- (\p_z(R)(y,z) - \p_z(R)(y,0)) + (\langle \n_z R, \n_z h \rangle - \langle \n_0 R, \n_0 h \rangle)]
\end{align*}
which satisfies similar error estimates also dampened by the distance between sheets
\begin{align*}
|I_{\beta}| & \lesssim |d_{\beta}(y,0)| e^{-\sqrt{2} d_{\beta}(y,0)}\eps^2 \\
& \lesssim \eps^2 \\
|II_{\beta}| &\lesssim \eps |d_{\beta}(y,0)| e^{-\sqrt{2} d_{\beta}(y,0)} \sup_{B_{\eps^{1/3}}(y)} |\nabla_0 h| \\
& \lesssim \eps^2 + |d_{\beta}(y,0)| e^{-\sqrt{2} d_{\beta}(y,0)} \sup_{B_{\eps^{1/3}}(y)} |\nabla_0 h|^2
\end{align*}
Doing the remaining analysis of Wang--Wei (and also Mantoulidis \cite[Appendix C]{mantoulidis2021allen}), we conclude the same equation of Wang--Wei \cite[Equation 10.2]{wang2019finite} with an extra $O(\sup_y|\nabla_0 h_{\alpha}|^2)$ term.
\begin{align} \label{eqn.10.2.modified}
(H^{\alpha}(y,0) &- \p_z(R)(y,0)) + \Delta_{0,R}(h_{\alpha})(y,0)  \\ \nonumber
&= \frac{4}{\sigma_0} \left[ A_{(-1)}^2 e^{-\sqrt{2} d_{\alpha-1}(y,0)} - A_{(-1)^{\alpha-1}}^2 e^{\sqrt{2} d_{\alpha+1}(y,0)}\right] + O(\eps^2) \\ \nonumber
& + O(|h_{\alpha}(y)| + |h_{\alpha-1}(\Pi_{\alpha-1}(y,z))| + \eps^{1/3})e^{-\sqrt{2} d_{\alpha-1}(y,0)} \\ \nonumber
& + O(|h_{\alpha}(y)| + |h_{\alpha+1}(\Pi_{\alpha+1}(y,z))| + \eps^{1/3})e^{-\sqrt{2} d_{\alpha+1}(y,0)} \\ \nonumber
& + O(e^{\frac{-3\sqrt{2}}{2} d_{\alpha-1}(y,0)} + e^{\frac{-3\sqrt{2}}{2} d_{\alpha+1}(y,0)}) \\ \nonumber
& + O(e^{-\sqrt{2} d_{\alpha-2}(y,0)} + e^{\sqrt{2} d_{\alpha+2}(y,0)})  \\ \nonumber
& + \sum_{\beta \neq \alpha} |d_{\beta}(y,0)| e^{-\sqrt{2} |d_{\beta}(y,0)|}  \\ \nonumber
& \qquad \cdot \left[ \sup_{B_{\eps^{1/3}(y)}} |H^{\beta}(y,0) - \p_z(R)(y,0) + \Delta_{0,R}^{\beta} h_{\beta}| + \sup_{B_{\eps^{1/3}}(y)} |\nabla h_{\beta}|^2 \right] \\ \nonumber
& + \sup_{(-6 |\log\eps|, 6 |\log \eps|)} \left( |\nabla_y^2 \phi (y,z)|^2 + |\nabla_y \phi (y,z)|^2 + |\phi(y,z)|^2 \right) \\ \nonumber
& + \sup_{B_{\eps^{1/3}}(y)} |\nabla_0 h|^2
\end{align}
Again, all lines in equation \eqref{eqn.10.2.modified} are identical to Wang--Wei \cite[Equation 10.2]{wang2019finite} except for the replacement of $H + \Delta_0 h \to H - \p_z R + \Delta_{0,R} h$ and the last line. \nl 
\indent However, for the purposes of deriving their Lemma 10.1 (specifically \cite[Equation 10.3]{wang2019finite}), we can use \cite[Equation 9.8]{wang2019finite}:
\[
\sup_y|\nabla_0 h_{\alpha}|^2 \lesssim \sup_y|\nabla_0 \phi(y,0)|^2 + o(e^{-2\sqrt{2} D_{\alpha}(y)})
\]
Combining this with \eqref{eqn.10.2.modified}, we conclude the analogous bound of Wang--Wei equation 10.3 
\begin{align*}
\sup_{B_r} |H^{\alpha}(y,0) &- \p_z(R)(y,0) + \Delta_{0,R} h_{\alpha}(y)|  \\
& \lesssim \sup_{B_{r+1}} e^{-\sqrt{2} D_{\alpha}} + \eps^2 + ||\phi||^2_{C^{2,\theta}(\mathcal{D}_{r+1})} \\
& + \sum_{\beta \neq \alpha} \sup_{B_{r+1}} \left[ |H^{\beta}(y,0) - \p_{z}(R)(y,0) + \Delta_{0,R}^{\beta} h_{\beta}|^2 + e^{-2\sqrt{2} D_{\beta}} \right]
\end{align*}

\begin{center}
\textbf{Adjustments to \cite{wang2019finite}[\S 11]} 
\end{center}
Wang--Wei's \cite[Prop 11.1]{wang2019finite} relies on their Lemma 11.2, which in turn relies on \cite[Equation 9.4]{wang2019finite}. Thus we must make the same chain of deductions, arguing with our equation \eqref{eqn.phi.prelim.a} instead. However, we note that 
\[
\Big|g_{\alpha}'[H^{\alpha} - \p_z(R) + \Delta_{z,R} h_{\alpha}]\Big| \lesssim e^{-cL} |H^{\alpha}(y,0) - \p_z(R)(y,0) + \Delta_{0,R} h_{\alpha}| + \eps^2 + |\nabla_0 h_{\alpha}|^2 + |\nabla_0^2 h_{\alpha}|^2
\]
This, along with the fact that 
\begin{align*}
|g_{\beta}'[H^{\beta} - \p_z R + \Delta_{z,R}h_{\beta}]| & \lesssim |g_{\beta}'|^2 + |H^{\beta} - \p_z R + \Delta_{z,R}h_{\beta}|^2 \\
& \lesssim e^{-\sqrt{2} d_{\beta}(y,z)} + |H^{\beta} - \p_z R + \Delta_{z,R}h_{\beta}|^2
\end{align*}
and applying \cite{wangsecond}[Lemma 3.6] we conclude a similar error bound of 
\begin{lemma} \label{lem.11.2.replace}
In $\mathcal{N}_{\alpha}^2(r)$, we have 
\[
\Delta_z \phi - H^{\alpha} \p_z \phi + \p_z^2 \phi + \langle \n R, \n \phi \rangle = [2 + O(e^{-cL})] \phi + E_{\alpha}^2
\]
where
\begin{align*}
|E_{\alpha}^2(y,z)| & \lesssim \eps^2 + e^{-\sqrt{2} D_{\alpha}(y)} + |\n_0^2 h_{\alpha}(y)|^2 + |\n_0 h_{\alpha}(y)|^2 \\
& + e^{-cL} |H^{\alpha}(y,0) -\p_z(R)(y,0) + \Delta_{0,R} h_{\alpha}(y)| \\
& + \sum_{\beta \neq \alpha} \sup_{B_{\eps^{1/3}}(y)} \left[ |H^{\beta}(y,0) - \p_{z}(R)(y,0) + \Delta_{0,R}^{\beta} h_{\beta}|^2 + |\nabla h_{\beta}|^4 + |\n^2 h_{\beta}|^4 \right]
\end{align*}
\end{lemma}
\noindent A similar analysis gives an analogous version of Wang--Wei equation \cite[Equation 11.5]{wang2019finite}. We note that
\begin{align*}
(\Delta_g (c_{\alpha}(y) g_{\alpha}') + \langle \n R, \n (c_{\alpha} g_{\alpha}') \rangle) &= \Delta_{0,R}( c_{\alpha}(y)) g_{\alpha}' \\
& + (- c_{\alpha} H_t + c_{\alpha}\langle \n^z R, \n^z h \rangle + \p_z(R) c_{\alpha}) g_{\alpha}'' + c_{\alpha} g_{\alpha}''' 
\end{align*}
And thus, we conclude:
\begin{align*}
\Delta_g \phi + \langle \n R, \n \phi \rangle &= W''(g_{\alpha}) \phi + \tilde{c}_{\alpha}(y) g_{\alpha}' + \tilde{E}_{\alpha} \\
\tilde{c}_{\alpha}(y) &= (-1)^{\alpha-1} [H^{\alpha}(y,0) - \p_z(R)(y,0) + \Delta_{0,R} h_{\alpha}] - [\Delta_{0,R} c_{\alpha}(y)]
\end{align*}
with the exact same form of the error term, noting that 
\[
|(- c_{\alpha} H_t + c_{\alpha}\langle \n^z R, \n^z h \rangle + \p_z(R) c_{\alpha}) g_{\alpha}''| \lesssim \eps [|c_{\alpha}| + |\nabla c_{\alpha}|] e^{-\sqrt{2} |z|}
\]
The remainder of this section goes through the same, though the proof of lemma 11.6 must be adjusted to account for the drift terms. We sketch the necessary modifications as follows:
\begin{align*}
\int_{\R} \phi \Delta_z \phi + H^{\alpha} \phi_z \phi + \phi \phi_{zz} &= \int_{\R} g^{yy} R_y \phi_y \phi + R_z \phi_z \phi + \tilde{E} \phi  
\end{align*}
Integrating by parts and applying \cite[Theorem A.2]{wang2019finite}, we have 
\begin{align*}
\int_{\R} \phi \Delta_z \phi - g^{yy} R_y \phi_y \phi &= \int_{\R} |\p_z \phi|^2 + W''(g_{\alpha}) \phi^2 + \tilde{E}_{\alpha} \phi + \frac{1}{2} \frac{\p (H^{\alpha} + R_t)}{\p z} \phi^2 \\
& \geq \frac{3 \mu}{4} \int_{\R} \phi^2 - C \int_{\R} \tilde{E}_{\alpha}^2
\end{align*}
Following the same manipulations as in the proof of \cite[Lemma 11.6]{wang2019finite} but the drift term, we have
\begin{align*}
\frac{1}{2} (\Delta_0 + (\p_y R) \p_y ) \int_{\R} \phi^2 & \geq \frac{\mu}{2} \int_{\R} \phi^2 - C \int_{\R} \ti{E}_{\alpha}^2 \\
& \quad - C \eps^2 \int_{\R} z^2 \left( |\n_y^2 \phi(y,z)|^2 + |\n_y \phi|^2 \right)
\end{align*}
noting that the Agmon-type exponential decay estimate (see e.g. \cite{agmonlectures}) still works for the operator $\Delta_0 + (\p_y R) \p_y = \p_y^2 + (\p_y R) \p_y$, we conclude \cite[Equation 11.6]{wang2019finite}. \nl 
\indent The rest of the section now follows with the appropriate modifications. 
\begin{center}
\textbf{Adjustments to \cite{wang2019finite}[\S 12]} 
\end{center}
With the replacements of Wang--Wei equation 9.4 by equation \eqref{eqn.phi.prelim.a}, we conclude the analogous version of their Lemma 12.1
\begin{lemma} \label{lemma.12.1.replace}
For any $(y,z) \in \mathcal{M}_{\alpha}^0$
\begin{align*}
||\Delta \phi &+ \langle \n R, \n \phi \rangle - W''(g_*) \phi||_{C^{\theta}(B_{2/3}(y,z))} \\
&\lesssim \eps^2 + \sup_{B_1(y)} e^{-\sqrt{2} D_{\alpha}} + ||\phi||_{C^{2,\theta}(B_1(y,z))}^2 \\
& + e^{-\sqrt{2} |z|} ||H^{\alpha} - \p_{z^{\alpha}} R + \Delta_{0,R} h_{\alpha}||_{C^{\theta}(B_1(y,0))} \\
& + \sum_{\beta \neq \alpha} e^{-\sqrt{2} |d_{\beta}(y,z)|}(||\phi||_{C^{2,\theta}(B_2^{\beta}(y,0))}^2+ \sup_{B_2(y)} e^{-2\sqrt{2} D_{\beta}}) \\
& + \sum_{\beta \neq \alpha} e^{-\sqrt{2} |d_{\beta}(y,z)|} ||H^{\beta} - \p_{z^{\beta}} R + \Delta_{0,R}^{\beta} h_{\beta}||_{C^{\theta}(B_2^{\beta}(y,0))}
\end{align*}
\end{lemma}
The analogous Schauder estimates lead to a version of Wang--Wei \cite[Equation 12.2]{wang2019finite} with $H + \Delta h \to H - \p_z R + \Delta_{0,R} h$. \nl
\indent We now replace the use of \cite[Equation 10.1]{wang2019finite} in \cite[Lemma 12.2]{wang2019finite} by \eqref{eqn.new.10.1}. The analogous analysis holds except we include an error term from the drift Laplacian on $\phi$. As such, it suffices to bound 
\begin{align*}
\Big|\int_{-\delta R}^{\delta R} \langle \n R, \n \phi \rangle g_{\alpha}'| & \lesssim \eps^2 + ||\phi||_{C^1(B_2(y) \times (-6 |\log \eps|, 6 |\log \eps|)}^2
\end{align*}
which is compatible with the bound on $||H^{\alpha} + \Delta_0 h_{\alpha}||$ at the end of the proof of \cite[Lemma 12.2]{wang2019finite}.
\begin{center}
\textbf{Adjustments to \cite{wang2019finite}[\S 13]} 
\end{center}
We adapt the improved horizontal estimates. \cite[Equation 13.1]{wang2019finite} must now be replaced by the result of applying $\p_y$ to equation \eqref{eqn.phi.prelim.a}, which we record as:
\begin{align*}
\Delta_z \phi_y + \p_z^2 \phi_y + \langle \n R, \n \phi_y \rangle &= W''(g_{\alpha}) \phi_y - (-1)^{\alpha} g_{\alpha}'[H^{\alpha}_{yz}(y,0) - R_{yz} + \Delta_{0,R} h_{\alpha,y}] + \tilde{E}_i
\end{align*}
where $\tilde{E}_i = E_i + F_i$ for $E_i$ as in \cite[Equation 13.1]{wang2019finite} with the analogous replacements, and $F_i$ has the following terms coming from the drift Laplacian
\begin{align*}
F_i &= g^{yy}_y R_y \phi_y + h^{yy} R_{yy} \phi_y + R_{zy} \phi_z  \\
& + (-1)^{\alpha} g_{\alpha}'[g^{yy}_y h_{\alpha, yy} - g^{yy}_y R_y h_{\alpha,y} - g^{yy} R_{yy} h_{\alpha, y}] \\
& - \sum_{\beta \neq \alpha} (-1)^{\beta}g_{\beta}'[g^{yy}_y h_{\beta, yy} - g^{yy}_y R_y h_{\beta,y} - g^{yy} R_{yy} h_{\beta, y}]
\end{align*}
Using the Cauchy inequality as well as our a priori bounds on $R$ via equation \eqref{eqn.R.bounds} we have
\begin{align*}
||F_i||_{C^{\theta}(M_{\alpha}^2(r))} &\lesssim \eps^2 + ||h||_{C^{2,\theta}(M_{\alpha}^2(r))}^2 + ||\phi||_{C^{2,\theta}(\mathcal{D}(r+1))}^2 \\
& \lesssim \eps^2 + \sup_{B_{r+1}}e^{-2\sqrt{2} D_{\alpha}(y)} + ||\phi||_{C^{2,\theta}(\mathcal{D}(r+1))}^2
\end{align*}
having used \cite[Equation 9.9]{wang2019finite}. Thus, up to a constant, $\tilde{E}_i$ and $E_i$ satisfy the same upper bounds as listed in \cite[Lemma 13.1]{wang2019finite}. \nl 
\indent The same analysis now leads to the analogous bound on $||H^{\alpha}_y - R_{yz} + \Delta_{0,R} h_{\alpha, y}||$ as in Wang--Wei \cite[Equation 13.3]{wang2019finite}, and the rest of the section follows.
\begin{center}
\textbf{Adjustments to \cite{wang2019finite}[\S 14]} 
\end{center}
This section follows after replacing $H^{\alpha} \to H^{\alpha} - \p_z R$ and $\Delta_0 h_{\alpha} \to \Delta_{0,R} h_{\alpha}$.
\begin{center}
\textbf{Adjustments to \cite{wang2019finite}[\S 17]} 
\end{center}
Combined with the modifications from Mantoulidis \cite[Appendix C]{mantoulidis2021allen}, we replace the linearized Allen--Cahn equation with the linearized Allen--Cahn equation with drift, so \cite[Equation 17.1]{wang2019finite} becomes
\begin{equation} \label{eqn.17.1.replace}
\eps \Delta v + \eps \langle \n \ln(\mathcal{V}), \n v \rangle = \frac{1}{\eps} W''(u_{\eps}) v
\end{equation}
Lemma 17.2 remains unchanged. As in \cite{mantoulidis2021allen}, the function $\varphi_{\eps} = \mathbf{1}_{\tilde{D}_{\alpha}} \frac{\p u_{\eps}}{\p x_2}$ is a solution to equation \eqref{eqn.17.1.replace} but with an error, i.e. 
\begin{equation} \label{eqn.17.1.replace.almost}
\eps \Delta \varphi_{\eps} + \eps \langle \n \ln(\mathcal{V}), \n \varphi_{\eps} \rangle = \frac{1}{\eps} W''(u_{\eps}) \varphi_{\eps} + O(\eps)
\end{equation}
for the same bump function, $\eta$, multiplying \eqref{eqn.17.1.replace.almost} by $\mathcal{V} \varphi_{\eps} \eta^2$ and integrating by parts yields
\begin{align*}
\int_{B_{1/100}(x_{\eps})} \mathcal{V}\Big[ \eps |\nabla (\eta \varphi_{\eps})|^2 &+ \eps^{-1} W''(u_{\eps}) \eta^2 \varphi_{\eps}^2 \Big]  \\
& \leq K \eps  + C \int \mathcal{V} \eps \varphi_{\eps}^2 \left[|\nabla \eta|^2 + \langle \n \mathcal{V}, \n \mathcal{\eta}\rangle \eta \right] \\
& \leq K \eps + C \eps \int \varphi_{\eps}^2 \\
& \leq C
\end{align*}
In the last line, we are using that one can also modify \cite[Lemma 17.3]{wang2019finite} to prove the result in the drift setting, using the same technique but with equation \eqref{eqn.17.1.replace.almost}. \nl 
\indent In section 17.2, we aim to solve
\[
\begin{cases} 
\eps \Delta \tilde{\varphi}_{\eps} + \eps \langle \n \ln(\mathcal{V}), \n \ti{\varphi}_{\eps}\rangle = \eps^{-1} W''(u_{\eps}) \varphi_{\eps} & \text{ in } \Omega_{\alpha, \eps} \\
\ti{\varphi}_{\eps} = \varphi_{\eps} & \text{ on } \partial \Omega_{\alpha, \eps}
\end{cases}
\]
which by stability for the weighted energy \eqref{eqn.energy.weighted}, exists and is unique. Following the computations in Mantoulidis \cite[Appendix C]{mantoulidis2021allen}, we have that 
\begin{align*}
\int_{\Omega_{\alpha, \eps}} \mathcal{V} \Big( \eps \langle \n \varphi_{\eps}, & \n \ti{\varphi}_{\eps} \rangle + \eps^{-1} W''(u_{\eps}) \varphi_{\eps} \ti{\varphi}_{\eps} \Big) \\
&= \int_{\p \Omega_{\alpha, \eps}} \eps \mathcal{V} \varphi_{\eps} \nu_{\p \Omega_{\alpha, \eps}}(\ti{\varphi}_{\eps}) + \int_{\Omega_{\alpha, \eps}} \varphi_{\eps}\left( -\eps \mathcal{V} \Delta(\ti{\varphi}_{\eps}) - \eps \langle \n \mathcal{V}, \n \ti{\varphi}_{\eps} \rangle + \eps^{-1} W''(u_{\eps}) \ti{\varphi}_{\eps} \right) \\
&= \int_{\p \Omega_{\alpha, \eps}} \eps \mathcal{V} \ti{\varphi}_{\eps} \nu_{\p \Omega_{\alpha, \eps}}(\ti{\varphi}_{\eps}) \\
&= \int_{\Omega_{\alpha, \eps}} \eps \mathcal{V} \ti{\varphi}_{\eps} (\Delta \ti{\varphi}_{\eps} + \langle \n \mathcal{V}, \n \ti{\varphi}_{\eps} \rangle) + \mathcal{V} \eps |\n \ti{\varphi}_{\eps}|^2 \\
&= \int_{\Omega_{\alpha, \eps}} \mathcal{V} \left(\eps |\n \ti{\varphi}_{\eps}|^2 + \eps^{-1} W''(u_{\eps}) \ti{\varphi}_{\eps}^2\right)
\end{align*}
From here, the rest of the section is the same. In particular, equation 17.5 also holds in our setting, with a potentially slightly worse exponent, as introduced by both the drift term and also the Riemannian metric (cf. \cite[Appendix C]{mantoulidis2021allen}). 
\begin{center}
\textbf{Adjustments to \cite{wang2019finite}[\S 18]} 
\end{center}
This section translates to the drift setting making the appropriate modifications and noting that the left hand side of equation 18.6 should be $H^{\alpha}(y,0) - \p_z(R)(y,0)$.
\begin{center}
\textbf{Adjustments to \cite{wang2019finite}[\S 19]} 
\end{center}
Again, the stability condition is now changed to 
\[
\int_{C_{5R/6}} \mathcal{V}(\eps y, \eps z) \left( |\n \varphi|^2 + W''(u) \varphi^2 \right) \geq 0
\]
As in Mantoulidis, there is no explicit formula for $\lambda(y,z)$ in the Riemannian setting, but it satisfies the relevant upper bounds to be of no issue. \nl 
\indent "The Horizontal part" (Section 19.1) is analogous in the drift setting yielding the same integrals in the upper bound but with factors of $\mathcal{V}$ included in the integrand. \nl 
\indent "The Vertical Part" (Section 19.2) requires an integration by parts in the expansion of $\int \varphi_z^2 \lambda dz dy$, which will affect the factor of $\mathcal{V}$ in the integrand. Thus, we can produce the same bounds on $\int \mathcal{V} \varphi_z^2 \lambda dz dy$, modulo an extra error term of the form:
\[
R = - \int_{-5R/6}^{5R/6} \eta(y)^2 \int_{-\delta R}^{\delta R} \mathcal{V}_z g_{\alpha}' g_{\alpha}'' \chi^2 \lambda 
\]
Integrating by parts in $t$ and using that $|\nabla \chi| \lesssim L^{-1}$ and $|\nabla \lambda| = O(\eps)$, we have 
\begin{align*}
R &= - \frac{1}{2}\int_{-5R/6}^{5R/6} \eta(y)^2 \int_{-\delta R}^{\delta R} \mathcal{V}_{zz} (g_{\alpha}')^2 \chi^2 \lambda  \\
& \qquad + \mathcal{V}_z (g_{\alpha}')^2 2 \chi_z \chi \lambda + \mathcal{V}_z (g_{\alpha}')^2 2 \chi^2 \lambda_z \\
& \lesssim \left(\eps^2 + \frac{\eps}{L}\right) \int_{-5R/6}^{5R/6} \eta(y)^2 \left( e^{-2\sqrt{2} \rho_{\alpha}^+(y)} + e^{2\sqrt{2} \rho_{\alpha}^-(y)} \right)
\end{align*}
which is compatible with the final error bound in \cite[Eq 19.2]{wang2019finite}. The integration by parts in the ``third term" of the vertical part also yields a higher order error, i.e.
\begin{align*}
\Big|\int_{-5R/6}^{5R/6} \eta(y)^2 \int_{-\delta R}^{\delta R} \mathcal{V} g_{\alpha}' g_{\alpha}'' \chi^2 \lambda_z\Big|  & \lesssim \left(\eps^2 + \frac{\eps}{L}\right) \int_{-5R/6}^{5R/6} \eta(y)^2 \left( e^{-2\sqrt{2} \rho_{\alpha}^+(y)} + e^{2\sqrt{2} \rho_{\alpha}^-(y)} \right)
\end{align*}
Now using that $\mathcal{V}$ is bounded, the replacement of \cite[Eq 19.2]{wang2019finite} becomes
\begin{align} \label{eqn.19.2.replace}
0 & \leq C \int_{-5R/6}^{5R/6} \eta'(y)^2 + C \eps^{2 - 2\sigma} \int_{-5R/6}^{5R} \eta(y)^2 \\ \nonumber
& + C \left( \frac{1}{L} + \eps \right) \int_{-5R/6}^{5R/6} \eta(y)^2 \left( e^{-2\sqrt{2} \rho_{\alpha}^+(y)} + e^{2\sqrt{2} \rho_{\alpha}^-(y)} \right) \\ \nonumber
& + \int_{-5R/6}^{5R/6} \int_{-\delta R}^{\delta R} \mathcal{V} \left(W''(u) - W''(g_{\alpha})\right) |g_{\alpha}'|^2 \chi^2 \lambda dz dy
\end{align}
Much of ``The Interaction Part", i.e. \cite[\S 19.3]{wang2019finite} holds the same though \cite[Equation 19.3]{wang2019finite} must be modified to include the differentiated drift term: 
\[
\frac{\p}{\p z} \langle \n R, \n \phi \rangle,
\]
in addition to the differentiated terms coming from replacing $H^{\alpha} + \Delta_z h_{\alpha} \to H^{\alpha} - \p_z R + \Delta_{z,R} h_{\alpha}$. The goal of this section is to now show that the appropriate modification of \cite[Equation 19.5]{wang2019finite} holds, i.e.
\begin{align} \label{eqn.19.5.replace}
\int_{-5R/6}^{5R/6} \eta(y)^2 \int_{-\delta R}^{\delta R} \mathcal{V} \Big[ W''(u) &- W''(g_{\alpha})\Big] |g_{\alpha}'|^2 \chi^2 \lambda \\ \nonumber
&= -4\int_{-5R/6}^{5R/6} \mathcal{V}(y,0) \eta(y)^2 [ A_{(-1)}^2 e^{-\sqrt{2} d_{\alpha-1}(y,0)} + A_{(-1)^{\alpha}}^2 e^{\sqrt{2} d_{\alpha + 1}(y,0)}] \lambda(y,0) \\ \nonumber
& + O(\eps^{4/3}) \int_{-5R/6}^{5R/6} \eta(y)^2 \mathcal{V}(y,0).
\end{align}
Thus, it suffices to show that the new terms coming from the drift Laplacian in the analogous version of \cite[Equation 19.3]{wang2019finite} are smaller than the error threshold of $O(\eps^{3/2 - 2 \sigma})$. We remark that 
\begin{align*}
\Big|g_{\alpha}'\frac{\p}{\p z} \langle \n R, \n \phi \rangle\Big| &= |g^{yy}_z R_y \phi_y + g^{yy} R_{yz} \phi_y + g^{yy} R_y \phi_{yz} + R_{zz} \phi_z + R_z \phi_{zz}| |g_{\alpha}'| \\
& \lesssim \eps^2 ||\phi||_{C^1(\mathcal{M}_{\alpha}(r)} + \eps ||\phi||_{C^2(\mathcal{M}_{\alpha}(r)} \\
& \lesssim \eps^{2 - \sigma}
\end{align*}
having used \cite[Equation 18.4]{wang2019finite} to bound $||\phi||_{C^2} \lesssim \eps^{1-\sigma}$. We also remark
\begin{align*}
|\p_z \mathcal{V}| + \eps^{-1} |\p_z^2 \mathcal{V}| & \lesssim \eps \\
|\mathcal{V}(y,z) - \mathcal{V}(y,0)| & \lesssim \eps |z| \\
\Big|\frac{\p}{\p z} \left( H^{\alpha} - \p_z R\right) \Big| &\lesssim \eps^2 \\
\Big|\frac{\p}{\p z} \Delta_{z,R} h_{\alpha}\Big| &\leq \Big| \p_z(\Delta_z h_{\alpha})\Big| + \Big| \p_z\left( g^{yy} R_y h_{\alpha,y} \right)\Big| \\
& \lesssim \eps^{2 - \sigma} + \eps^{3 - \sigma} + \eps^{2 - \sigma} = \eps^{2-\sigma} 
\end{align*}
having also used \cite[Lemma 9.6]{wang2019finite}. The rest of this section now follows by using the drift analogue of \cite[Equation 19.3]{wang2019finite}, multiplying by $\mathcal{V}(y,z) \eta^2 g_{\alpha}' \chi^2 \lambda$ and integrating in $y,z$. \nl 
\indent Equation 19.6 now becomes 
\begin{align*}
\int_{-5R/6}^{5R/6} \mathcal{V}(y,0) &\eta(y)^2 [ A_{(-1)}^2 e^{-\sqrt{2} d_{\alpha-1}(y,0)} + A_{(-1)^{\alpha}}^2 e^{\sqrt{2} d_{\alpha + 1}(y,0)}] \lambda(y,0) \\
& \leq C \int_{-5R/6}^{5R/6} \eta'(y)^2 + C \eps^{4/3} \int_{-5R/6}^{5R/6} \eta(y)^2 dy \\
& + C \left(\frac{1}{L} + \eps \right) \int_{-5R/6}^{5R/6} \eta(y)^2 \left[ e^{-2\sqrt{2} \rho_{\alpha}^+(y)} + e^{2\sqrt{2} \rho_{\alpha}^-(y)}\right]
\end{align*}
Again noting that $\mathcal{V}(y,0)$ is bounded from above and below by positive numbers, we can change the constant slightly and conclude exactly the same result as \cite[Equation 19.6]{wang2019finite}:
\begin{align} \label{eqn.19.6}
\int_{-5R/6}^{5R/6} &\eta(y)^2 [ A_{(-1)}^2 e^{-\sqrt{2} d_{\alpha-1}(y,0)} + A_{(-1)^{\alpha}}^2 e^{\sqrt{2} d_{\alpha + 1}(y,0)}] \lambda(y,0) \\ \nonumber
& \leq C \int_{-5R/6}^{5R/6} \eta'(y)^2 + C \eps^{4/3} \int_{-5R/6}^{5R/6} \eta(y)^2 dy \\ \nonumber
& + C \left(\frac{1}{L} + \eps \right) \int_{-5R/6}^{5R/6} \eta(y)^2 \left[ e^{-2\sqrt{2} \rho_{\alpha}^+(y)} + e^{2\sqrt{2} \rho_{\alpha}^-(y)}\right]
\end{align}
The rest of the section is now exactly the same.
\begin{center}
\textbf{Adjustments to \cite{wang2019finite}[\S 20]} 
\end{center}
In this section, we use that for the conformal metric $\tilde{g} = \mathcal{V} g$, 
\[
H_{\tilde{g}} \sim H_g - \p_z(R)
\]
This follows via the formula for mean curvature under a conformal change of metric
\begin{align*}
H_{\tilde{g}} &= \mathcal{V}^{-1/2}\left(H_g - 2\p_z \ln(\mathcal{V}^{1/2})\right) \\
&= \mathcal{V}^{-1/2}\left(H_g - \frac{\p_z \mathcal{V}}{\mathcal{V}}\right)
\end{align*}
The proof of theorem 3.6 is the same following the appropriate modifications. The rest of the section follows after the adjustment to the Riemannian setting (using $\tilde{g}_{\eps} = \eps^{-2} \tilde{g}$) as described in \cite{mantoulidis2021allen}.
\end{proof}

\subsection{Monotonicity for the Allen--Cahn equation with drift}
In this section, we prove the following adaptation of \cite[Prop 3.4]{HutchinsonTonegawa}, \cite[Appendix B]{GuaracoMinmax} to the drift setting: 
\begin{lemma} \label{lem.linear.energy.growth}
Suppose that $u$ satisfies equation \eqref{eqn.weak.one} on $U \subseteq M^2$ with 
\[
E_{\eps}(u, U, \mathcal{V}, g) \leq E_0
\]
then for all $U' \subset \subset U$, there are $C, \eps_0, r_0 > 0$ depending on $E_0$, $\text{dist}(U', \p U)$, $\text{inj}(U, g)$, and $||g||_{C^2}$ (taken with respect to a fixed smooth background metric) so that if $\eps \in (0, \eps_0)$, then 
\[
E_{\eps}(u, B_r(p), \mathcal{V}, g) \leq C r, \qquad \forall r \in (0, r_0), \; \forall p \in U'
\]
\end{lemma}
\begin{proof}
As in \cite[\S 10.2]{GuaracoMinmax}, it suffices to establish a local monotonicity formula. Multiplying equation \eqref{eqn.weak.one} by $(g \cdot \n u)$ (where $g$ is a smooth vector field and integrating over all of $M$ gives 
\begin{equation} \label{eqn.IBP.one}
\int \mathcal{V} \eps |\n u|^2 (\text{div}(g) - \n_{\nu} g(\nu) + (1/2) \langle \n V, g \rangle) = \int \mathcal{V}\left(\frac{\eps}{2} |\n u|^2 - \frac{W(u)}{\eps}\right) \text{div}(g) - \frac{W(u)}{\eps} \langle \n \mathcal{V}, g \rangle
\end{equation}
We now use the same hessian comparison theorem and test function of $g = r \psi(r) \n r$ (where $\psi(s) = \varphi((s - \rho)/\delta)$ and $\varphi$ is a smooth bump function). We note that as $\delta \to 0$, we have
\[
\lim_{\delta \to 0} \Big| \int_{M} (\langle \n \mathcal{V}, g \rangle) \left(\frac{\eps}{2} |\n u|^2 - \frac{W(u)}{\eps} \right)\Big| \leq C \int_M r \mathcal{V} e_{\eps}(u)
\]
Thus, following the computations in \cite[Equation 17]{GuaracoMinmax}, we have that as $\delta \to 0$, %
\begin{align*}
-(n-1) \int_{B_{\rho}} \mathcal{V} e_{\eps}(u) + \rho \int_{\p B_{\rho}} e_{\eps}(u) & \geq \int_{B_{\rho}} \mathcal{V}(- \xi_{\eps}) + \eps \rho \int_{\p B_{\rho}}(\n u \cdot \n r)^2 \\
& - (\sqrt{k} + C)\int_{B_{\rho}} \mathcal{V} r(e_{\eps}(u) + \eps |\n u|^2)
\end{align*}
The proof now concludes the same way after choosing larger constants (i.e. $m \geq 3 (\sqrt{k} + C)$ in Guaraco's notation).
\end{proof}